\documentclass[11pt]{amsart}

\calclayout

\usepackage[utf8]{inputenc}  % to allow unicode in source

\usepackage{amsfonts}
\usepackage{amsmath}
\usepackage{amsthm}
\usepackage{amssymb}
\usepackage{mathtools} % for \mathclap, used to center extra-wide diagrams.
\usepackage{eucal} % for \mathscr

\usepackage{xcolor}
\definecolor{darkgreen}{rgb}{0,0.45,0}
\definecolor{darkred}{rgb}{0.75,0,0}
\definecolor{darkblue}{rgb}{0,0,0.6}
\usepackage[colorlinks,citecolor=darkgreen,linkcolor=darkred,urlcolor=darkblue]{hyperref}
\usepackage{breakurl}
\usepackage{enumerate} % for customising enumerated lists

\usepackage{mathpartir}

\usepackage{tikz}
\usetikzlibrary{arrows}
\usepackage[all]{xy}
\xyoption{2cell}
\xyoption{curve}
\UseTwocells
\input{diagxy}  % for better inline arrows
\usetikzlibrary{matrix,arrows}

\pgfarrowsdeclare{xyto}{xyto}
{
  \arrowsize=0.22pt
  \advance\arrowsize by .7\pgflinewidth  % so at default, \arrowsize=0.5pt
  \pgfarrowsleftextend{-10\arrowsize-.5\pgflinewidth}
  \pgfarrowsrightextend{.5\pgflinewidth}
}
{
  \arrowsize=0.22pt
  \advance\arrowsize by .7\pgflinewidth  % so at default, \arrowsize=0.5pt
  \pgfsetdash{}{0pt} % do not dash
  \pgfsetroundjoin	% fix join
  \pgfsetroundcap	% fix cap 
  \pgfpathmoveto{\pgfpointorigin}
  \pgfpathcurveto
     {\pgfpoint{-4.2\arrowsize}{0.6\arrowsize}}
     {\pgfpoint{-7.2\arrowsize}{1.5\arrowsize}}
     {\pgfpoint{-10\arrowsize}{3.1\arrowsize}}
  \pgfusepathqstroke
  \pgfpathmoveto{\pgfpointorigin}
  \pgfpathcurveto
     {\pgfpoint{-4.2\arrowsize}{-0.6\arrowsize}}
     {\pgfpoint{-7.2\arrowsize}{-1.5\arrowsize}}
     {\pgfpoint{-10\arrowsize}{-3.1\arrowsize}}
  \pgfusepathqstroke
}

\newlength{\ontotipoffset}

\pgfarrowsdeclare{xyonto}{xyonto}
{
  \arrowsize=0.22pt
  \advance\arrowsize by .7\pgflinewidth  % so at default, \arrowsize=0.5pt
  \pgfarrowsleftextend{-10\arrowsize-.5\pgflinewidth}
  \pgfarrowsrightextend{.5\pgflinewidth}
}
{
  \arrowsize=0.22pt
  \advance\arrowsize by .7\pgflinewidth  % so at default, \arrowsize=0.5pt
  \ontotipoffset=-6\arrowsize
  \pgfsetdash{}{0pt} % do not dash
  \pgfsetroundjoin	% fix join
  \pgfsetroundcap	% fix cap 
  \pgfpathmoveto{\pgfpointorigin}
  \pgfpathcurveto
     {\pgfpoint{-4.2\arrowsize}{0.6\arrowsize}}
     {\pgfpoint{-7.2\arrowsize}{1.5\arrowsize}}
     {\pgfpoint{-10\arrowsize}{3.1\arrowsize}}
  \pgfusepathqstroke
  \pgfpathmoveto{\pgfpointorigin}
  \pgfpathcurveto
     {\pgfpoint{-4.2\arrowsize}{-0.6\arrowsize}}
     {\pgfpoint{-7.2\arrowsize}{-1.5\arrowsize}}
     {\pgfpoint{-10\arrowsize}{-3.1\arrowsize}}
  \pgfusepathqstroke
  \pgfpathmoveto{\pgfpoint{\ontotipoffset}{0}}
  \pgfpathcurveto
     {\pgfpoint{\ontotipoffset-4.2\arrowsize}{0.6\arrowsize}}
     {\pgfpoint{\ontotipoffset-7.2\arrowsize}{1.5\arrowsize}}
     {\pgfpoint{\ontotipoffset-10\arrowsize}{3.1\arrowsize}}
  \pgfusepathqstroke
  \pgfpathmoveto{\pgfpoint{\ontotipoffset}{0}}
  \pgfpathcurveto
     {\pgfpoint{\ontotipoffset-4.2\arrowsize}{-0.6\arrowsize}}
     {\pgfpoint{\ontotipoffset-7.2\arrowsize}{-1.5\arrowsize}}
     {\pgfpoint{\ontotipoffset-10\arrowsize}{-3.1\arrowsize}}
  \pgfusepathqstroke
}

\pgfarrowsdeclare{xytri}{xytri}
{
  \arrowsize=0.22pt
  \advance\arrowsize by .7\pgflinewidth  % so at default, \arrowsize=0.5pt
  \pgfarrowsleftextend{-.5\pgflinewidth}
  \pgfarrowsrightextend{10\arrowsize+.5\pgflinewidth}
}
{
  \arrowsize=0.22pt
  \advance\arrowsize by .7\pgflinewidth  % so at default, \arrowsize=0.5pt
  \pgfsetdash{}{0pt} % do not dash
  \pgfsetroundjoin	% fix join
  \pgfsetroundcap	% fix cap 
  \pgfpathmoveto{\pgfpoint{10\arrowsize}{0}}
  \pgfpathcurveto
     {\pgfpoint{5.8\arrowsize}{0.6\arrowsize}}
     {\pgfpoint{2.8\arrowsize}{1.5\arrowsize}}
     {\pgfpoint{0\arrowsize}{3.1\arrowsize}}
  \pgfpathmoveto{\pgfpoint{10\arrowsize}{0}}
  \pgfpathcurveto
     {\pgfpoint{5.8\arrowsize}{-0.6\arrowsize}}
     {\pgfpoint{2.8\arrowsize}{-1.5\arrowsize}}
     {\pgfpoint{0\arrowsize}{-3.1\arrowsize}}
  \pgflineto{\pgfpoint{0}{3.1\arrowsize}}
  \pgfusepathqstroke
} % for arrow tips in TikZ matching xypic’s style
\tikzset{>=xyto}
\pgfarrowsdeclarealias{c}{c}{right hook}{left hook}
\pgfarrowsdeclarealias{c'}{c'}{left hook}{right hook}

\newbox\pbbox
\setbox\pbbox=\hbox{\xy \POS(65,0)\ar@{-} (0,0) \ar@{-} (65,65)\endxy}
\def\pb{\save[]+<3.5mm,-3.5mm>*{\copy\pbbox} \restore}
\theoremstyle{plain}
\newtheorem{theorem}{Theorem}[subsection]
\newtheorem{axiom}[theorem]{Axiom}

\newtheorem{proposition}[theorem]{Proposition}
\newtheorem{lemma}[theorem]{Lemma}
\newtheorem{corollary}[theorem]{Corollary}
\newtheorem{conjecture}[theorem]{Conjecture}

\theoremstyle{definition}
\newtheorem{definition}[theorem]{Definition}

\newtheorem{example}[theorem]{Example}
\newtheorem{examples}[theorem]{Examples}
\newtheorem{notation}[theorem]{Notation}
\newtheorem{remark}[theorem]{Remark}

\newcommand{\extraqedhere}{ \pushQED{\qed} \qedhere  \popQED } % for correctly-placed QEDs in theorem statements

\makeatletter
\renewcommand{\paragraph}{\@startsection{paragraph}{4}{0mm}{-0.5\baselineskip}{-1ex}{\bf}}
\makeatother

\newcommand{\CC}{\mathcal{C}}
\newcommand{\DD}{\mathcal{D}}
\newcommand{\Ebar}{\overline{E}}
\newcommand{\Hbar}{\overline{H}}
\newcommand{\I}{\mathcal{I}}
\newcommand{\N}{\mathbb{N}}
\newcommand{\paths}{\mathrm{P}}
\renewcommand{\P}{\mathbf{P}}
\newcommand{\PP}{\mathrm{P}}

\newcommand{\TT}{\mathsf{\mathbf{T}}}
\newcommand{\U}{\mathbf{U}_{\alpha}}
\newcommand{\UU}[1][\alpha]{{\mathrm{U}_{#1}}}
\newcommand{\UUt}{{\widetilde{\mathrm{U}}_{\alpha}}}
\newcommand{\W}{{\mathbf{W}_{\alpha}}}
\newcommand{\WW}{{\mathrm{W}_{\alpha}}}
\newcommand{\WWt}{{\widetilde{\mathrm{W}}_{\alpha}}}
\newcommand{\y}{\mathbf{y}}

\newcommand{\IdStrux}{\mathrm{Id}}
\newcommand{\PiStrux}{\Pi}
\newcommand{\SigmaStrux}{\Sigma}
\newcommand{\WStrux}{\mathrm{W}}
\newcommand{\zeroStrux}{\mathbf{0}}
\newcommand{\oneStrux}{\mathbf{1}}
\newcommand{\plusStrux}{+}

\newcommand{\colim}{\operatorname{colim}}
\newcommand{\ev}{\operatorname{ev}}
\newcommand{\ft}{\operatorname{ft}}
\newcommand{\gen}{{\mathrm{gen}}}
\newcommand{\ob}{\operatorname{Ob}}
\newcommand{\op}{\mathrm{op}}
\newcommand{\pt}{1}
\newcommand{\Sets}{\mathcal{S}\mathrm{et}}% {\mathbf{Sets}}
\newcommand{\sSets}{\mathrm{s}\mathcal{S}\mathrm{et}}% {\mathbf{sSets}}
\newcommand{\intl}[1]{\mathbf{#1}}
\newcommand{\intEq}{\intl{Eq}}
\newcommand{\intHIso}{\intl{HIso}}
\newcommand{\intHom}{\intl{Hom}}
\newcommand{\intHomLHInv}{\intl{HomLInv}}
\newcommand{\intHomRHInv}{\intl{HomRInv}}
\newcommand{\intEqLHInv}{\intl{EqLInv}}
\newcommand{\intEqRHInv}{\intl{EqRInv}}
\newcommand{\extEq}{\mathrm{Eq}}
\newcommand{\extHIso}{\mathrm{HIso}}
\newcommand{\extHomLHInv}{\mathrm{HomLInv}}
\newcommand{\extHomRHInv}{\mathrm{HomRInv}}
\newcommand{\Hom}{\operatorname{Hom}}

\newcommand{\adjoint}{\dashv}

\newcommand*{\into}{\ensuremath{\lhook\joinrel\relbar\joinrel\rightarrow}}
\newcommand{\iso}{\cong}
\newcommand{\homot}{\simeq}
\newcommand{\shortto}{\rightarrow}

\newcommand{\name}[1]{{\ulcorner {#1} \urcorner}}
\newcommand{\nbhyph}{\mbox{-}}
\newcommand{\lscott}{[\![}
\newcommand{\rscott}{]\!]}
\newcommand{\interp}[1]{{\lscott {#1} \rscott}} 

\newcommand{\toposPi}{\mathit{\Pi}}
\newcommand{\toposSigma}{\mathit{\Sigma}}

\let\syn\mathsf
\newcommand{\oftype}{\mathord{:}}
\newcommand{\types}{\vdash}
\newcommand{\type}{\syn{type}}
\newcommand{\cxt}{\syn{cxt}}
\newcommand{\emptycxt}{[\ ]}
\newcommand{\form}{\textsc{form}}
\newcommand{\intro}{\textsc{intro}}
\newcommand{\appRule}{\textsc{app}}
\newcommand{\elim}{\textsc{elim}}
\newcommand{\comp}{\textsc{comp}}
\newcommand{\Weak}{\textsc{Wkg}}
\newcommand{\Vble}{\textsc{Vble}}

\newcommand{\Subst}{\textsc{Subst}}
\newcommand{\synId}{\syn{Id}}
\newcommand{\synPi}{\syn{\Pi}}
\newcommand{\synSigma}{\syn{\Sigma}}
\newcommand{\synW}{\syn{W}}
\newcommand{\synOne}{\syn{1}}
\newcommand{\synZero}{\syn{0}}

\newcommand{\synU}{\syn{U}}
\newcommand{\el}{\syn{El}}

\newcommand{\app}{\syn{app}}
\newcommand{\pair}{\syn{pair}}
\newcommand{\synsplit}{\syn{split}}
\newcommand{\refl}{\syn{refl}}
\newcommand{\synJ}{\syn{J}}
\newcommand{\synsup}{\syn{sup}}
\newcommand{\wrec}{\syn{wrec}}
\newcommand{\synrec}{\syn{rec}}
\newcommand{\inl}{\syn{inl}}
\newcommand{\inr}{\syn{inr}}
\newcommand{\syncase}{\syn{case}}
\newcommand{\ext}{\syn{ext}}
\newcommand{\extcomp}{\syn{ext\text{-}comp}}

\newcommand{\synpi}{\boldsymbol{\pi}}
\newcommand{\synsigma}{\boldsymbol{\sigma}}
\newcommand{\synid}{\syn{id}}
\newcommand{\synz}{\syn{z}}
\newcommand{\syno}{\syn{o}}
\newcommand{\synw}{\syn{w}}
\newcommand{\smallplus}{\mathbin{\mathchoice
  {\raisebox{0.26ex}{$\scriptscriptstyle{+}$}}
  {\raisebox{0.26ex}{$\scriptscriptstyle{+}$}}
  {\raisebox{0.12ex}{$\scriptscriptstyle{+}$}}
  {+}
}}

\newcommand{\synisHIso}{\syn{isHIso}}
\newcommand{\synHIso}{\syn{HIso}}

\newcommand{\synHomLHInv}{\syn{HomLInv}}
\newcommand{\synHomRHInv}{\syn{HomRInv}}
\newcommand{\synLHInv}{\syn{LInv}}
\newcommand{\synRHInv}{\syn{RInv}}
\newcommand{\synisUnivalent}{\syn{isUnivalent}}

\newcommand{\Utildestrut}{\mathclap{\phantom{\tilde{U}}}}
\newcommand{\USigstrut}{\mathclap{\phantom{U}}}
\begin{document}
\title[The simplicial model of Univalent Foundations (after Voevodsky)]%
  {The simplicial model of Univalent Foundations \\ (after Voevodsky)}

\author[K. Kapulkin]{Krzysztof Kapulkin}
\address[Krzysztof Kapulkin]{University of Pittsburgh}
\email{krk56@pitt.edu}

\author[P. LeF. Lumsdaine]{Peter LeFanu Lumsdaine}
\address[Peter LeFanu Lumsdaine]{Institute for Advanced Study, Princeton}
\email{plumsdaine@ias.edu}


\date{November 2012; last revised September 2018}

\begin{abstract}
We present Voevodsky’s construction of a model of univalent type theory in the category of simplicial sets.

To this end, we first give a general technique for constructing categorical models of dependent type theory, using universes to obtain coherence.  We then construct a (weakly) universal Kan fibration, and use it to exhibit a model in simplicial sets.  Lastly, we introduce the Univalence Axiom, in several equivalent formulations, and show that it holds in our model.

As a corollary, we conclude that Martin-Löf type theory with one univalent universe (formulated in terms of contextual categories) is at least as consistent as ZFC with two inaccessible cardinals.
\end{abstract}

\maketitle

\begin{flushright}
\emph{To the memory of Vladimir Voevodsky.}
\end{flushright}

\tableofcontents

\clearpage

\section*{Introduction}

The Univalent Foundations programme is a new proposed approach to foundations of mathematics, originally suggested by Vladimir Voevodsky in \cite{voevodsky:homotopy-lambda-calculus}, building on the systems of dependent type theory developed by Martin-Löf and others.

A major motivation for earlier work with such logical systems has been their well-suitedness to computer implementation.  One notable example is the Coq proof assistant, based on the Calculus of Inductive Constructions (a closely related dependent type theory), which has shown itself feasible for large-scale formal verification of mathematics, with developments including formal proofs of the Four-Colour Theorem \cite{gonthier:four-color} and the Feit-Thompson (Odd Order) Theorem \cite{gonthier:feit-thompson}.

One feature of dependent type theory which has previously remained comparatively unexploited, however, is its richer treatment of equality.  In traditional foundations, equality carries no information beyond its truth-value: if two things are equal, they are equal in at most one way.  This is fine for equality between elements of discrete sets; but it is unnatural for objects of categories (or higher-dimensional categories), or points of spaces.  In particular, it is at odds with the informal mathematical practice of treating isomorphic (and sometimes more weakly equivalent) objects as equal; which is why this usage must be so often disclaimed as an abuse of language, and kept rigorously away from formal statements, even though it is so appealing.

In dependent type theory, equalities can carry information: two things may be equal in multiple ways.  So the basic objects—the \emph{types}—may behave not just like discrete sets, but more generally like higher groupoids (with equalities being morphisms in the groupoid), or spaces (with equalities being paths in the space).  And, crucially, this is the \emph{only} equality one can talk about within the logical system: one cannot ask whether elements of a type are “equal on the nose”, in the classical sense.\footnote{There is a strict equality, called \emph{judgemental} or \emph{definitional}, but there is no \emph{type/proposition} in the system expressing it, just as with e.g.\ literal syntactic equality traditional foundations.}  The logical language only allows one to talk about properties and constructions which respect its equality.

The \emph{Univalence Axiom}, introduced by Voevodsky, strengthens this characteristic.  In classical foundations one has sets of sets, or classes of sets, and uses these to quantify over classes of structures.  Similarly, in type theory, types of types—\emph{universes}—are a key feature of the language.  The Univalence Axiom states that equality between types, as elements of a universe, is the same as equivalence between them, as types.  It formalises the practice of treating equivalent structures as completely interchangeable; it ensures that one can only talk about properties of types, or more general structures, that respect such equivalence.  In sum, it helps solidify the idea of types as some kind of spaces, in the homotopy-theoretic sense; and more practically—its original motivation—it provides for free many theorems (transfer along equivalences, naturality with respect to these, and so on) which must otherwise be re-proved by hand for each new construction.

The main goal of this paper is to justify the intuition outlined above, of types as spaces.
To this end, we focus on the Quillen model category $\sSets$ of simplicial sets, a well-studied model for topological spaces in homotopy theory; we construct a model of type theory in $\sSets$, and show that it satisfies the Univalence Axiom.
The fibrations of this model structure, called Kan fibrations, will serve as an interpretation of type dependency.
In particular the closed types will be interpreted as Kan complexes, which also serve as a model for $\infty$-groupoids, for instance in Joyal and Lurie's approach to higher category theory.

It follows from this model that Martin-Löf type theory plus the Univalence Axiom (presented in terms of contextual categories) is consistent, provided that the classical foundations we use are---precisely, ZFC together with the existence of two strongly inaccessible cardinals, or equivalently two Grothendieck universes.

As hinted above, there is one important technical caveat regarding our treatment of type theory: we state the model and consistency results in terms of contextual categories, not syntax, so as to avoid reliance on initiality results.

This paper therefore includes a mixture of logical and homotopy-theoretic ingredients; however, we have aimed to separate the two wherever possible.  Good background references for the logical parts include \cite{n-p-s:programming}, a general introduction to the type theory; \cite{hofmann:syntax-and-semantics}, for the categorical semantics; and \cite{martin-lof:bibliopolis}, the \emph{locus classicus} for the logical rules.  For the homotopy-theoretic aspects, \cite{goerss-jardine} and \cite{hovey:book} are both excellent and sufficient references.  Finally, for the category-theoretic language used throughout, \cite{mac-lane:cwm} is canonical.

\paragraph{Organisation} In Section~\ref{section:models-from-universes} we consider general techniques for constructing models of type theory. After setting out (in Section~\ref{subsec:the-type-theory}) the specific type theory that we will consider, we review (Section~\ref{subsec:contextual-cats}) some fundamental facts about its intended semantics in contextual categories, following \cite{streicher:book}.  In Section~\ref{subsec:contextualization}, we use universes to construct contextual categories, representing the structural core of type theory; and in Section~\ref{subsec:logical-structure-on-universes}, we use categorical constructions on the universe to model the logical constructions of type theory.  Together, these present a new solution to the \emph{coherence problem} for modelling type theory (cf.\ \cite{hofmann:on-the-interpretation}).

In Section~\ref{section:the-model}, we turn towards constructing a model in the category of simplicial sets. Sections~\ref{subsec:representability-of-fibs} and \ref{subsec:fibrancy-of-u} are dedicated to the construction and investigation of a (weakly) universal Kan fibration (a “universe of Kan complexes”); in Section~\ref{subsec:model-in-ssets} we use this universe to apply the techniques of Section~\ref{section:models-from-universes}, giving a model of the full type theory in simplicial sets.

Section~\ref{section:univalence} is devoted to the Univalence Axiom.   We formulate univalence first in type theory (Section~\ref{subsec:type-theoretic-univalence}), then directly in homotopy-theoretic terms (Section~\ref{subsec:simplicial-univalence}), and show that these definitions correspond under the simplicial model (Section~\ref{subsec:univalence-equivalence}).  In Section~\ref{subsec:univalence-of-uu}, we show that the universal Kan fibration is univalent, and hence that the Univalence Axiom holds in the simplicial model.  Finally, in Section~\ref{subsec:pullback-reps} we discuss an alternative formulation of univalence, shedding further light on the universal property of the universe.

We include also two appendices, setting out in full the type theory under consideration: first a conventional syntactic presentation in Appendix~\ref{app:type-theory}, and then in Appendix~\ref{app:cxl-structure} its translation into algebraic structure on contextual categories.

\paragraph*{History of the paper}
This paper started life as notes by the current authors based on Vladimir Voevodsky's lectures at the 2011 Oberwolfach workshop \cite{oberwolfach-report} along with his associated manuscript \cite{voevodsky:notes-on-type-systems}.
It was subsequently expanded with Voevodsky’s collaboration into the present full exposition of the simplicial model, and appeared as a preprint in 2012 with Voevodsky included as co-author.

In 2016, due to his dissatisfaction with the existing literature on type theory, which this paper took as background, Voevodsky asked us to remove him as co-author and delay publication until he had finished developing his own treatment of semantics of type theory (cf.~\cite{voevodsky:c-system-of-module} and sequels) and presentation of the simplicial model in that framework.

Tragically, Voevodsky passed away in September 2017, before completing that project.  This paper therefore remains the only account of Voevodsky's construction of the simplicial model, so with the support of Daniel Grayson, Voevodsky’s academic executor, we have prepared it again for publication.  We have made several changes to accommodate Voevodsky's reservations regarding the treatment of semantics; most importantly, we present the initiality of syntax as a conjecture rather than a theorem (Conjecture \ref{conj:initiality}), and give all main results in terms of contextual categories rather than syntax.  Otherwise, the paper remains substantially unchanged from the original 2012 version.

The main results of the paper are due to Voevodsky, including Theorems~\ref{thm:structure-on-U-to-CU}, \ref{thm:the-model-in-ssets}, \ref{thm:simplicial-model-univalent} and \ref{thm:univalent-characterization}. Mathematical contributions of Kapulkin and Lumsdaine include all of Section~\ref{subsec:univalence-equivalence}, along with streamlining various parts of the main constructions and completing portions omitted in \cite{voevodsky:notes-on-type-systems}.

Out of respect for Voevodsky’s stated wishes, and following discussion with his executor, he remains absent as an author of the final version of this paper. 
However, we wish to leave no doubt regarding the share of the credit that is his due.
We mourn the loss of an exceptional mathematician and mentor, and dedicate this paper to his memory.

\paragraph*{Related work}
While the present paper discusses just models of type theory with the univalence axiom, the major motivation for this is the actual development of mathematics within these foundations. Introductions to various aspects of this are given in \cite{grayson:intro-to-univalent-foundations}, \cite{voevodsky:experimental-library}, \cite{pelayo-warren:univalent-foundations-paper}, and \cite{hott:book},
while large computer-formalised developments include the UniMath\footnote{\url{https://www.github.com/UniMath/UniMath}} and HoTT\footnote{\url{https://www.github.com/HoTT/HoTT}} libraries, presented in \cite{voevodsky-ahrens-grayson:unimath} and \cite{bauer-et-al:hott-library}.

Earlier work on homotopy-theoretic models of type theory can be found in \cite{hofmann-streicher}, \cite{awodey-warren}, \cite{warren:strict-w-groupoid-model}.  Other current and recent work on such models includes \cite{garner-van-den-berg}, \cite{arndt-kapulkin}, and \cite{shulman:inverse-diagrams}.  Other general coherence theorems, for comparison with the results of Section~\ref{section:models-from-universes}, can be found in \cite{hofmann:on-the-interpretation} and \cite{lumsdaine-warren:local-universes}.  Univalence in homotopy-theoretic settings is also considered in %\cite{moerdijk:univalence} and 
\cite{gepner-kock:univalence}.  (These references are, of course, far from exhaustive.)

\paragraph*{Acknowledgements} 
First and foremost we would like to thank Vladimir Voevodsky: the creator of Univalent Foundations, a mentor to both the authors, and whose insight and ingenuity underlie not only the present paper but much subsequent work in the field.
We are particularly indebted also to Michael Warren, whose illuminating seminars and discussions heavily influenced our understanding and presentation of the simplicial model.
We also thank Daniel Grayson, Ieke Moerdijk, Mike Shulman, Raffael Stenzel, and Karol Szumi{\l}o, for helpful correspondence, conversations, and corrections to drafts and earlier versions, and Steve Awodey, for support and encouragement.

The first-named author was financially supported during this work by the NSF, Grant DMS-1001191 (P.I.~Steve Awodey), and by a grant from the Benter Foundation (P.I.~Thomas C.~Hales); the second-named author, by an AARMS postdoctoral fellowship at Dalhousie University, and grants from NSERC (P.I.’s~Peter Selinger, Robert Dawson, and Dorette Pronk).

%%% Local Variables:
%%% mode: latex
%%% TeX-master: "simplicial-model"
%%% End: 

\section{Models from Universes} \label{section:models-from-universes}

In this section, we set up the machinery which we will use, in later sections, to model type theory in simplicial sets.  The type theory we consider, and some of the technical machinery we use, are standard; the main original contribution is a new technique for solving the so-called coherence problem, using universes.

\subsection{The type theory under consideration} \label{subsec:the-type-theory}

Formally, the type theory we will consider is a slight variant of Martin-Löf’s Intensional Type Theory, as presented in e.g.\ \cite{martin-lof:bibliopolis}.  The rules of this theory are given in full in Appendix~\ref{app:type-theory}; briefly, it is a dependent type theory, taking as basic constructors $\synPi$-, $\synSigma$-, $\synId$-, and $\synW$-types, $\synZero$, $\synOne$, $+$, and one universe à la Tarski closed under these constructors.

A related theory of particular interest is the Calculus of Inductive Constructions, on which the Coq proof assistant is based (\cite{werner:thesis}).  CIC differs from Martin-Löf type theory most notably in its very general scheme for inductive definitions, and in its treatment of universes.  We do not pursue the question of how our model might be adapted to CIC, but for some discussion and comparison of the two systems, see  \cite{paulin-mohring:habilitation}, \cite{barras:habilitation}, and \cite[6.2]{voevodsky:notes-on-type-systems}.

One abuse of notation that we should mention: we will sometimes write e.g.\ $A(x)$ or $t(x,y)$ to indicate free variables on which a term or type may depend, so that we can later write $A(g(z))$ to denote the substitution $[g(z)/x]A$ more readably.  Note however that the variables explicitly shown need not actually appear; and there may also always be other free variables in the term, not explicitly displayed.

\subsection{Contextual categories} \label{subsec:contextual-cats}

Rather than working formally with the syntax of this type theory, we work instead in terms of \emph{contextual categories}, a class of algebraic objects abstracting the key structure given by the syntax.\footnote{Contextual categories are not the only option; the closely related notions of \emph{categories with attributes} \cite{cartmell:thesis,moggi:program-modules,pitts:categorial-logic}, \emph{categories with families} \cite{dybjer:internal-type-theory,hofmann:syntax-and-semantics}, and \emph{comprehension categories} \cite{jacobs:comprehension-categories} would all also serve our purposes.}  The plain definition of a contextual category corresponds to the structural core of the syntax; further syntactic rules (logical constructors, etc.)\ correspond to extra algebraic structure that contextual categories may carry.  Essentially, contextual categories are intended to provide a completely equivalent alternative to the syntactic presentation of type theory.

Why do we make this bait-and-switch?  The trouble with the syntax is that it is very tricky to handle rigorously. Any full presentation must account for (among other complications) variable binding, capture-free substitution, and the possibility of multiple derivations of a judgement; and so any careful construction of an interpretation must deal with all of these, at the same time as tackling the details of the particular model in question.  Contextual categories, by contrast, are a purely algebraic notion, with no such subtleties.  The idea is therefore that given any contextual category $\CC$ with structure corresponding to the logical rules of some syntactic type theory $\TT$, one should obtain an interpretation of the syntax of $\TT$ in $\CC$; and in proving this, one deals with the subtleties and bureaucracy of $\TT$ once and for all, giving a clear framework for subsequently constructing models of $\TT$.

Such an “initiality theorem” (cf.\ Conjecture~\ref{conj:initiality} below) has been proven for some specific rather small type theories, e.g.\ by Streicher in the Correctness Theorem of \cite[Ch.~III, p.~181]{streicher:book}.  For larger type theories such as the present one, however, its status is debatable, and at best unsatisfactory.  The traditional view is that for suitable type theories, a straightforward extension of Streicher’s and other standard methods suffices, and therefore the theorem can be regarded as established.  However, Voevodsky has argued persuasively that this is an unacceptably unrigorous attitude.  No precise definition has been given of what “suitable type theories” the methods apply to; nor (to our knowledge) has the proof been even sketched in detail for any type theory beyond small toy examples; and while a “straightforward” extension of standard methods may indeed suffice for a type theory such as the present one, that sufficiency is far from obvious.

For the present paper, therefore, we work formally entirely in terms of contextual categories, and avoid relying on initiality theorems in any form.

Conversely, then, why bring up syntax at all, other than as motivation?  The trouble on this side is that working with higher-order logical structure in contextual categories quickly becomes unreadable: compare, for instance, the statements of functional extensionality in Sections~\ref{subsec:optional-rules} and \ref{subsec:optional-rules-alg}.

We therefore make free use of the syntax of type theory, as a \emph{notation} for working in contextual categories.  The situation is rather comparable to that of string diagrams, as used in monoidal and more elaborately structured categories \cite{selinger:graphical-languages}, or indeed of the traditional notations for differentiation and integration.  In each case, one has a powerful, flexible, and intuitive notation, whose rigorous definition and validity requires quite non-trivial work to establish; but in lieu of such a general justification, one may still fruitfully make use of the notation, trusting the reader to translate it into the unproblematic algebraic form as required.

\begin{definition}[Cartmell {\cite[Sec.~2.2]{cartmell:thesis}}, Streicher {\cite[Def.~1.2]{streicher:book}}] \label{def:cxl-cat}
A \emph{contextual category} $\CC$ consists of the following data:
\begin{enumerate}
\item a category $\CC$;
\item a grading of objects as $\ob \CC = \coprod_{n:\N} \ob_n \CC$;
\item an object $\pt \in \ob_0 \CC$;
\item maps $\ft_n : \ob_{n+1} \CC \to \ob_n \CC$ (whose subscripts we usually suppress);
\item for each $X \in \ob_{n+1} \CC$, a map $p_X \colon X \to \ft X$ (the \emph{canonical projection} from $X$);
\item for each $X \in \ob_{n+1} \CC$ and $f \colon Y \to \ft(X)$, an object $f^*(X)$ together with a map $q(f,X) \colon f^*(X) \to X$;
\newcounter{tempcounter}
\setcounter{tempcounter}{\theenumi}
\end{enumerate}  
such that:
\begin{enumerate}
\setcounter{enumi}{\thetempcounter}
\item $\pt$ is the unique object in $\ob_0(\CC)$;
\item $\pt$ is a terminal object in $\CC$;
\item for each $n > 0$, $X \in \ob_n \CC$, and $f \colon Y \to \ft(X)$, we have $\ft(f^*X) = Y$, and the square
\[\xymatrix{
  f^* X \ar[d]_{p_{f^*X}} \ar[r]^{q(f,X)} \pb & X \ar[d]^{p_x} \\
  Y                   \ar[r]^f      & \ft (X)
}\]
is a pullback (the \emph{canonical pullback} of $X$ along $f$); and
\item these canonical pullbacks are strictly functorial: that is, for $X \in \ob_{n+1} \CC$, $1_{\ft X}^* X = X$ and $q(1_{\ft X},X) = 1_X$; and for $X \in \ob_{n+1} \CC$, $f \colon Y \to \ft X$ and $g \colon Z \to Y$, we have $(fg)^*(X) = g^*(f^*(X))$ and $q(fg,X) = q(f,X)q(g,f^*X)$.
\end{enumerate}

  Contextual cateories have also been studied under the name \emph{$C$-systems} (\cite{voevodsky:c-system-of-module} et seq.)
\end{definition}

\begin{remark}
  Note that these may be seen as models of a multi-sorted essentially algebraic theory (\cite[3.34]{adamek-rosicky}), with sorts indexed by $\N + \N \times \N$.
\end{remark}

This definition is best understood in terms of its prototypical example:

\begin{example}[Cartmell {\cite[p2.6]{cartmell:thesis}}; cf.\ also \cite{voevodsky:c-system-of-module}, \cite{voevodsky:subsystems-and-quotients}]
Let $\TT$ be the dependent type theory given by the structural rules of Section~\ref{subsec:structural-rules}, plus any selection of the subsequent logical rules.\footnote{Heuristically, $\TT$ may be “any type theory” here; but there is no established definition of what this means!}  Then there is a contextual category $\CC(\TT)$, described as follows:
\begin{itemize}
\item $\ob_n \CC(\TT)$ consists of the contexts $[ x_1 \oftype A_1,\ \ldots,\ x_n \oftype A_n ]$ of length $n$, up to definitional equality and renaming of free variables;
\item maps of $\CC(\TT)$ are \emph{context morphisms}, or \emph{substitutions}, considered up to definitional equality and renaming of free variables.  That is, a map 
\[f \colon [ x_1 \oftype A_1,\ \ldots,\ x_n \oftype A_n] \to [ y_1 \oftype B_1,\ \ldots,\ y_m \oftype B_m(y_1, \ldots, y_{m-1}) ] \]
is an equivalence class of sequences of terms $f_1, \ldots, f_m$ such that
\begin{equation*}
\begin{split}
  x_1 \oftype A_1,\ \ldots,\ x_n \oftype A_n & \types f_1 : B_1 \\
  & \vdots  \\
  x_1 \oftype A_1,\ \ldots,\ x_n \oftype A_n & \types f_m : B_m(f_1,\ \ldots,\ f_{m-1}),
\end{split}
\end{equation*}
 and two such maps $[f_i]$, $[g_i]$ are equal exactly if for each $i$,
\[  x_1 \oftype A_1,\ \ldots,\ x_n \oftype A_n \types f_i = g_i \colon B_i(f_1,\ \ldots\ f_{i-1}); \]
\item composition is given by substitution, and the identity $\Gamma \to \Gamma$ by the variables of $\Gamma$, considered as terms;
\item $\pt$ is the empty context $\emptycxt$;
\item $\ft [ x_1 \oftype A_1,\ \ldots,\ x_{n+1} \oftype A_{n+1}] = [ x_1 \oftype A_1,\ \ldots,\ x_n \oftype A_n]$;
\item for $\Gamma = [ x_1 \oftype A_1,\ \ldots,\ x_{n+1} \oftype A_{n+1}]$, the map $p_\Gamma \colon \Gamma \to \ft \Gamma$ is the \emph{dependent projection} context morphism 
\[ (x_1,\, \ldots,\, x_n) \colon [ x_1 \oftype A_1,\ \ldots,\ x_{n+1} \oftype A_{n+1}] \to [ x_1 \oftype A_1,\ \ldots,\ x_n \oftype A_n ], \]
simply forgetting the last variable of $\Gamma$;
\item for contexts
\[\Gamma = [ x_1 \oftype A_1,\ \ldots,\ x_{n+1} \oftype A_{n+1}(x_1,\ldots,x_n)],\] 
\[\Gamma' = [ y_1 \oftype B_1,\ \ldots,\ y_{m} \oftype B_{m}(y_1,\ldots,y_{m-1})],\]
and a map $f = [f_i(\vec y)]_{i \leq n} \colon \Gamma' \to \ft \Gamma$, the pullback $f^* \Gamma$ is the context
\[  [ y_1 \oftype B_1,\ \ldots,\ y_{m} \oftype B_{m}(y_1,\ldots,y_{m-1}),\ y_{m+1} \oftype A_{n+1}(f_1(\vec y),\ldots,f_n(\vec y))], \]
(for some fresh $y_{m+1}$) and $q(\Gamma,f) \colon f^*\Gamma \to \Gamma$ is the map
\[ [ f_1,\ \ldots, f_n,\ y_{m+1} ]. \]
\end{itemize}
\end{example}

Note that typed terms $\Gamma \types t : A$ of $\TT$ may be recovered from $\CC(\TT)$, up to definitional equality, as sections of the projection $p_{[\Gamma,\;x \oftype A]} \colon [\Gamma,\ x \oftype A] \to \Gamma$.  For this reason, when working with contextual categories, we will often write just “sections” to refer to sections of dependent projections.

We will also use several other notations deserving of particular comment.  For a fixed contextual category $\CC$ and an object $\Gamma \in \ob_n \CC$, we write $(\Gamma,A)$ for any object in $\ob_{n+1} \CC$ with $\ft(\Gamma,A) = \Gamma$, shall such object exist, and $p_A$ for the dependent projection $p_{(\Gamma,A)}$. Similarly, we write $(\Gamma,A,B)$ for an arbitrary object in $\ob_{n+2} \CC$ with $\ft(\Gamma, A, B) = (\Gamma, A)$, and so on.  

Given a morphism $f \colon \Delta \to \Gamma$ and an object $(\Gamma, A)$, we write $(\Delta, f^*A)$ for the canonical pullback $f^*(\Gamma, A)$ and similarly $(\Delta, f^*A, f^*B)$ for $f^*(\Gamma, A, B)$. We also extend the notation $f^*$ to apply not only to the canonical pullbacks of appropriate objects, but also the pullbacks of maps between them. \\

As mentioned above, Definition~\ref{def:cxl-cat} alone corresponds precisely to the basic judgements and structural rules of dependent type theory.  Similarly, each logical rule or type- or term-constructor should correspond to certain extra structure on a contextual category.   We state this intended correspondence precisely in Conjecture~\ref{conj:initiality} below, once we have set up the appropriate definitions.

\begin{definition}[cf.~{\cite[Sec.~4]{voevodsky:products-on-c-systems}}] \label{def:pi-type-structure}
 A \emph{$\synPi$-type structure} on a contextual category $\CC$ consists of:
 \begin{enumerate}
  \item for each $(\Gamma, A, B) \in \ob_{n+2} \CC$, an object $(\Gamma, \synPi(A, B)) \in \ob_{n+1} \CC$;
  \item for each such $(\Gamma, A, B)$ and section $b \colon (\Gamma, A) \to (\Gamma, A, B)$ (of the dependent projection $p_B)$, a section $\lambda(b) \colon \Gamma \to (\Gamma, \synPi(A, B))$ (of $p_{\synPi(A,B)}$);
  \item for each $(\Gamma, A, B)$ and pair of sections $k \colon \Gamma \to (\Gamma, \synPi (A, B))$ and $a \colon \Gamma \to (\Gamma, A)$, a section $\app(k,a) \colon \Gamma \to (\Gamma, A, B)$ such that the following diagram commutes:
      \[\xymatrix{ & (\Gamma, A, B) \ar[d]^{p_B} \\
       & (\Gamma, A) \ar[d]^{p_A} \\
       \Gamma \ar@/^/[ruu]^{\app(k,a)} \ar[ru]^{a} \ar@{=}[r] & \Gamma ;
      }\]
 \item such that for all such $(\Gamma, A, B)$, $a \colon \Gamma \to (\Gamma, A)$, and $b \colon (\Gamma, A) \to (\Gamma, A, B)$, we have $\app(\lambda(b),a) = b \cdot a$;
 \item and moreover such that all the above operations are stable under substitution: for any morphism $f \colon \Delta \to \Gamma$, and suitable $(\Gamma,A,B)$, $a$, $b$, $k$, we have
   \begin{gather*}
     (\Delta, f^*\synPi(A, B)) = (\Delta, \synPi(f^*A, f^*B)), \\
     \lambda({f^*b}) = f^*\lambda(b), \qquad \app(f^*k,f^*a) = f^*(\app(k,a)).
   \end{gather*}
 \end{enumerate}
\end{definition}

These are direct translations of the rules for $\synPi$-types given in Section~\ref{subsec:logical-rules}.%
\footnote{A partial exception is the stability axiom, which corresponds not to any explicitly given rule of the syntax, but rather to clauses for $\synPi$, $\lambda$, and $\app$ in the inductive \emph{definition} of substitution.}
Similarly, all the other logical rules of Appendix~\ref{app:type-theory} may be routinely translated into structure on a contextual category; see Appendix~\ref{app:cxl-structure} and \cite[3.3]{hofmann:syntax-and-semantics} for more details and discussion.

\begin{example} If $\TT$ is a type theory with $\synPi$-types, then $\CC(\TT)$ carries an evident $\synPi$-type structure; similarly for $\synSigma$-types and the other constructors of Sections~\ref{subsec:logical-rules} and \ref{subsec:universe-rules}. \qedhere
\end{example}

\begin{remark}
Note that all of these structures, like the definition of contextual categories themselves, are essentially algebraic in nature.
\end{remark}

\begin{definition}
A map $F \colon \CC \to \DD$  of contextual categories, or \emph{contextual functor}, consists of a functor $\CC \to \DD$ between underlying categories, respecting the gradings, and preserving (on the nose) all the structure of a contextual category.

Similarly, a map of contextual categories with $\synPi$-type structure, $\synSigma$-type structure, etc., is a contextual functor preserving the additional structure.
\end{definition}

\begin{remark}
These are exactly the maps given by considering contextual categories as essentially algebraic structures.
\end{remark}

We are now equipped to state precisely the sense in which the structures defined above are expected to correspond to the appropriate syntactic rules:

\begin{conjecture} \label{conj:initiality} % was: \label{thm:free-cxl-cat}
Let $\TT$ be the type theory given by the structural rules of Section~\ref{subsec:structural-rules}, plus any combination of the logical rules of Sections~\ref{subsec:logical-rules}, \ref{subsec:universe-rules}.  Then $\CC(\TT)$ is initial among contextual categories with the correspondingly-named extra structure.
\end{conjecture}

In other words, if $\CC$ is a contextual category with structure corresponding to the logical rules of a type theory $\TT$, then there should be a unique homomorphism $\CC(\TT) \to \CC$, interpreting the syntax of $\TT$ in $\CC$.  As discussed at the beginning of this section, the Correctness Theorem of \cite[Ch.~III, p.~181]{streicher:book} gives an analogous fact for a rather smaller type theory, while the status of the present conjecture is debated, accepted by some but not all in the field as a straightforward extension of that theorem.

Bearing this intended correspondence in mind, therefore, but avoiding relying on it, we take for the present paper the following definitions:

\begin{definition} \label{def:uf-and-models}
  By \emph{Martin-L\"of Type Theory plus the Univalence Axiom} ($\mathsf{MLTT}+\mathsf{UA}$ for short), we mean dependent type theory with $\synPi$-, $\synSigma$-, $\synId$-, $\synW$-, unit, zero, and sum types, along with one universe closed under all these type formers and satisfying the Univalence Axiom, as set out in Appendix~\ref{app:type-theory}.

  By a \emph{model} of $\mathsf{MLTT}+\mathsf{UA}$, or more generally of dependent type theory with any selection of the logical rules of Appendix~\ref{app:type-theory}, we mean a contextual category equipped with the corresponding structure from Appendix~\ref{app:cxl-structure}.  By the \emph{contextual-category presentation} of such a type theory, we mean the essentially algebraic theory of such structures.
\end{definition}

Note that by definition as an essentially algebraic theory, it is immediate that any such type theory has an initial model.

\begin{definition}
  A dependent type theory of the form considered in Definition~\ref{def:uf-and-models} and including the empty type $\synZero$ is \emph{inconsistent} just if in the initial model, the map $p_{\synZero_\pt} : (\pt,\synZero_\pt) \to \pt$ admits a section, and is \emph{consistent} if it is not inconsistent.
\end{definition}

Assuming initiality, this corresponds to the usual type-theoretic sense of inconsistency: a closed term inhabiting the empty type.
Readers who accept the initiality conjecture as true may therefore read Theorem~\ref{thm:simplicial-model-univalent} as providing an interpretation of the usual syntactic presentation of $\mathsf{MLTT}+\mathsf{UA}$, and Theorem~\ref{thm:uf-consistent} as its consistency in the usual type-theoretic sense.

\subsection{Contextual categories from universes} \label{subsec:contextualization}

The major difficulty in constructing models of type theories is the so-called \emph{coherence problem}: the requirement for pullback to be strictly functorial, and for the logical structure to commute strictly with it.  In most natural categorical situations, operations on objects commute with pullback only up to isomorphism, or even more weakly; and for constructors with weak universal properties, operations on maps (corresponding for example to the $\synId$-\elim\ rule) may also fail to commute with pullback.  Hofmann \cite{hofmann:on-the-interpretation} gives a construction which solves the issue for $\synPi$- and $\synSigma$-types, but $\synId$-types in particular remain problematic with this method.  Other methods exist for certain specific categories (\cite{hofmann-streicher}, \cite{warren:thesis}), but are not applicable to the present case.

In order to obtain coherence for our model, we thus use a construction based on \emph{universes} (not necessarily the same as universes in the type-theoretic sense, though the two may sometimes coincide), studied in more detail in \cite{voevodsky:c-systems-from-universes}. 

\begin{definition}[{\cite[Def.~2.1]{voevodsky:c-systems-from-universes}}]
Let $\CC$ be a category.  A \emph{universe} in $\CC$ is an object $U$ together with a morphism $p \colon \tilde{U} \to U$, and for each map $f \colon X \to U$ a choice of pullback square
\[ \xymatrix{(X;f) \ar[r]^{Q(f)}  \ar[d]_{P_{(X,f)}} \pb & \tilde{U} \ar[d]^p \\ X \ar[r]^f & U. } \]
\end{definition} 

The intuition here is that the map $p$ represents the generic family of types over the universe $U$.  
 
By abuse of notation, we often refer to the universe simply as $U$, with $p$ and the chosen pullbacks understood.  

Given a map $q \colon Y \to X$, we will often write $\name{q}$ (or $\name{Y}$, if $q$ is understood) for a map $X \to U$ such that $q \iso P_{(X,\name{q})}$ in $\CC/X$.  Also, for a sequence of maps $f_1 \colon X \to U$, $f_2 \colon (X;f_1) \to U$, etc., we write $(X;f_1,\ldots, f_n)$ for $((\ldots(X;f_1);\ldots); f_n)$.  (In particular, with $n=0$, $(X;\ ) = X$.)

\begin{definition}[{\cite[Constr.~2.12]{voevodsky:c-systems-from-universes}}] \label{def:contextualisation}
 Given a category $\CC$, together with a universe $U$ and a terminal object $\pt$, we define a contextual category $\CC_U$ as follows:
\begin{itemize}
\item $\ob_n \CC_U :=$ $\{\ (f_1, \ldots, f_n)  \in (\mathrm{Mor} \CC )^n \ |\ f_i \colon (\pt;f_1,\ldots,f_{i-1}) \to U\ (1 \leq i \leq n)\ \};$

\item $\CC_U((f_1,\ldots,f_n),(g_1,\ldots,g_m)) :=$ $\CC((\pt;f_1,\ldots,f_n),(\pt;g_1,\ldots,g_n));$

\item $\pt_{\CC_U} := ( )$, the empty sequence;

\item $\ft (f_1,\ldots,f_{n+1}) := (f_1,\ldots,f_{n})$;

\item the projection $p_{(f_1,\ldots,f_{n+1})}$ is the map $P_{(X,f_{n+1})}$ provided by the universe structure on $U$;

\item given $(f_1,\ldots,f_{n+1})$ and a map $\alpha \colon (g_1, \ldots, g_m) \to (f_1, \ldots, f_{n})$ in $\CC_U$, the canonical pullback $\alpha^*(f_1,\ldots,f_{n+1})$ in $\CC_U$ is given by $(g_1, \ldots, g_{m},$ $f_{n+1} \cdot \alpha)$, with projection induced by $Q(f_{n+1}\cdot\alpha)$:
\[ \xymatrix@C=1.5cm{
(1; g_1, \ldots, g_m, f_{n+1} \cdot \alpha) \ar@/^1.2em/[rr]^-{Q(f_{n+1}\cdot\alpha)} \ar[r] \ar[d] \pb & (1; f_1, \ldots, f_{n+1}) \ar[d] \ar[r]_-{Q(f_{n+1})} \pb & \tilde{U} \ar[d]^p \\
(1; g_1, \ldots, g_m) \ar[r]^\alpha & (1; f_1, \ldots, f_n) \ar[r]^-{f_{n+1}} & U
} \]
\end{itemize}
\end{definition}

\begin{proposition}[{\cite[Constr.~2.12, Ex.~4.9]{voevodsky:c-systems-from-universes}}]  \leavevmode
\begin{enumerate}
\item These data define a contextual category $\CC_U$.
\item This contextual category is well-defined up to canonical isomorphism given just $\CC$ and $p \colon \tilde{U} \to U$, independently of the choice of pullbacks and terminal object.
\end{enumerate}
\end{proposition}

\begin{proof}
Routine computation.
\end{proof}

Justified by the second part of this proposition, we will not explicitly consider the choices of pullbacks and terminal object when we construct the universe in the category $\sSets$ of simplicial sets.

As an aside, let us note that every small contextual category arises in this way:

\begin{proposition}[{\cite[Constr.~5.2]{voevodsky:c-systems-from-universes}}]
Let $\CC$ be a small contextual category.  Consider the universe $U$ in the presheaf category $[\CC^\op,\Sets]$ given by
\begin{align*} U(X) &= \{ Y\ |\ \ft Y = X \} \\
\tilde{U}(X) & = \{ (Y,s)\ |\ \ft Y = X,\ s\ \textnormal{a section of}\ p_Y\}, \end{align*}
with the evident projection map, and any choice of pullbacks.

Then $[\CC^\op,\Sets]_U$ is isomorphic, as a contextual category, to $\CC$.
\end{proposition}

\begin{proof}
Straightforward, with liberal use of the Yoneda lemma.
\end{proof}

\subsection{Logical structure on universes} \label{subsec:logical-structure-on-universes}

Given a universe $U$ in a category $\CC$, we want to know how to equip $\CC_U$ with various logical structure---$\synPi$-types, $\synSigma$-types, and so on.  For general $\CC$, this is rather fiddly; but when $\CC$ is locally cartesian closed (as in our case of interest), it is more straightforward, since local cartesian closedness allows us to construct and manipulate “objects of $U$-contexts”, and hence to construct objects representing the premises of each rule.

In working with locally cartesian closed categories (LCCC’s), we will follow topos-theoretic convention and write $\toposSigma_f$ and $\toposPi_f$ respectively for the left and right adjoints to the pullback functor $f^*$ along a map $f \colon A \to B$: 
\[ \xymatrix{
\CC / A \ar@/^1.5em/[rr]^{\toposSigma_f} \ar@/_1.5em/[rr]_{\toposPi_f} \ar@{{}{ }{}}@/^0.9em/[rr]|{\bot}  \ar@{{}{ }{}}@/_0.9em/[rr]|{\bot} & & \CC/B \ar[ll]|{f^*}
} \]
Also, the intended map $A \to B$ is often clearly determined by the objects $A$ and $B$, as some sort of associated projection; in such a case, we will write $\toposSigma_{A\shortto B}$, $\toposPi_{A\shortto B}$ for the functors arising from this map.

An alternative notation for locally cartesian closed categories is their internal logic, \emph{extensional} dependent type theory \cite{seely:lcccs}, \cite{hofmann:on-the-interpretation}.  While this language is convenient and powerful, we avoid it due to the difficulties of working clearly with two logical languages in parallel.  \\

Returning to the question at hand, first consider $\synPi$-types.\footnote{The following construction is studied in considerable detail in \cite{voevodsky:products-from-universes}.}  We know that dependent products exist in $\CC$; so informally, we need only to ensure that $U$ (considered as a universe of types) is closed under such products.  Specifically, given a type $A$ in $U$ over some base $X$ (that is, a map $\name{A} \colon X \to U$), and a dependent family of types $B$ over $A$, again in $U$ (i.e.\ a map $\name{B} \colon A := (X;\name{A}) \to U$), the product $\toposPi_{A \shortto X} B$ of this family in the slice $\CC/X$ should again “live in $U$”; that is, there should be a map $\name{\Pi(A,B)} \colon X \to U$ such that $(X;\name{\Pi(A,B)}) \iso \toposPi_{A \shortto X} B$.  Moreover, we need this construction to be strictly natural in $X$.

Due to the strict naturality requirement, we cannot simply provide this structure for each $X$ and $A, B$ individually.  Instead, we construct an object $U^{\synPi}$ representing such pairs $(A,B)$, and a generic such pair $(A_\gen, B_\gen)$ based on $U^{\synPi}$.  It is sufficient to define $\PiStrux$ in this generic case $X = U^{\synPi}$; the construction then extends to other $X$ by precomposition, and as such, is automatically strictly natural in $X$.  

Precisely:

\begin{definition}
Given a universe $U$ in an lccc $\CC$, define
\[ U^{\synPi} := \toposSigma_{\Utildestrut U\shortto 1} \toposPi_{\tilde{U} \shortto U} (\pi_2 \colon U \times \tilde{U} \shortto \tilde{U}). \]

(This definition can be expressed in several ways, according to one’s preferred notation. In the internal language of $\CC$ as an LCCC, it can be written as $\interp{ A \oftype U,\, B \oftype [\tilde{U}_A,U] }$, showing it more explicitly as an internalisation of the premises of the $\synPi$-\form\ rule.  Using a more traditional internal-hom notation, it could alternatively be written as $\intHom_U(\tilde{U}, U \times U)$.)

Pulling back $\tilde{U}$ along the projection $U^{\synPi} \to U$ induces an object $A_\gen = \tilde{U} \times_U U^{\synPi}$, along with a projection map $\alpha_\gen \colon A_\gen \to U^{\synPi}$.  Similarly, pulling back $\tilde{U}$ along the counit (the evaluation map of the internal hom)
\[A_\gen = \tilde{U} \times_U \toposPi_{\tilde{U} \shortto U} (U \times \tilde{U}) \to U \times \tilde{U} \to U\]
induces an object $(B_\gen,\beta_\gen)$ over $A_\gen$:
\[ \begin{tikzpicture}[hole/.style={fill=white,inner sep=1pt},x=1.4cm,y=1.4cm]
\node (UPi) at (0,0) {$U^{\synPi}$};
\node (U) at (1,0) {$U$};
\node (Ut) at (1,1) {$\tilde{U}$};
\node (Agen) at (0,1) {$A_\gen$};
\node (U2) at (1.6,1) {$U$};
\node (Ut2) at (1.6,2) {$\tilde{U}$};
\node (Bgen) at (0,2) {$B_\gen$};
\draw[->,font=\scriptsize] (UPi) to (U);
\draw[->,font=\scriptsize] (Agen) to (UPi);
\draw[->,font=\scriptsize] (Ut) to (U);
\draw[->,font=\scriptsize] (Agen) to (Ut);
\draw[->,font=\scriptsize] (Agen) to [out=30,in=135] (U2);
\draw[->,font=\scriptsize] (Bgen) to (Agen);
\draw[->,font=\scriptsize] (Ut2) to (U2);
\draw[->,font=\scriptsize] (Bgen) to (Ut2);
\draw (0.15,0.65) -- (0.35,0.65) -- (0.35,0.85);
\draw (0.15,1.65) -- (0.35,1.65) -- (0.35,1.85);
\end{tikzpicture} \]

Moreover, the universal properties of the LCCC structure ensure that for any sequence $B \to A \to \Gamma$ with maps $\Gamma \to U$, $A \to U$, $A \to \tilde{U}$, $B \to \tilde{U}$ exhibiting $A \to \Gamma$ and $B \to A$ as pullbacks of $\tilde{U} \to U$, there is a unique map $\name{(A,B)} \colon \Gamma \to U^{\synPi}$ which induces the given sequence via precomposition and pullback:
\[ \begin{tikzpicture}[hole/.style={fill=white,inner sep=1pt},x=1.4cm,y=1.4cm]
\node (Gamma) at (-1.3,0) {$\Gamma$};
\node (A) at (-1.3,1) {$A$};
\node (B) at (-1.3,2) {$B$};
\node (UPi) at (0,0) {$U^{\synPi}$};
\node (U) at (1,0) {$U$};
\node (Ut) at (1,1) {$\tilde{U}$};
\node (Agen) at (0,1) {$A_\gen$};
\node (U2) at (1.6,1) {$U$};
\node (Ut2) at (1.6,2) {$\tilde{U}$};
\node (Bgen) at (0,2) {$B_\gen$};
\draw[->,font=\scriptsize,auto] (Gamma) to node {$\name{(A,B)}$} (UPi);
\draw[->,font=\scriptsize] (A) to (Gamma);
\draw[->,font=\scriptsize] (A) to (Agen);
\draw[->,font=\scriptsize] (B) to (A);
\draw[->,font=\scriptsize] (B) to (Bgen);
\draw[->,font=\scriptsize] (UPi) to (U);
\draw[->,font=\scriptsize] (Agen) to (UPi);
\draw[->,font=\scriptsize] (Ut) to (U);
\draw[->,font=\scriptsize] (Agen) to (Ut);
\draw[->,font=\scriptsize] (Agen) to [out=30,in=135] (U2);
\draw[->,font=\scriptsize] (Bgen) to (Agen);
\draw[->,font=\scriptsize] (Ut2) to (U2);
\draw[->,font=\scriptsize] (Bgen) to (Ut2);
\draw (-1.15,0.65) -- (-0.95,0.65) -- (-0.95,0.85);
\draw (-1.15,1.65) -- (-0.95,1.65) -- (-0.95,1.85);
\draw (0.15,0.65) -- (0.35,0.65) -- (0.35,0.85);
\draw (0.15,1.65) -- (0.35,1.65) -- (0.35,1.85);
\end{tikzpicture} \]
\end{definition}
So $B_\gen \to A_\gen \to U^{\synPi}$ is \emph{generic} among such sequences, and $U^{\synPi}$ \emph{represents} the inputs for a $\synPi$ operation (that is, the premises of the $\synPi\text{-}\form$ rule) on $\CC_U$.

\begin{definition}[cf.~{\cite[Def.~4.1]{voevodsky:products-from-universes}}]
A \emph{$\PiStrux$-structure} on a universe $U$ in a lccc $\CC$ consists of a map
\[ \PiStrux \colon U^{\synPi} \to U. \]
 whose realisation is a dependent product for the generic dependent family of types; that is, it is equipped with an isomorphism $ \PiStrux^* \tilde{U} \iso \toposPi_{\alpha_\gen} B_\gen$ over $U^\synPi$, or equivalently with a map $ \tilde{\PiStrux} \colon \toposSigma_{U^{\synPi} \shortto 1} \toposPi_{\USigstrut \alpha_\gen} B_\gen \to \tilde{U}$ making the square
\[\xymatrix{ 
  \toposSigma_{U^{\synPi} \shortto 1} \toposPi_{\USigstrut \alpha_\gen} B_\gen \ar[r]^-{\tilde{\PiStrux}} \ar[d] & \tilde{U} \ar[d] \\
  U^{\synPi} \ar^-\PiStrux[r] & U   
}\]
a pullback.
\end{definition}

The approach used here gives a template which we follow for all the other constructors, with extra subtleties entering the picture just in the cases of $\synId$-types and (type-theoretic) universes, since these structures are not characterised by strict category-theoretic universal properties.
 
\begin{definition}
Take $U^{\synSigma}$ to be the object representing the premises of the $\synSigma$-\form\ rule:
\[ U^{\synSigma} := \toposSigma_{\Utildestrut U \shortto 1} \toposPi_{\tilde{U} \shortto U} (U \times \tilde{U}) \]
Since these are the same as the premises of the $\synPi$-\form\ rule, we have in this case that $U^{\synSigma} = U^{\synPi}$; and we have again the generic family of types $B_\gen \to A_\gen \to U^{\synSigma}$, as over $U^{\synPi}$.
\end{definition}

\begin{definition}
A \emph{$\SigmaStrux$-structure} on a universe $U$ in a lccc $\CC$ consists of a map
\[ \SigmaStrux \colon U^{\synSigma} \to U \]
whose realisation is a dependent sum for the generic dependent family of types; that is, it is equipped with an isomorphism $ \SigmaStrux^* \tilde{U} \iso \toposSigma_{\alpha_\gen} B_\gen$ over $U^\synSigma$ (or again equivalently with a map $\tilde{\Sigma} \colon \toposSigma_{U^{\synSigma} \shortto 1} \toposSigma_{\USigstrut \alpha_\gen} B_\gen \to \tilde{U}$ making the appropriate square a pullback).
\end{definition}

$\IdStrux$-structure requires a few auxiliary definitions.\footnote{We should thank here Michael Warren and Steve Awodey, who both strongly influenced the current presentation of the definition.} Recall first the classical notion of weak orthogonality of maps:
\begin{definition} \leavevmode 
For maps $i \colon A \to B$, $f \colon Y \to X$ in a category $\CC$, say $i$ is \emph{(weakly) orthogonal} to $f$ if any commutative square from $i$ to $f$ has some diagonal filler:
\[\xymatrix{ A \ar[r] \ar[d]_i &Y \ar[d]^f \\
              B \ar[r] \ar@{.>}[ur] & X }\]
or, in other words, if the function
\begin{align*}
\Hom (B,Y) & \to \Hom (A,Y) \times_{\Hom(A,X)} \Hom (B,X) \\
g & \mapsto (g \cdot i, f \cdot g)
\end{align*}
has a section.

Say $i$ is moreover \emph{stably orthogonal}  to $f$ if for every object $C$ of $\CC$, $C \times i$ is orthogonal to $f$.
\end{definition}

In a cartesian closed category, this notion has an internal analogue:
\begin{definition}
Given maps $i \colon A \to B$, $f \colon Y \to X$ in a cartesian closed category $\CC$, an \emph{internal lifting operation} for $i$ against $f$ is a section of the evident map
$Y^B \to Y^A \times_{X^A} X^B$.  % In case a referee asks for more details: the map is (\lambda g.\, (g \cdot i, f \cdot g))
\end{definition}

The following proposition connects the classical and internal notions:
\begin{proposition} \label{prop:lifting-op-iff-stably-orthog}
Given $i,f$ as above, there exists an internal lifting operation for $i$ against $f$ if and only if $i$ is stably orthogonal to $f$.
\end{proposition}

\begin{proof}
If $i$ is stably orthogonal to $f$, then an internal lifting operation may be obtained as (the exponential transpose of) a filler for the canonical square
\[ \xymatrix@C=2cm{ A \times (Y^A \times_{X^A} X^B) \ar[r]^-{\ev_{A,Y} \cdot (A \times \pi_1)} \ar[d]_{i \times Y^A \times_{X^A} X^B} & Y \ar[d]^f \\
              B \times (Y^A \times_{X^A} X^B) \ar[r]^-{\ev_{B,X} \cdot (B \times \pi_2)} & X. }\]
Conversely, any square from $C \times i$ to $f$ induces a map $C \to Y^A \times_{X^A} X^B$; composing this with an internal lifting operation provides a map $C \to Y^B$, whose transpose is a filler for the square. 
\end{proof}

As shown in \cite{awodey-warren} and \cite{gambino-garner}, the rules for $\synId$-types can be understood roughly as follows.  In a model where dependent types are interpreted as fibrations, the identity type over a type $A$ (in any slice $\CC/\Gamma$) is a factorisation of the diagonal $\Delta_A \colon A \to A \times_\Gamma A$ as a stable trivial cofibration, followed by a fibration.  (Here, by a stable trivial cofibration, we mean a map which is stably orthogonal to fibrations, in $\CC/\Gamma$.)  Additionally, choices of all data (including liftings) must be given which commute with pullbacks in the base $\Gamma$.

In our case, the “fibrations” are just the pullbacks of $p$; so it suffices to consider orthogonality between the first map of the factorisation and $p$ itself.  Moreover, as for $\Pi$- and $\Sigma$-structure above, we demand the structure just in the universal case where $A$ is $\tilde{U}$, in the slice $\CC/U$.  Finally, an internal lifting operation turns out to be exactly the structure required to give chosen lifts commuting with pullbacks. We therefore define:

\begin{definition}
An \emph{$\IdStrux$-structure} on a universe consists of maps 
\[ \IdStrux \colon U^{\synId} := \tilde{U} \times_U \tilde{U} \to U,  \qquad \qquad r \colon \tilde{U} \to \IdStrux^* \tilde{U} \]
such that the triangle
\[  \xymatrix@C=0.1cm{ \tilde{U} \ar[rr]^r \ar[dr]_{\Delta_{\tilde{U}}} & & \IdStrux^* \tilde{U} \ar[dl]^{\IdStrux^* p} \\
             & \tilde{U} \times_U \tilde{U} & } \]
commutes, together with an internal lifting operation $J$ for $r$ against $p \times U$ in $\CC/U$.
\end{definition}

\begin{remark}By virtue of Proposition~\ref{prop:lifting-op-iff-stably-orthog}, we could instead simply stipulate that $r$ be stably orthogonal to $p \times U$. We choose the current version since it provides exactly the structure required for Theorem~\ref{thm:structure-on-U-to-CU}, without requiring any arbitrary choices. 

Another alternative is described in \cite[Sec.\ 2.3]{voevodsky:identity-types-from-universes}.
\end{remark}

\begin{definition}
A \emph{$\WStrux$-structure} on a universe consists of a map 
\[ \WStrux \colon U^{\synW} := \toposSigma_{\Utildestrut U \shortto 1} \toposPi_{\tilde{U} \shortto U} (U \times \tilde{U}) \to U \]
such that $\WStrux^* \tilde{U}$ is an initial algebra for the polynomial endofunctor of $\CC/{U^{\synW}}$ specified by $\beta_\gen \colon B_\gen \to A_\gen$, i.e.\ the endofunctor 
\[ \xymatrix{
\CC/U^{\synW} \ar[rr]^-{\beta_\gen^* \alpha_\gen^*} & & \CC/B_\gen \ar[r]^-{\toposPi_{\beta_\gen}} & \CC/A_\gen \ar[r]^-{\toposSigma_{\alpha_\gen}} & \CC/U^{\synW}
}. \]
(For details on polynomial endofunctors in logical settings, see \cite{moerdijk-palmgren}, \cite{gambino-hyland}.  Intuitively, their initial algebras may be seen as types of well-founded trees, or of syntax over algebraic signatures.)
\end{definition}

\begin{definition}
A \emph{$\zeroStrux$-structure} on $U$ consists of a map $\zeroStrux \colon 1 \to U$ such that $\zeroStrux^*\tilde{U} \iso 0$.

(By analogy with the preceding definitions, one might write $\zeroStrux \colon U^{\synZero}  \to U$ instead and similarly in the next two definitions.  However, since $U^{\synZero}$ is a terminal object, we choose not to do so simply for the sake of readability.)
\end{definition}

\begin{definition}
A \emph{$\oneStrux$-structure} on $U$ consists of a map $\oneStrux \colon 1 \to U$ such that $\oneStrux^* \tilde{U} \iso 1$.
\end{definition}

\begin{definition}
A \emph{$\plusStrux$-structure} on $U$ consists of a map $\plusStrux \colon U \times U \to U$, together with an isomorphism $\plusStrux^* \tilde{U} \iso \pi_1^* \tilde{U} + \pi_2^* \tilde{U}$ in $\CC/(U \times U)$.
\end{definition}

Finally, we consider the structure on $U$ needed to give a universe (in the type-theoretic sense) in $\CC_U$.  Here, for the first time, we need to consider a nested pair of universes, since the internal universe of $\CC_U$ must be some smaller universe $U_0$ in $\CC$.

\begin{definition}
An \emph{internal universe} $(U_0,i)$ in $U$ consists of arrows 
\[ u_0 \colon \pt \to U \qquad \qquad i \colon U_0 := u_0^* \tilde{U} \to U. \]

Given these, $i$ induces by pullback a universe structure $(p_0,\tilde{U}_0,\ldots)$ on $U_0$.  We say that $U_0$ is closed under $\synPi$-types in $U$ if $U_0$ carries a $\Pi$-structure $\PiStrux_0$, commuting with $i$ in the sense that the square
\[ \xymatrix{ U_0^{\synPi} \ar[r]^{i^{\synPi}} \ar[d]_{\PiStrux_0} & U^{\synPi} \ar[d]^\PiStrux \\
              U_0 \ar[r]^i & U } \]
commutes (where the top map is induced by the evident functoriality of $U^{\synPi}$ in $U$).

Similarly, we say that $U_0$ is closed under $\synSigma$-types (resp.\ $\synId$-types, etc.) if it carries a $\SigmaStrux$-structure $\SigmaStrux_0$ (resp.\ an $\IdStrux$-structure $(\IdStrux_0,r_0)$, etc.) commuting with $i$.
\end{definition}

With these structures defined, we can now prove that they are fit for purpose:

\begin{theorem}[cf.~{\cite[Constr.~4.3]{voevodsky:products-from-universes}}, {\cite[Sec.\ 2.4]{voevodsky:identity-types-from-universes}}] \label{thm:structure-on-U-to-CU}
A $\Pi$-structure (resp.\ $\Sigma$-structure, etc.)\ on a universe $U$ induces $\synPi$-type structure (resp.\ $\synSigma$-type structure, etc.)\ on $\CC_U$.

Moreover, an internal universe $(U_0,i)$ in $U$ closed under any combination of $\synPi$\nbhyph types, $\synSigma$\nbhyph types, etc., induces a universe à la Tarski in $\CC_U$ closed under the corresponding constructors.
\end{theorem}

\begin{proof}
This proof is esentially a routine verification; we give the case of $\synPi$-types in full, and leave the rest mostly to the reader.

In a nutshell, the constructor $\synPi$ is induced by the map $\PiStrux$; and the constructors $\lambda$ and $\app$ are induced by the corresponding lccc structure in $\CC$.

Precisely, we treat the rules of $\synPi$-types (corresponding to the components of the desired $\synPi$-type structure) one at a time.

($\synPi$-\form): The premises
\[ \Gamma \types A\ \type \qquad \Gamma,\ x \oftype A \types B\ \type \]
in $\CC_U$ correspond to data in $\CC$ of the form
\[\xymatrix{
   A \ar[r] \ar[d] \pb & \tilde{U} \ar[d] \\
  \Gamma \ar^{\name{A}}[r] & U }
\qquad
\xymatrix{
 B \ar[r] \ar[d] \pb & \tilde{U} \ar[d] \\
 A \ar^{\name{B}}[r] & U }
\]
and hence to a map 
\[ (\name{A},\name{B}) \colon \Gamma \to U^{\synPi}. \]

Then the composite $\PiStrux \cdot (\name{A},\name{B})$ gives a type $\Gamma \to U$ which we take as $\synPi(A,B)$.  By construction, this is stable under substitution along any map $f \colon \Delta \to \Gamma$, since substitution in $\CC_U$ is again just composition in $\CC$.

($\synPi$-\intro): Besides $\Gamma$, $A$, $B$ as before, we have an additional premise
\[ \Gamma,\ x \oftype A \types t : B(x). \]

This is by definition a map $1_A \to B$ in $\CC/A$, corresponding by adjunction to a map $\hat{t} \colon 1_\Gamma \to \toposPi_{A \shortto \Gamma} B$ in $\CC/\Gamma$.  But 
\begin{align*}
  \toposPi_{A \shortto \Gamma} B \iso & (\name{A},\name{B})^* \toposPi_{A_\gen \shortto U^{\synPi}}B_\gen \\
                       \iso & (\name{A},\name{B})^* \PiStrux^* \tilde{U} \\
                       \iso & (\PiStrux \cdot (\name{A},\name{B}))^* \tilde{U} 
\end{align*}
so $\hat{t}$ corresponds to a section of $\synPi(A,B)$ over $\Gamma$, which we take as $\lambda(t)$.

Stability under substitution follows by the uniqueness in the universal property of $\toposPi_{A \shortto \Gamma}B$.

We could alternatively have defined $\lambda$ more analogously to $\synPi$, by representing the premises as a single map $(\name{A},\name{B},t) \colon \Gamma \to U^\lambda$ (where $U^\lambda := \toposSigma_{U^{\synPi} \shortto 1} \toposPi_{A_\gen \shortto U^{\synPi}} B_\gen$ represents the inputs of $\lambda$, i.e.\ the premises of $\synPi\text{-}\intro$); then taking the transpose of the generic term $t_\gen$ over $U^\lambda$; and then pulling this back along $(\name{A},\name{B},t)$.  In fact, thanks to the uniqueness in the universal property of $\toposPi_{A_\gen \shortto U^{\synPi}}B_\gen$, that would give the same result as the present, more straightforward, definitition.  However, the alternative definition has the advantage that its stability under substitution follows simply from properties of pullbacks; this becomes important for $\synId$-types, whose universal property lacks a uniqueness condition.

($\synPi$-\appRule):  The premises now are
\[ \Gamma \types A\ \type \qquad \Gamma,\ x \oftype A \types B\ \type \]
\[ \Gamma \types f : \synPi (A,B) \qquad \Gamma \types a : A \]
corresponding to $\Gamma$, $A$, $B$ as before, plus sections
\[ \xymatrix{
A \ar@/^/[dr] & & \synPi(A,B) \mathrlap{{} \iso \toposPi_{A \shortto \Gamma} B }\ar@/^/[dl] \\
& \Gamma \ar@/^/[ul]^-a \ar@/^/[ur]^-f
}
\phantom{{} \iso \toposPi_{A \shortto \Gamma} B} \]

Together, these give a section over $\Gamma$ of $\toposPi_{A \shortto \Gamma} B \times_\Gamma A$; so composing this with the evaluation map $\ev_{A,B}$ of $\toposPi_{A \shortto \Gamma} B$ gives a map $\Gamma \to B$ lifting $a$, which we take to be $\app(f,a)$.

($\synPi$-\comp): here, we have premises $\Gamma, A, B, t$ as in $\synPi$-\intro, and $a$ as in $\synPi$-\appRule; and we have formed $\app(\lambda(t),a)$ as prescribed above.  So, unwinding the isomorphism $\synPi(A,B) \iso \toposPi_{A \shortto \Gamma} B$ used in each case, 
\begin{align*}
 \app(\lambda(t),a) & = \ev_{A,B} \cdot (\hat{t},a) \\
                                    & = t \cdot a
\end{align*}
as desired, by the usual rules of LCCCs.

This completes the proof for $\Pi$-structures.

As indicated above, the remaining constructors are for the most part entirely analogous; the only subtlety is in the case for the $\synId$-\elim\ rule.   In this case, there are two ways that one could define the appropriate structure: one can either pull back to each specific context and then choose liftings, or choose a lifting in the universal context and then pull it back (as discussed following the $\synPi$-\intro\ case above).  The second of these is the correct choice: the first is not automatically stable under substitution.  (For other constructors, this distinction does not arise, since their strict categorical universal properties canonically determine the maps involved.)  And, in fact, the “universal lifting” required is precisely the internal lifting operation provided by the $\IdStrux$-structure on $U$.

\end{proof}

%%% Local Variables: 
%%% mode: latex
%%% TeX-master: "simplicial-model"
%%% End: 

\section{The Simplicial Model} \label{section:the-model}

In this section, we will apply the techniques of Section \ref{section:models-from-universes} to construct a model of type theory in the category $\sSets$. As mentioned in the Introduction, type dependency is interpreted using Kan fibrations and in particular the closed type will be Kan complexes. To this end, we construct (for any regular cardinal $\alpha$) a Kan fibration $p_\alpha \colon \UUt \to \UU$, weakly universal among Kan fibrations with $\alpha$-small fibers, and investigate the key properties of $\UU$ and $p_\alpha$. We then show that $\UU$ is a Kan complex, and (when $\alpha$ is inaccessible) carries the various logical structures defined in Section~\ref{subsec:logical-structure-on-universes}.  Together, these yield our first main goal: a model of type theory in $\sSets$, with an internal universe.

\subsection{A universe of Kan complexes} \label{subsec:representability-of-fibs}

In constructing a universe $\UU$ intended to represent $\alpha$-small Kan fibrations, one might expect (by the Yoneda lemma) to simply define $(\UU)_n$ as the set of $\alpha$-small fibrations over $\Delta[n]$.  This definition has two problems: firstly, it gives not sets, but proper classes; and secondly, it is not strictly functorial, since pullback is functorial only up to isomorphism.

Some extra technical device is therefore needed to resolve these issues.  Several possible solutions exist\footnote{Other possible approaches include ones based on the general results of \cite{hofmann:on-the-interpretation} and \cite{lumsdaine-warren:local-universes}, or taking $\UU_n$ as $[(\int\! \Delta[n])^\op, \Sets_{<\alpha}]$ as in \cite{hofmann-streicher:lifting-grothendieck-universes}.}; we take the approach of passing to isomorphism classes, having first added well-orderings to the mix so that fibrations have no non-trivial automorphisms (without which the crucial Lemmas~\ref{lemma:w-preserves-lims}, \ref{lemma:w-representable} would fail).  We emphasise, however, that this is the sole reason for introducing the well-orderings: they are of no intrinsic interest or significance. % (and are indeed occasionally something of an inconvenience). % inconvenience = pain in the ass

\begin{definition}
A \emph{well-ordered morphism} of simplicial sets consists of an ordinary map of simplicial sets $f \colon Y \to X$, together with a function assigning to each simplex $x \in X_n$ a well-ordering on the fiber $Y_x := f^{-1}(x) \subseteq Y_n$.

If $f \colon Y \to X$, $f' \colon Y' \to X$ are well-ordered morphisms into a common base $X$, an \emph{isomorphism} of well-ordered morphisms from $f$ to $f'$ is an isomorphism $Y \iso Y'$ over $X$ preserving the well-orderings on the fibers.
\end{definition}

\begin{proposition} \label{prop:well-orderings-rigid}
Given two well-ordered sets, there is at most one isomorphism between them.  Given two well-ordered morphisms over a common base, there is at most one isomorphism between them.
\end{proposition}

\begin{proof}
The first statement is classical (and immediate by induction); the second follows from the first, applied in each fiber.
\end{proof}

\begin{definition}
Fix (for the remainder of this and the following section) a regular cardinal $\alpha$.  Say a map of simplicial sets $f \colon Y \to X$ is \emph{$\alpha$-small} if each of its fibers $Y_x$ has cardinality $< \alpha$.
\end{definition}

Given a simplicial set $X$, define $\W (X)$ to be the set of isomorphism classes\footnote{We use \emph{isomorphism classes} in the sense of “Scott’s trick” \cite{scott:scotts-trick} for constructing proper class quotients.  The class of \emph{all} well-ordered morphisms isomorphic to a given one is a proper class, so one instead uses the subclass of such morphisms \emph{of minimal rank}, which is a set.} of $\alpha$-small well-ordered morphisms $Y \to X$; together with the pullback action $\W(f) := f^* \colon \W(X) \to \W(X')$, for $f \colon X' \to X$, this gives a contravariant functor $\W \colon \sSets^{\op} \to \Sets$.

\begin{lemma} \label{lemma:w-preserves-lims}
$\W$ preserves all limits: $\W(\colim_i X_i) \iso \lim_i \W(X_i)$. 
\end{lemma}

\begin{proof}
Suppose $F \colon \I \to \sSets$ is some diagram, and $X = \colim_\I F$ is its colimit, with injections $\nu_i \colon F(i) \to X$.  We need to show that the canonical map $\W(X) \to \lim_\I \W(F(i))$ is an isomorphism.

To see that it is surjective, suppose we are given $[f_i \colon Y_i \to F(i)] \in \lim_\I \W(F(i))$.  For each $x \in X_n$, choose some $i$ and $\bar{x} \in F(i)$ with $\nu(\bar{x}) = x$, and set $Y_x := (Y_i)_{\bar{x}}$.  By Proposition~\ref{prop:well-orderings-rigid}, this is well-defined up to \emph{canonical} isomorphism, independent of the choices of representatives $i$, $\bar{x}$, $Y_i$, $f_i$.  The total space of these fibers then defines a well-ordered morphism $f \colon Y \to X$, with fibers of size $<\alpha$, and with pullbacks isomorphic to $f_i$ as required.

For injectivity, suppose $f, f'$ are well-ordered morphisms over $X$, and $\nu_i^* f \iso \nu_i^* f'$ for each $i$.  By Proposition~\ref{prop:well-orderings-rigid}, these isomorphisms must agree on each fiber, so together give an isomorphism $f \iso f'$.
\end{proof}

Define the simplicial set $\WW$ by
\[\WW := \W \cdot \y^{\op} \colon \Delta^{\op} \to \Sets,\]
where $\y$ denotes the Yoneda embedding $\Delta \to \sSets$.

\begin{lemma} \label{lemma:w-representable}
The functor $\W$ is representable, represented by $\WW$.
\end{lemma}

\begin{proof}
The functors $\W$ and $\Hom(-,\WW)$ agree up to isomorphism on the standard simplices (by the Yoneda lemma), and send colimits in $\sSets$ to limits; but every simplicial set is canonically a colimit of standard simplices.
\end{proof}

\begin{notation}
Given an $\alpha$-small well-ordered map $f \colon Y \to X$, the corresponding map $X \to \WW$ will be denoted by $\name{f}$.
\end{notation}

Applying the natural isomorphism above to the identity map $\WW \to \WW$ yields a universal $\alpha$-small well-ordered simplicial set $\WWt \to \WW$.  Explicitly, $n$-simplices of $\WWt$ are classes of pairs
\[(f \colon Y \to \Delta [n], s \in f^{-1}(1_{[n]}))\]
i.e.\ the fiber of $\WWt$ over an $n$-simplex $\name{f} \in \WW$ is exactly (an isomorphic copy of) the main fiber of $f$.  So, by construction:

\begin{proposition}
The canonical projection $\WWt \to \WW$ is strictly universal for $\alpha$-small well-ordered morphisms; that is, any such morphism can be expressed uniquely as a pullback of this projection. \qed
\end{proposition}

\begin{corollary}
The canonical projection $\WWt \to \WW$ is weakly universal for $\alpha$-small morphisms of simplicial sets: any such morphism can be given, not necessarily uniquely, as a pullback of this projection.
\end{corollary}

\begin{proof}
By the well-ordering principle and the axiom of choice, one can well-order the fibers, and then use the universal property of $\WW$.
\end{proof}

\begin{definition}
 Let $\U \subseteq \W$ (respectively, $\UU \subseteq \WW$) be the subobject consisting of (isomorphism classes of) $\alpha$-small well-ordered fibrations\footnote{Here and throughout, by ``fibration'' we always mean ``Kan fibration''.}; and define $p_\alpha \colon \UUt \to \UU$ as the pullback:
 \[\xymatrix{ \UUt \ar[r] \ar[d]_{p_\alpha}  \pb & \WWt \ar[d] \\
  \UU \ar@{^{(}->}[r] & \WW
 }\]
\end{definition}

\begin{lemma}\label{U:Kan_fib}
 The map $p_\alpha \colon \UUt \to \UU$ is a fibration.
\end{lemma}

\begin{proof}
 Consider a horn to be filled
  \[\xymatrix{ \Lambda^k [n] \ar[r] \ar@{^{(}->}[d] & \UUt \ar[d]^{p_\alpha} \\
  \Delta [n] \ar[r]^{\name{x}} & \UU
 }\]
 for some $0 \leq k \leq n$.  It factors through the pullback
   \[\xymatrix{ \Lambda^k [n] \ar[r] \ar@{^{(}->}[d] & \bullet \ar[r] \ar[d]^{x} \pb & \UUt \ar[d]^{p_\alpha} \\
  \Delta [n] \ar@{=}[r] & \Delta [n] \ar[r]^{\name{x}} & \UU
 }\]
 where by the definition of $\UU$ and $\UUt$, $x$ is a fibration. Thus the left square admits a diagonal filler, and hence so does the outer rectangle.
\end{proof}

\begin{lemma}
 An $\alpha$-small well-ordered morphism $f \colon Y \to X \in \W (X)$ is a fibration if and only if $\name{f} \colon X \to \WW$ factors through $\UU$.
\end{lemma}

\begin{proof}
 For `$\Rightarrow$', assume that $f \colon Y \to X$ is a fibration. Then the pullback of $f$ to any representable is certainly a fibration:
 \[\xymatrix{ \bullet \ar[r] \ar[d]_{x^*f} \pb & Y \ar[d]^{f} \\
 \Delta [n] \ar[r]^x & X.
 }\]
 so $\name{f}(x) =\name{ x^*f} \in \UU$, and hence $\name{f}$ factors through $\UU$.

 Conversely, suppose $\name{f}$ factors through $\UU$. Then we obtain:
 \[\xymatrix{ Y \ar[r] \ar[d]_{f} \pb & \UUt \ar[r] \ar[d]^{p_\alpha} \pb & \WWt \ar[d] \\
 X \ar[r] & \UU \ar@{^{(}->}[r] & \WW,
 }\]
 where the lower composite is $\name{f}$, and the outer rectangle and the right square are by construction pullbacks.  Hence so is the left square; so by Lemma~\ref{U:Kan_fib} $f$ is a fibration.
\end{proof}

\begin{corollary} \label{cor:U_classifies}
The functor $\U$ is representable, represented by $\UU$; so ${p_\alpha} \colon \UUt \to \UU$ is strictly universal for $\alpha$-small well-ordered fibrations, and weakly universal for $\alpha$-small fibrations. \qed
\end{corollary}

In Section~\ref{subsec:pullback-reps}, we will strengthen this universal property, showing that while the representation of a fibration as a pullback of $p_\alpha$ may not be strictly unique, it is unique up to homotopy: precisely, the space of such representations is contractible.

\subsection{Kan fibrancy of the universe}  \label{subsec:fibrancy-of-u}

The previous section provides the main ingredients needed to use $\UU$ as a universe in the sense of Section~\ref{section:models-from-universes}, and hence to give a model of the core type theory.  However, to give additionally a type-theoretic universe within that model, we need to show that each $\UU$ itself can be seen as a \emph{type} of the model; in other words, that it is Kan.  The main goal of this section is therefore to prove the following theorem:

\begin{theorem}\label{U:KanCpx}
 The simplicial set $\UU$ is a Kan complex.
\end{theorem}

Before proceeding with the proof we will gather four useful lemmas.  The first two concern \emph{minimal fibrations}, which for the present purposes are a technical device whose details, beyond these two lemmas, are unimportant.

\begin{lemma}[Quillen's Lemma, {\cite{quillen:minimal}}] \label{Quillen_lemma}
Any fibration $f \colon Y \to X$ may be factored as $f = pg$, where $p$ is a minimal fibration and $g$ is a trivial fibration.
\end{lemma}

\begin{lemma}[{\cite[III.5.6]{barratt-gugenheim-moore}}; see also {\cite[Cor.~11.7]{may:simplicial-book}}] \label{lemma:may-lemma}
Suppose $X$ is contractible, with $x_0 \in X$, and $p \colon Y \to X$ is a minimal fibration with fiber $F := Y_{x_0}$. Then there is an isomorphism over $X$:
\[\xymatrix@C=0.5cm{
  Y \ar[rr]^<>(0.5)g \ar[rd]_p & & F \times X \ar[ld]^{\pi_2} \\
  & X. &
}\]
\end{lemma}

For Lemma~\ref{lemma:joyal-lemma}, the proof we give is due to Andr\'e Joyal; we include details here since the original \cite{joyal:kan} is not currently publicly available.  For this, and again for Theorem~\ref{thm:univalence} below, we make crucial use of exponentiation along cofibrations; so we pause first to establish some facts about this.

\begin{lemma}[Cf.~{\cite[Lemma~0.2]{joyal:kan}}] \label{lemma:exp_along_cofib}
For any map $i \colon A \to B$,
\begin{enumerate}[1.]
\item $\toposPi_i \colon \sSets/A \to \sSets/B$ preserves trivial fibrations;
\end{enumerate}
\noindent and if moreover $i$ is a cofibration, then: 
\begin{enumerate}[1.] \setcounter{enumi}{1}
\item the counit $i^* \toposPi_i \to 1_{\sSets/A}$ is an isomorphism;
\item if $p \colon E \to A$ is $\alpha$-small, then so is $\toposPi_i p$. \end{enumerate}
\end{lemma}

\begin{proof} \
\begin{enumerate}[1.]
\item By adjunction, since $i^*$ preserves cofibrations.

\item Since $i$ is mono, $i^* \toposSigma_i \iso 1_{\sSets/A}$; so by adjointness, $i^*\toposPi_i \iso 1_{\sSets/A}$.

\item For any $n$-simplex $x \colon \Delta[n] \to B$, we have $(\toposPi_i p)_x \iso \Hom_{\sSets/B}(x,\toposPi_i p) \iso \Hom_{\sSets/B}(i^*x,p)$.  As a subobject of $\Delta[n]$, $i^*x$ has only finitely many non-degenerate simplices, so $(\toposPi_i p)_x$ injects into a finite product of fibers of $p$ and is thus of size $< \alpha$. \qedhere
\end{enumerate} \end{proof}

\begin{lemma}[{\cite[Lemma~0.2]{joyal:kan}}] \label{lemma:joyal-lemma}
Trivial fibrations extend along cofibrations.  That is, if $t \colon Y \to X$ is a trivial fibration and $j \colon X \to X'$ is a cofibration, then there exists a trivial fibration $t' \colon Y' \to X'$ and a pullback square of the form:
 \[\xymatrix{ Y \ar@{.>}[r] \ar[d]_t \pb & Y' \ar@{.>}[d]^{t'} \\
 X \ar@{^{(}->}[r]^j & X'.
 }\]

Moreover, if $t$ is $\alpha$-small, then $t'$ may be chosen to also be.
\end{lemma}

\begin{proof}
Take $t' := \toposPi_j t$.  By part 1 of Lemma~\ref{lemma:exp_along_cofib}, this is a trivial fibration; by part 2, $j^*Y' \iso Y$; and by part 3, it is $\alpha$-small.
\end{proof}

We are now ready to prove that $\UU$ is a Kan complex.

\begin{proof}[Proof of Theorem~\ref{U:KanCpx}]
We need to show that we can extend any horn in $\UU$ to a simplex:
 \[\xymatrix{ \Lambda^k[n] \ar[r] \ar@{^{(}->}[d] & \UU \\
  \Delta [n] \ar@{.>}[ur] & \\
 }\]
By Corollary~\ref{cor:U_classifies}, any such horn $\name{q}$ corresponds to an $\alpha$-small well-ordered fibration $q \colon Y \to \Lambda^k [n]$.  To extend $\name{q}$ to a simplex, we just need to construct an $\alpha$-small fibration $Y'$ over $\Delta[n]$ which restricts on the horn to $Y$:
\[\xymatrix{
  Y \ar@{.>}[r] \ar[d]_q \pb & Y' \ar@{.>}[d]^{q'} \\
 \Lambda^k[n] \ar@{^{(}->}[r] & \Delta[n].
}\]
By the axiom of choice one can then extend the well-ordering of $q$ to $q'$, so the map $\name{q'} \colon \Delta [n] \to \UU$ gives the desired simplex.

By Quillen's Lemma, we can factor $q$ as
\[\xymatrix{ Y \ar[r]^{q_t} & Y_0 \ar[r]^{q_m} & \Lambda^k[n],}\]
where $q_t$ is a trivial fibration and $q_m$ is a minimal fibration.  Both are still $\alpha$-small: each fiber of $q_t$ is a subset of a fiber of $q$, and since a trivial fibration is onto, each fiber of $q_m$ is a quotient of a fiber of $q$.

By Lemma~\ref{lemma:may-lemma}, we have an isomorphism $Y_0 \cong F \times \Lambda^k[n]$, and hence a pullback diagram:
 \[\xymatrix{Y_0 \ar@{^{(}->}[r] \ar[d] \pb & F \times \Delta[n] \ar[d] \\
 \Lambda^k[n] \ar@{^{(}->}[r] & \Delta[n]
 }\]

By Lemma~\ref{lemma:joyal-lemma}, we can then complete the upper square in the following diagram, with both right-hand vertical maps $\alpha$-small fibrations:
 \[\xymatrix{ Y \ar[d]_{q_t} \ar[r] \pb & Y' \ar[d] \\
  Y_0 \ar@{^{(}->}[r] \ar[d]_{q_m} \pb & F \times \Delta[n] \ar[d] \\
 \Lambda^k[n] \ar@{^{(}->}[r] & \Delta[n]
 }\]

  Since $\alpha$ is regular, the composite of the right-hand side is again $\alpha$-small; so we are done.
\end{proof}

\subsection{Modelling type theory in simplicial sets} \label{subsec:model-in-ssets}

To prove that $\UU$ carries the structure to model type theory, we will need a couple of further lemmas; firstly, that taking dependent products preserves fibrations:

\begin{lemma} \label{lemma:dep-prod-of-fibs}
Suppose $Z \to^q Y \to^p X$ are fibrations.  Then the dependent product $\toposPi_p q$ is a fibration over $X$.
\end{lemma}

\begin{proof}
The pullback functor $p^* \colon \sSets/X \to \sSets/Y$ preserves trivial cofibrations (since $\sSets$ is right proper and cofibrations are monomorphisms); so its right adjoint $\toposPi_p$ preserves fibrant objects.
\end{proof}

Secondly, to model $\synId$-types, we will require well-behaved fibered path objects.  The construction below may be found in \cite[Thm.~2.25]{warren:thesis}; we recall it in more elementary terms, which will be useful to us later.

\begin{definition}
Given a fibration $p \colon E \to B$, define the \emph{fibered path object} $\paths_B(E)$ as the pullback
\[\xymatrix{
  \paths_B(E) \ar[r] \ar[d] \pb & E^{\Delta[1]} \ar[d]^-{p^{\Delta[1]}} \\
  B \ar[r]^-{c} & B^{\Delta[1]},
}\]
the object of paths in $E$ that are constant in $B$.

The “constant path” map $c \colon E \to E^{\Delta[1]}$ factors through $\paths_B(E)$; call the resulting map $r_p \colon E \to \paths_B(E)$.  There are also evident source and target maps $s_p,t_p \colon \paths_B(E) \to E$.  (On all of these maps, we will omit the subscripts when they are clear from context.)
\end{definition}

\begin{proposition} \label{prop:stable-fibered-path-obs}
For any fibration $p \colon E \to B$, the maps 
\[E \to^r \paths_B(E) \to^{(s,t)} E \times_B E\]
give a factorisation of the diagonal map $\Delta_p \colon E \to E \times_B E$ over $B$ as a (trivial cofibration, fibration); and this is \emph{stable over $B$} in that the pullback along any $B' \to B$ is again such a factorisation. 
\end{proposition}

\begin{proof}
It is clear that these maps give a factorisation of $\Delta_p$ over $B$.  To see that they are a trivial cofibration and a fibration respectively, consider the pullback construction of $\paths_B(E)$ via two intermediate stages:
\[\xymatrix{
  \paths_B(E) \ar[d]_-{(s,t)} \ar[r] \pb & E^{\Delta[1]} \ar[d]^-{(s,p^{\Delta[1]},t)} \\
  E \times_B E \ar[d]_-{\pi_1} \ar[r] \pb & E \times_B B^{\Delta[1]} \times_B E \ar[d]^-{(\pi_1,\pi_2)} \\
  E \ar[d] \ar[d] \ar[r] \pb & E \times_B B^{\Delta[1]} \ar[d] \\
  B \ar[r]^{c} & B^{\Delta[1]} 
}\]

Now $(s,t)$ is certainly a fibration, since it is a pullback of the map $E^{\Delta[1]} \to E \times_B B^{\Delta[1]} \times_B E \iso E^{1+1} \times_{B^{1+1}} B^{\Delta[1]}$, which is a fibration by the monoidal model category axioms \cite[Lemma~4.2.2(3)]{hovey:book}, applied to the cofibration $1+1 \to \Delta[1]$ and the fibration $p$.

Similarly, the source map $s \colon \paths_B(E) \to E$ is a trivial fibration, since it is a pullback of $E^{\Delta[1]} \to E^1 \times_{B^1} B^{\Delta[1]}$, which is one by the monoidal model category axioms.  But $s$ is a retraction of $r$, so $r$ is a weak equivalence (by 2-out-of-3) and a monomorphism, so is a trivial cofibration as desired.

Finally, stability of these properties under pullback follows immediately from the stability (up to isomorphism) of the construction itself: for any $f \colon B' \to B$, there is a canonical isomorphism $\paths_{B'} (f^*E) \iso f^* \paths_B(E)$, commuting with the maps $r,s,t$.
\end{proof}

We are now fully equipped for the main result of the present section:
\begin{theorem} \label{thm:the-model-in-ssets}
Let $\alpha$ be an inaccessible cardinal\footnote{I.e.\ infinite, regular, and strong limit; \emph{strongly inaccessible} in some literature.}. Then $\UU$ carries $\PiStrux$-, $\SigmaStrux$-, $\IdStrux$-, $\WStrux$-, $\oneStrux$-, $\zeroStrux$-, and $\plusStrux$-structures.  

Moreover, if $\beta < \alpha$ is also inaccessible, then $\UU[\beta]$ gives an internal universe in $\UU[\alpha]$ closed under all these constructors.
\end{theorem}

\begin{proof}
($\PiStrux$-structure): Given a pair of $\alpha$-small fibrations $Z \to^q Y \to^p X$, the dependent product $\toposPi_p q$ in $\sSets/X$ is again a fibration, by Lemma~\ref{lemma:dep-prod-of-fibs}; it is also $\alpha$-small, since $\alpha$ is inaccessible.

Hence by Corollary~\ref{cor:U_classifies}, the universal dependent product over $\UU^{\synPi\text{-\form}}$ is representable as the pullback of $\UUt$ along some map $\PiStrux \colon \UU^{\synPi\text{-\form}} \to \UU$, giving the desired $\PiStrux$-structure.

($\SigmaStrux$-structure): Similarly, given $\alpha$-small fibrations $Z \to^q Y \to^p X$, the composite $p \cdot q$ is again an $\alpha$-small fibration.  So the universal dependent sum over $\UU^{\synSigma\text{-\form}}$ is representable by some map $\SigmaStrux \colon \UU^{\synSigma\text{-\form}} \to \UU$.

($\IdStrux$-structure): Given any $\alpha$-small fibration $p \colon Y \to X$, consider the factorisation of the diagonal $\Delta_p$ as $Y \to^{r} \paths_X(Y) \to^{(s,t)} Y \times_X Y$.  The fibration $(s,t)$ is easily seen to be $\alpha$-small; and by Proposition~\ref{prop:stable-fibered-path-obs}, $r$ is stably orthogonal to $(s,t)$ over $X$.

Applying this construction to $p_\alpha \colon \UUt \to \UU$ itself yields, via Proposition~\ref{prop:lifting-op-iff-stably-orthog}, the desired $\IdStrux$-structure on $\UU$.

($\WStrux$-structure): Given $\alpha$-small fibrations $Z \to^q Y \to^p X$, the initial algebra $W_q \to X$ for the induced polynomial endofunctor on $\sSets/X$ may be obtained as a transfinite colimit of iterations of the endofunctor; it can be shown from this description that it is again an $\alpha$-small fibration (cf.~\cite[Thm.~3.4]{moerdijk-van-den-berg:w-types-in-hott} and \cite{moerdijk-van-den-berg:w-types-in-hott-corr}).

($\zeroStrux$-structure), ($\oneStrux$-structure), ($\plusStrux$-structure): straightforward.

(Internal universe.)  Since $\beta < \alpha$, $\UU[\beta]$ is itself $\alpha$-small; and by Theorem~\ref{U:KanCpx}, it is Kan.  So $\UU[\beta]$ is representable as the pullback of $\UUt$ along some $u_\beta \colon 1 \to \UU$.  Moreover, there is a natural inclusion $i \colon \UU[\beta] \to \UU$, with $\UUt[\beta] \iso i^* \UUt$ by construction.  Together these give the desired internal universe $(u_\beta, i)$.

Finally, to see that $(u_\beta,i)$ is closed under the appropriate constructors in $i$, note that for each of $\PiStrux$, $\SigmaStrux$, and $\IdStrux$ as constructed above, the image of the composite with $i$ lies again in $\UU[\beta]$, and hence factors through $i$; for instance, in the case of $\PiStrux$,
\[ \xymatrix{ \UU[\beta]^{\synPi\text{-\form}} \ar[r]^{i^{\synPi\text{-\form}}} \ar@{.>}[d]^{\PiStrux} & \UU^{\synPi\text{-\form}} \ar[d]^\PiStrux \\
              \UU[\beta] \ar[r]^i & \UU } \]

(Note that while we do already have a $\PiStrux$-structure (and so on) on $\UU[\beta]$ as constructed in the first parts of this theorem, those choices of the structure do not automatically commute with $i$.)  
\end{proof}

\begin{corollary} \label{cor:simplicial-model}
Let $\beta < \alpha$ be inaccessible cardinals.  Then there is a model of dependent type theory in $\sSets_{\UU}$ with all the logical constructors of Section~\ref{subsec:logical-rules}, and a universe (given by $\UU[\beta]$) closed under these constructors. \qed
\end{corollary}

Assuming initiality (Conjecture~\ref{conj:initiality}), this implies the existence of a morphism $\CC(\TT) \to \sSets_{\UU}$, interpreting the syntax of Martin-Löf type theory in simplicial sets.  Even without assuming initiality, it gives us the operations with which to heuristically interpret individual judgements of the syntax by hand.  We therefore freely make notational use of the interpretation, writing $\interp{ \mathcal{J} }$ for the interpretation of a judgement $\mathcal{J}$.

In doing so, we will make several systematic abuses of notation.  Firstly, referring in the syntax to fibrations, we will write $E$ rather than $\name{E}$, and so on, whenever some choice of name $\name{E} \colon B \to \UU$ for the fibration is understood; and conversely, referring to the interpretation of a type $\Gamma \types T\ \type$, we use $\interp{T}$ to refer to the fibration over $\interp{\Gamma}$ given by pulling back $\UUt$ along the literal interpretation $\interp{\Gamma \types T\ \type} \colon \interp{\Gamma} \to \UU$.

As a first characteristic of the model, we note that both of the extra principles on equality of functions hold.

\begin{proposition} \label{prop:eta-and-funext}
The $\eta$-rule and functional extensionality rules of Section~\ref{subsec:optional-rules} hold in the simplicial model.
\end{proposition}

\begin{proof}
The $\eta$-rule follows immediately from our use of categorical exponentials to interpret $\synPi$-types, by the uniqueness in the categorical universal property.

For functional extensionality, Garner \cite[Sec.~5]{garner:on-the-strength} shows that it holds just if each product of identity types, 
\[ f, g \oftype \synPi_{x \oftype A} B(x) \types \synPi_{x \oftype A} \synId_{B(x)} (\syn{app}(f,x),\syn{app}(g,x))\ \type\]
admits the structure given by the rules for the identity type on the corresponding product types,
\[f,g \oftype \synPi_{x \oftype A} B(x) \types \synId_{\synPi_{x \oftype A} B(x)} (f,g)\ \type . \]

So it is enough to show that for any pair of ($\alpha$-small, well-ordered) fibrations $Z \to^q Y \to^p X$, given by names
\[ \name{Y} \colon X \to \UU, \qquad \name{Z} \colon Y \to \UU,\]
the interpretation of the product of identity types
\[ \interp{\synPi_{x \oftype Y} \synId_{Z(x)} (\syn{app}(f,x),\syn{app}(g,x)) } \iso \toposPi_p (\paths_Y Z),\]
 gives a suitably stable path object for the interpretation of the product types, 
\[ \interp{\synId_{\synPi_{x \oftype Y} Z(x)} (f,g)} \iso \toposPi_p Z.\]

For this, it is clear that $\toposPi_p (s,t) \colon \toposPi_p (\paths_Y Z) \to \toposPi_p (Z \times_Y Z) \iso \toposPi_p Z \times_X \toposPi_p Z$ is a fibration, since $\toposPi_p$ preserves fibrations (Lemma~\ref{lemma:dep-prod-of-fibs}). Similarly, $\toposPi_p r_q \colon \toposPi_p Z \to \toposPi_p (\paths_Y Z)$ is a cofibration since $\toposPi_p$ preserves monomorphisms; and it is a weak equivalence, since $\toposPi_p$ preserves trivial fibrations (Lemma~\ref{lemma:joyal-lemma}), and so the retraction $\toposPi_p s_q \colon \toposPi_p (\paths_Y Z) \to \toposPi_p Z$ is again a trivial fibration.  Finally, by the Beck-Chevalley condition in an LCCC, the entire construction is stable under pullback in $X$, as required.
\end{proof}

It now remains only to show that the Univalence Axiom holds in this model.

%%% Local Variables: 
%%% mode: latex
%%% TeX-master: "simplicial-model"
%%% End: 

\section{Univalence} \label{section:univalence}

In this section, we will introduce the Univalence Axiom, and show that it holds in the simplicial model.

The proof of this involves both simplicial and type-theoretic components; we keep these separate, as far as possible.  First of all (Section~\ref{subsec:type-theoretic-univalence}), we define univalence type-theoretically and state the Univalence Axiom; next, we define an analogous simplicial concept of univalence (Section~\ref{subsec:simplicial-univalence}); we then show that via the simplicial model, the two notions coincide (Section~\ref{subsec:univalence-equivalence}).  Finally, in Section~\ref{subsec:univalence-of-uu}, we prove our main theorem: that $\UU$ is univalent (using the simplicial sense), and hence that the Univalence Axiom holds in the simplicial model of type theory.  Lastly, in Section~\ref{subsec:pullback-reps}, we discuss an alternative formulation of simplicial univalence, and so obtain an up-to-homotopy uniqueness statement for the weak universal property of $\UU$.

Once again, we freely use the syntax of type theory as a notation; to avoid formal dependence on Conjecture~\ref{conj:initiality}, the reader should translate each individual syntactic expression used into the language of contextual categories.

\subsection{Type-theoretic univalence} \label{subsec:type-theoretic-univalence}

To state the univalence axiom, we first need to define a few basic notions in the type theory.

\begin{definition}[Joyal] \label{def:hiso} Let $f \colon A \to B$ be a function in some context $\Gamma$, i.e.\ $\Gamma\types f : [A,B]$ (where the function type $[A,B]$ is defined using $\synPi$, as described in Section~\ref{subsec:logical-rules}).
\begin{itemize}
\item A \emph{left homotopy inverse} for $f$ is a function $g \colon B \to A$, together with a homotopy $g \cdot f \homot 1_A$.  Formally, we define the type $\synLHInv(f)$ of left homotopy inverses to $f$:
\[ \Gamma \types \synLHInv(f) := \Sigma_{g \oftype [B,A]} \synPi_{x \oftype A} \synId_A(g(f(x)),x)\ \type \]

\item Analogously, we define the type $\synRHInv(f)$ of \emph{right homotopy inverses}:
\[ \Gamma \types \synRHInv(f) := \Sigma_{g \oftype [B,A]} \synPi_{y \oftype B} \synId_B(f(g(y)),y)\ \type \]

\item We say $f \colon A \to B$ is a \emph{homotopy isomorphism} (or more briefly, an h-isomorphism) if it is equipped with both a left and a right inverse:
\[ \Gamma \types \synisHIso(f) := \synLHInv(f) \times \synRHInv(f) \ \type \]

\item For any types $A$ and $B$, we thus have the type of h-isomorphisms from $A$ to $B$:
\[ \Gamma \types \synHIso(A,B) := \Sigma_{f \oftype [A,B]} \synisHIso(f) \]
\end{itemize}
\end{definition}

It may perhaps be surprising that we use homotopy isomorphisms rather than the more familiar homotopy equivalences, with a single two-sided homotopy inverse.  The reason is that while a map carries either structure if and only if it carries the other, the type, or object, of such structures on a map is different.  In particular, the analogue of Lemma~\ref{lemma:hiso-over-eq} for homotopy equivalences does not hold; for further discussion of these issues, see \cite[Ch.~4]{hott:book}.

\begin{example}
For any type $B$, the identity function on $B$ is canonically an h-isomorphism.
\end{example}

Suppose now that $A$ is any type, and $x \colon A \types B(x)\ \type$ a family of types over $A$.  By the identity elimination rule, we can derive
\[ x,y \oftype A,\ u \oftype \synId_A(x,y) \types w_{x,y,u} : \synHIso(B(x),B(y)). \]

This can equivalently be seen as a map
\[ x,y \oftype A \types w_{x,y} : [\synId_A(x,y), \synHIso(B(x),B(y))]. \]

\begin{definition} \label{def:type-theoretic-univalence}
We say the family $B(x)$ is \emph{univalent} if for each $x,y$, the map $w_{x,y}$ is itself a homotopy isomorphism:
\[ \types \synisUnivalent(x \oftype A . B(x)) := \synPi_{x,y \oftype A} \synisHIso(w_{x,y}). \]
\end{definition}

\begin{axiom} \label{axiom:univalence}
The \emph{Univalence Axiom}, for a given type-theoretic universe $U$, is the statement that the canonical family $\el$ of types over $U$ is univalent.
\end{axiom}

Informally, the Univalence Axiom says that just as elements of the universe correspond to types, so equalities in the universe correspond to equivalences between types.  In particular, since every statement or construction must respect propositional equality, the Univalence Axiom stipulates that the language can never distinguish between equivalent types.

\subsection{Simplicial univalence}  \label{subsec:simplicial-univalence}

To define a simplicial notion of univalence, we first need to construct the \emph{object of weak equivalences} between fibrations $p_1 \colon E_1 \to B$ and $p_2 \colon E_2 \to B$ over a common base.  In other words, we want an object representing the functor sending $(X,f) \in \sSets/B$ to the set $\extEq_X(f^*E_1,f^*E_2)$.  As we did for $\U$, we proceed in two steps, first exhibiting it as a subfunctor of a functor more easily seen (or already known) to be representable.

For the remainder of the section, fix fibrations $E_1$, $E_2$ as above over a base $B$. Since $\sSets$ is locally Cartesian closed, we can construct the exponential object between them:

\begin{definition} \label{def:internal-eq}
 Let $\intHom_B (E_1, E_2) \to B$ denote the internal hom from $E_1$ to $E_2$ in $\sSets/B$.

 Then for any $X$, a map $X \to \intHom_B (E_1,E_2)$ corresponds to a map $f \colon X \to B$, together with a map $u \colon f^*E_1 \to f^*E_2$ over $X$.

 Together with the Yoneda lemma, this implies the explicit description: an $n$-simplex of $\intHom_B (E_1,E_2)$ is a pair
 \[(b \colon \Delta[n] \to B, u \colon b^* E_1 \to b^* E_2) .\]
\end{definition}

\begin{lemma} \label{lemma:HOM_is_fib}
 $\intHom_B (E_1, E_2) \to B$ is a Kan fibration.
\end{lemma}

\begin{proof}
Follows immediately from Lemma~\ref{lemma:dep-prod-of-fibs}, since the exponential is a special case of dependent products.
\end{proof}

Within $\intHom_B (E_1, E_2)$, we now want to construct the subobject of weak equivalences.

\begin{lemma} \label{lemma:weqs_pull_back} %% VV Lemma 1.7
Let $f \colon E_1 \to E_2$ be a weak equivalence over $B$, and suppose $g \colon B' \to B$. Then the induced map between pullbacks $g^*E_1 \to g^*E_2$ is a weak equivalence.
\end{lemma}

\begin{proof}
The pullback functor $g^* \colon \sSets/B \to \sSets/B'$ preserves trivial fibrations; so by Ken Brown's Lemma \cite[Lemma~1.1.12]{hovey:book}, it preserves all weak equivalences between fibrant objects.
\end{proof}

Thus, weak equivalences from $E_1$ to $E_2$ form a subfunctor of the functor of maps from $E_1$ to $E_2$.  To show that this is representable, we need just to show:

\begin{lemma} \label{lemma:weq_fibers}  %% VV Lemma 1.8
 Let $f \colon E_1 \to E_2$ be a morphism over $B$.  If for each simplex $b \colon \Delta[n] \to B$ the induced map $f_b \colon b^*E_1 \to b^* E_2$ is a weak equivalence, then $f$ is a weak equivalence.
\end{lemma}

\begin{proof}
Without loss of generality, $B$ is connected; otherwise, apply the result over each connected component separately.  Take some vertex $b \colon \Delta[0] \to B$, and set $F_i := b^*E_i$.  Now for any vertex $e \colon \Delta[0] \to F_1$, and any $n \geq 1$, we have by the long exact sequence for a fibration:
 \[\mathclap{\xymatrix@C=0.5cm{
 \pi_{n+1} (B, b) \ar[r] \ar[d]_{1} & \pi_{n} (F_1, e) \ar[r] \ar[d]_{\pi_n (f_b)} & \pi_{n} (E_1, e) \ar[r] \ar[d]_{\pi_n(f)} & \pi_{n} (B, b) \ar[r] \ar[d]^{1} & \pi_{n-1} (F_1, e) \ar[d]^{\pi_{n-1} (f_b)} \\
 \pi_{n+1} (B, b) \ar[r] & \pi_{n} (F_2, f(e)) \ar[r] & \pi_{n} (E_2, f(e)) \ar[r] & \pi_{n} (B, b) \ar[r]  & \pi_{n-1} (F_2, f(e))  }}\]
Each $\pi_n(f_b)$ is an isomorphism, so by the Five Lemma, so is $\pi_n (f)$, for $n \geq 1$.  

The case $n = 0$ is the same in spirit, but requires a little more work since the Five Lemma is unavailable.  We have a square
 \[\xymatrix{
 \pi_0 (F_1) \ar@{->>}[r]^{\pi_0(i_1)} \ar[d]_{\pi_0 (f_b)} & \pi_0 (E_1) \ar[d]_{\pi_0 (f)} \\
 \pi_0 (F_2) \ar@{->>}[r]^{\pi_0(i_2)} & \pi_0 (E_2) }\]
with both horizontal arrows surjections, and $\pi_0(f_b)$ an isomorphism.  To show that $\pi_0(f)$ is an isomorphism, it therefore suffices to show that for each $x \in \pi_0(E_1)$, the restriction of $f_b$ to a map of fibers 
\[ \pi_0(i_1)^{-1}(x) \to \pi_0(i_2)^{-1}(\pi_0(f)(x)) \]
is again an isomorphism.  But this follows from the continuation of the long exact sequence to an exact sequence of pointed sets:
 \[\xymatrix{
 \pi_1(B,b) \ar[d]_{1} \ar[r] & \pi_0(F_1,e) \ar@{->>}[r]^{\pi_0(i_1)} \ar[d]_{\pi_0 (f_b)} & \pi_0(E_1,x) \ar[d]_{\pi_0 (f)} \\
 \pi_1(B,b) \ar[r] & \pi_0 (F_2,f(e))\ar@{->>}[r]^{\pi_0(i_2)} & \pi_0(E_2,f(e)) }\]
where $e$ is any point of $F_1$ such that $[e] = x$.

Thus $\pi_n(f)$ is an isomorphism for each $n \geq 0$ and basepoint $e \in E_1$; so $f$ is a weak equivalence.
\end{proof}

\begin{definition}
 Take $\intEq_B(E_1, E_2)$ to be the subobject of $\intHom_B (E_1, E_2)$ consisting of all $n$-simplices
 \[ (b \colon \Delta [n] \to B, w \colon b^*E_1 \to b^* E_2)\]
 such that $w$ is a weak equivalence.  (By Lemma~\ref{lemma:weqs_pull_back}, this indeed defines a simplicial subset.)
\end{definition}

From Lemma~\ref{lemma:weq_fibers}, we immediately have:

\begin{corollary}\label{cor:weqs_representable}
 Let $(f, u) \colon X \to \intHom_B (E_1, E_2)$.  Then $u$ is a weak equivalence if and only if $(f, u)$ factors through $\intEq_B (E_1, E_2)$.

 Thus, maps $X \to \intEq_B(E_1,E_2)$ correspond to pairs of maps
 \[(f \colon X \to B, w \colon f^*E_1 \to f^* E_2),\]
 where $w$ is a weak equivalence. \qed
\end{corollary}

While Lemma~\ref{lemma:weq_fibers} was stated just as required by representability, its proof actually gives a slightly stronger statement:

\begin{lemma} \label{lemma:connected_weq}  %% VV Lemma 1.10
 Let $f \colon E_1 \to E_2$ be a morphism over $B$.  If for some vertex $b \colon \Delta [0] \to B$ in each connected component the map of fibers $f_b \colon b^*E_1 \to b^* E_2$ is a weak equivalence, then $f$ is a weak equivalence. \qed
\end{lemma}

\begin{corollary} \label{cor:eq_is_fib}
 The map $\intEq_B(E_1,E_2) \to B$ is a fibration.
\end{corollary}

\begin{proof}
 Suppose we wish to fill a square:
 \[\xymatrix{ \Lambda^k[n] \ar[r] \ar@{^{(}->}[d]^i & \intEq_B(E_1,E_2) \ar[d] \\ % } to match brackets
  \Delta[n] \ar@{.>}[ur] \ar[r]^b & B \\
 }\]
 By the universal property of $\intEq_B(E_1,E_2)$ this corresponds to showing that we can extend a weak equivalence $w \colon i^*b^*E_1 \to i^*b^*E_2$ over $\Lambda^k[n]$ to a weak equivalence $\overline{w} \colon b^*E_1 \to b^*E_2$ over $\Delta[n]$.

 By Lemma~\ref{lemma:HOM_is_fib}, we can certainly find some map $\overline{w}$ extending $w$.  But then since $\Delta[n]$ is connected, Lemma~\ref{lemma:connected_weq} implies that $\overline{w}$ is a weak equivalence.
\end{proof}

While on the subject, we collect a proposition which is not required for the definition of univalence, but which will be useful later:
\begin{proposition} \label{prop:eq-respects-weq}
Suppose that $E_1,E_1',E_2,E_2'$ are fibrations over a common base $B$, and $w_1 \colon E'_1 \to E_1$, $w_2 \colon E_2 \to E'_2$ are weak equivalences over $B$.  Then the induced map $\intEq_B(w_1,w_2) \colon \intEq_B (E_1,E_2) \to \intEq_B (E'_1,E'_2)$ is a weak equivalence.

\[ \begin{tikzpicture}[hole/.style={fill=white,inner sep=1pt},x=1.5cm,y=1.5cm]
\node (E1') at (0,1) {$E'_1$};
\node (E1) at (1,1) {$E_1$};
\node (E2) at (2,1) {$E_2$};
\node (E2') at (3,1) {$E'_2$};
\node (B) at (1.5,0) {$B$};
\draw[->,font=\scriptsize] (E1') to node[hole] {$p'_1$} (B);
\draw[->,font=\scriptsize] (E1) to node[hole] {$p_1$} (B);
\draw[->,font=\scriptsize] (E2) to node[hole] {$p_2$} (B);
\draw[->,font=\scriptsize] (E2') to node[hole] {$p'_2$} (B);
\draw[->,font=\scriptsize,auto] (E1') to node {$w_1$} (E1);
\draw[->,font=\scriptsize,auto] (E2) to node {$w_2$} (E2');
\end{tikzpicture} \]
\end{proposition}

\begin{proof}
As weak equivalences between fibrations, $w_1$ and $w_2$ are fibered homotopy equivalences over $B$.  Choosing fibered homotopy inverses $v_1$, $v_2$ for $w_1$ and $w_2$ respectively gives a homotopy inverse $\intHom_B(v_1,v_2)$ for $\intHom_B(w_1,w_2) \colon \intHom_B (E_1,E_2) \to \intHom_B (E'_1,E'_2)$.  But by Lemma~\ref{lemma:connected_weq}, the image of a homotopy in $\intHom$ whose endpoints lie in $\intEq$ must lie entirely in $\intEq$; so the restriction $\intEq_B(v_1,v_2)$  gives a homotopy inverse for $\intEq_B(w_1,w_2)$, as desired.
\end{proof}

We are now ready to define univalence.

Let $p \colon E \to B$ be a fibration.  We then have two fibrations over $B \times B$, given by pulling back $E$ along the projections.  Call the object of weak equivalences between these $\intEq(E) := \intEq_{B \times B}(\pi_1^*E,\pi_2^*E)$.  Concretely, simplices of $\intEq(E)$ are triples
\[ (b_1,b_2 \in B_n,\, w \colon b_1^*E \to b_2^*E ). \]

By Corollary~\ref{cor:weqs_representable}, a map $f \colon X \to \intEq(E)$ corresponds to a pair of maps $f_1, f_2 \colon X \to B$ together with a weak equivalence $f_1^*E \to f_2^*E$ over $X$.  In particular, there is a “diagonal” map $\delta_E \colon B \to \intEq(E)$ corresponding to the triple $(1_B,1_B,1_E)$, sending a simplex $b \in B_n$ to the triple $(b,b,1_{E_b})$.

There are also source and target maps $s,t \colon \intEq(E) \to B$, given by the composites $\intEq(E) \to B \times B \to^{\pi_i} B$, sending $(b_1,b_2,w)$ to $b_1$ and $b_2$ respectively.  These are both retractions of $\delta$; and by Corollary~\ref{cor:eq_is_fib}, if $B$ is fibrant then they are moreover fibrations.

\begin{definition} \label{def:simplicial-univalence}
 A fibration $p \colon E \to B$ is \emph{univalent} if the diagonal map $\delta_E \colon B \to \intEq (E)$ is a weak equivalence.
\end{definition}

Since $\delta_E$ is always a monomorphism (thanks to its retractions), this is equivalent to saying that $B \to \intEq(E) \to B \times B$ is a (trivial cofibration, fibration) factorisation of the diagonal $\Delta_B \colon B \to B \times B$, i.e.\ that $\intEq(E)$ is a \emph{path object} for $B$.

We conclude this section with a few examples, and non-examples, of univalent fibrations.

\begin{examples} \leavevmode
\begin{enumerate}
  \item The canonical map $X \to 1$ is univalent if and only if the space of homotopy auto-equivalences of $X$ is contractible.
% Remove unless we can find a good reference: In particular, for a group $G$, the map $\B G \to 1$ is univalent just if the center and group of outer automorphisms of $G$ are trivial, as for instance in the case $G = S_n$.
  \item The identity map $X \to X$ is univalent if and only if $X$ is either empty or contractible.  In particular, the identity map $1 + 1 \to 1 + 1$ is \emph{not} univalent: it has two fibers which are equivalent, over points that are not connected by any path.
  \item Any fibration weakly equivalent to a univalent fibration is itself univalent (essentially, by Proposition~\ref{prop:eq-respects-weq}).
\end{enumerate}
\end{examples}

\subsection{Equivalence of type-theoretic and simplicial univalence} \label{subsec:univalence-equivalence}

Having defined the type-theoretic and simplicial notions of univalence, we now wish to show that they coincide.  As ever, we make essential use of representability; in particular, we work with the interpretations of type-theoretic notions entirely via their universal properties.  With this in view, we need to define what are represented by the interpretations of $\synLHInv$, $\synisHIso$, etc.

\begin{definition}
Let $p_1 \colon E_1 \to B$, $p_2 \colon E_2 \to B$ be fibrations over a common base (as in Definition~\ref{def:internal-eq}).

Define $\extHomLHInv_B(E_1,E_2)$ to be the set of \emph{maps with a left homotopy inverse} from $E_1$ to $E_2$, i.e.\ triples $(f,g,H)$, where $f \colon E_1 \to E_2$ and $g \colon E_2 \to E_1$ are maps over $B$, and $H$ is a fibred homotopy from $g \cdot f$ to $1_{E_1}$, defined using the fibred path space $\paths_B(E_1)$ (as used for the $\IdStrux$-structure in the proof of Theorem~\ref{thm:the-model-in-ssets}).

Similarly, define $\extHomRHInv_B(E_1,E_2)$ to consist of triples $(f,g,H)$, where $f, g$ are as before, and $H$ is now a fibred homotopy from $f \cdot g$ to $1_{E_2}$, defined using $\paths_B(E_2)$.

Finally, these both come with evident projections to $\Hom_B(E_1,E_2)$; define $\extHIso_B(E_1,E_2)$ as the pullback $\extHomLHInv_B(E_1,E_2) \times_{\Hom_B(E_1,E_2)} \extHomRHInv_B(E_1,E_2).$
\end{definition}

\begin{lemma} \label{lemma:interpretation_correct}
Let $B,E_1,E_2$ be as above; additionally, suppose they are given by names $\name{B} \colon 1 \to \UU$, $\name{E_i} \colon B \to \UU$.  Then for any $f \colon X \to B$, there are horizontal isomorphisms as in the diagram below, making the diagram commute, and natural in $(X,f)$.

% editing note: width may be easily adjusted by changing defn of x at start.
\[\mathclap{\begin{tikzpicture}[x={(7cm,0cm)},y={(1.8cm,-1.8cm-1em)},z={(1.8cm,1.8cm-1em)}]
  % objects from the type theory:
  \node (tMap) at (0,0,0) {$\Hom_B(X,\interp{ [E_1,E_2] } )$};
  \node (tLInv) at (0,-1,0) {$\Hom_B(X,\interp{ \synHomLHInv(E_1,E_2) } )$};
  \node (tRInv) at (0,0,1) {$\Hom_B(X,\interp{ \synHomRHInv(E_1,E_2) }) $};
  \node (tHEq) at (0,-1,1) {$\Hom_B(X,\interp{ \synHIso(E_1,E_2) }) $};
  \draw[->] (tLInv) to (tMap);
  \draw[->] (tRInv) to (tMap);
  \draw[->] (tHEq) to (tLInv);
  \draw[->] (tHEq) to (tRInv);
  \draw (0,-0.92,0.8) -- (0,-0.84,0.8) -- (0,-0.84,0.9); % pullback sign
 % objects defined directly in sets:
  \node (sMap) at (1,0,0) {$\Hom_X(f^*E_1,f^*E_2)$};
  \node (sLInv) at (1,-1,0) {$\extHomLHInv_X(f^*E_1,f^*E_2)$};
  \node (sRInv) at (1,0,1) {$\extHomRHInv_X(f^*E_1,f^*E_2)$};
  \node (sHEq) at (1,-1,1) {$\extHIso_X(f^*E_1,f^*E_2)$};
  \draw[->] (sLInv) to (sMap);
  \draw[->] (sRInv) to (sMap);
  \draw[->] (sHEq) to (sLInv);
  \draw[->] (sHEq) to (sRInv);
  \draw (1,-0.92,0.8) -- (1,-0.84,0.8) -- (1,-0.84,0.9); % pullback sign
% connecting arrows:
  \draw[->,dotted] (tMap) to node {$\iso$} (sMap);
  \draw[->,dotted] (tLInv) to node {$\iso$}  (sLInv);
  \draw[->,dotted] (tRInv) to node {$\iso$}  (sRInv);
  \draw[->,dotted] (tHEq) to node {$\iso$}  (sHEq);
\end{tikzpicture}}\]

(Here $\interp{-}$ denotes the interpretation of type theory, as described following Corollary~\ref{cor:simplicial-model}; and $[-,-]$ is the ordinary function type, taken as a special case of $\Pi$-types.)
\end{lemma}

\begin{proof} This is essentially a routine verification; we prove just the first case, that of $\interp{ [ {E_1}, {E_2}] }$.  For this, we need to give a natural isomorphism $\Hom_B( X, \interp{ [ {E_1}, {E_2}] } ) \iso \Hom_X(f^*E_1, f^*E_2)$; in other words, to show that $\interp{ [ {E_1}, {E_2}] }$ is the exponential between $E_1$ and $E_2$ in $\sSets/B$.

Recall that by definition, $\interp{ [ {E_1}, {E_2}] }$ is constructed as the pullback of $\UUt$ along the map $\PiStrux \cdot \name{(E_1,E_2)} \colon B \to \UU$:
\[\xymatrix{ 
\interp{ [{E_1}, {E_2} ] } \ar[d] \ar[r] \pb & \toposPi_{A_\gen \shortto B_\gen} B_\gen \ar[d] \ar[r] \pb & \UUt \ar[d] \\
B \ar[r]^-{\name{(E_1,E_2)}} & \UU^{\synPi\text{-}\form} \ar[r]^-\PiStrux & \UU
}\]

$\interp{ [ E_1, E_2] }$ is thus a pullback of the dependent product of the universal pair of fibrations over $\UU^{\synPi\text{-}\form}$, and so by the Beck-Chevalley condition is a dependent product for the pullbacks of these fibrations along $\name{(E_1,E_2)}$.  But these pullbacks are isomorphic to $E_1$, $E_1 \times_B E_2$, by the two pullbacks lemma and the construction of $A_\gen$, $B_\gen$ as pullbacks of $\UUt \to \UU$.

\[\mathclap{\begin{tikzpicture}[x={(6cm,0cm)},y={(-0.8cm,1.5cm+0.7em)},z={(0.6cm,1.5cm-0.7em)}]
  % objects from the type theory:
  \node (B) at (0,0,0) {$B$};
  \node (PiE1E2) at (0,1,0) {$\Pi_{E_1}(E_1^* E_2) $};
  \node (E1) at (0,0,1) {$E_1$};
  \node (E1E2) at (0,0,2) {$E_1^* E_2$};
  \draw[->] (PiE1E2) to (B);
  \draw[->] (E1) to (tMap);
  \draw[->] (E1E2) to (E1);
% objects defined directly in sets:
  \node (UPiForm) at (1,0,0) {$\UU^{\synPi\text{-}\form}$};
  \node (PiGen) at (1,1,0) {$\Pi_{A_\gen} B_\gen$};
  \node (AGen) at (1,0,1) {$A_\gen$};
  \node (BGen) at (1,0,2) {$B_\gen$};
  \draw[->] (PiGen) to (UPiForm);
  \draw[->] (AGen) to (UPiForm);
  \draw[->] (BGen) to (AGen);
% connecting arrows:
  \draw[->] (B) to (UPiForm);
  \draw[->] (PiE1E2) to (PiGen);
  \draw (0.04,0.74,0) -- (0.08,0.74,0) -- (0.08,0.87,0);
  \draw[->] (E1) to (AGen);
  \draw (0.04,0,0.66) -- (0.08,0,0.66) -- (0.08,0,0.83);
  \draw[->] (E1E2) to (BGen);
  \draw (0.04,0,1.66) -- (0.08,0,1.66) -- (0.08,0,1.83);
\end{tikzpicture} } \]

So $\interp{ [ {E_1}, {E_2}] }$ is the dependent product of $E_1 \times_B E_2 \to E_1$ along $E_1 \to B$; but this is exactly the usual construction of exponentials in slices from dependent products \cite[A1.5.2]{johnstone:elephant-i}.
\end{proof}

We also note, from the proof of the preceding lemma:
\begin{corollary}
There is a natural isomorphism over $B$:
\[\interp{ [E_1,E_2] } \iso \intHom_B(E_1,E_2). \extraqedhere \]
\end{corollary}

Following this, we define $\intHIso_B(E_1,E_2) := \interp{ \synHIso_B(E_1,E_2) }$, and $\intHomLHInv$, $\intHomRHInv$ similarly.

\begin{lemma} \label{lemma:hiso-over-eq}
The map $ \intHIso_B(E_1,E_2) \to \intHom_B(E_1,E_2)$ factors through $\intEq_B(E_1,E_2)$; and the resulting map $ \intHIso_B(E_1,E_2) \to \intEq_B(E_1,E_2)$ is a trivial fibration.
\end{lemma}

\begin{proof}
The given map  $ \intHIso_B(E_1,E_2) \to [E_1,E_2] \iso \intHom_B(E_1,E_2)$ corresponds, under the isomorphisms of Lemma~\ref{lemma:interpretation_correct}, to the maps on hom-sets
\begin{equation}
\begin{split}
\Hom_B (X, \intHIso_B (E_1, E_2)) & \iso \extHIso_X (f^*E_1, f^*E_2) \\
& \to \Hom_X (f^*E_1, f^*E_2) \\
& \iso \Hom_B (X, \intHom_B(E_1, E_2)) 
\end{split}
\end{equation}
where the middle map just forgets the chosen homotopy inverses of an h-isomorphism.  But since any map admitting both homotopy inverses is a weak equivalence, the natural map
\[\extHIso_X (f^*E_1, f^*E_2) \to \Hom_X (f^*E_1, f^*E_2)\]
factors through $\extEq_X (f^*E_1, f^*E_2)$; so by Yoneda, $\intHIso_B(E_1,E_2) \to \intHom_B(E_1,E_2)$ factors through $\intEq_B(E_1,E_2)$.
 
Thus, we obtain the desired map $\intHIso_B(E_1, E_2) \to \intEq_B (E_1, E_2)$, corresponding to the forgetful function $\extHIso_X (f^*E_1, f^*E_2) \to \extEq_X (f^*E_1, f^*E_2)$.

Combining this with the left-hand pullback square in Lemma~\ref{lemma:interpretation_correct}, we can consider $\intHIso_B(E_1,E_2)$ as the pullback:
\[\mathclap{\begin{tikzpicture}[x={(-2.4cm,1.6cm)},y={(2.8cm,1.2cm)},z={(0,2.4cm)}]
 \node (Hom) at (0,0,0) {$\intHom_B(E_1,E_2)$};
  \node (HomLInv) at (1,0,0) {$\intHomLHInv_B(E_1,E_2)$};
  \draw[->] (HomLInv) to (Hom);
  \node (HomRInv) at (0,1,0) {$\intHomRHInv_B(E_1,E_2)$};
  \draw[->] (HomRInv) to (Hom);
  \node (Eq) at (0,0,1) {$\intEq_B(E_1,E_2)$};
  \draw[->] (Eq) to (Hom);
  \node (EqLInv) at (1,0,1) {$\intEqLHInv_B(E_1,E_2)$};
  \draw[->] (EqLInv) to (HomLInv);
  \draw[->] (EqLInv) to (Eq);
  \draw (0.9,0,0.8) -- (0.8,0,0.8) -- (0.8,0,0.9); 
  \node (EqRInv) at (0,1,1) {$\intEqRHInv_B(E_1,E_2)$};
  \draw[->] (EqRInv) to (Eq);
  \draw[->] (EqRInv) to (HomRInv);
  \draw (0,0.9,0.8) -- (0,0.8,0.8) -- (0,0.8,0.9); 
  \node (HIso) at (1,1,1) {$\intHIso_B(E_1,E_2)$};
  \draw[->] (HIso) to (EqLInv);
  \draw[->] (HIso) to (EqRInv);
  \draw (0.9,0.8,1) -- (0.8,0.8,1) -- (0.8,0.9,1); 
\end{tikzpicture} } \]
where $\intEqLHInv$, $\intEqRHInv$ are defined by the pullbacks above, and represent weak equivalences equipped with a left (resp.\ right) homotopy inverse.  To show that $\intHIso_B(E_1, E_2) \to$ $\intEq_B (E_1, E_2)$ is a trivial fibration, it thus suffices to show that the maps
 \begin{eqnarray*}
   \intEqLHInv_B (E_1, E_2) & \to & \intEq_B (E_1, E_2) \\
   \intEqRHInv_B (E_1, E_2) & \to & \intEq_B (E_1, E_2)
\end{eqnarray*}
are each trivial fibrations.  We show this in the following two lemmas. 

\begin{lemma} \label{lemma:lhinv-over-eq}
For $B$, $E_1$, $E_2$ as above, the map 
\[\intEqLHInv_B (E_1, E_2) \to \intEq_B (E_1, E_2)\]
is a trivial fibration.  Equivalently, left homotopy inverses to equivalences between fibrant objects extend along cofibrations.
\end{lemma}

\begin{proof}
For $\intEqLHInv_B (E_1, E_2) \to \intEq_B (E_1, E_2)$, we need to find a filler for any diagram of the form
 \[\xymatrix{Y \ar[rr] \ar@{^{(}->}[d]_i & & \intEqLHInv_B (E_1, E_2) \ar[d] \\
 X \ar[rr] \ar@{..>}[rru] & & \intEq_B(E_1, E_2)  }\]
where  $i \colon Y \into X$ is a cofibration.

Writing $f$ for the induced map $X \to B$ and $F_i$ for $f^*E_i$, this square corresponds (by the universal properties of $\intEq$ and $\intEqLHInv$) to a weak equivalence $\bar{w} \colon F_1 \to F_2$, and a fibered left homtopy inverse to $w := i^* \bar{w}$; that is, $l \colon i^* F_2 \to i^* F_1$, and a homotopy $H \colon l \cdot w \homot 1_{i^*F_1}$, all fibered over $Y$:
\[\begin{tikzpicture}[hole/.style={fill=white,inner sep=1pt},x={(1.5cm,0cm)},y={(0cm,1.8cm)},z={(1.8cm,-0.7cm)}]
  % part over Y:
  \node (Y) at (0,0,0) {$Y$};
  \node (iF1) at (0,1,-0.5) {$i^*F_1$};
  \node (iF2) at (0,1,0.5) {$i^*F_2$};
  \draw[->] (iF1) to (Y);
  \draw[->] (iF2) to (Y);
  \draw[->,bend left=12,font=\scriptsize] (iF1) to node[hole] {$w$} (iF2);
  \draw[->,bend left=32,font=\scriptsize] (iF2) to node[hole] {$l$} (iF1);
%  \draw[->,bend left=10,font=\scriptsize,double equal sign distance,-implies,shorten >=1pt,shorten <=3pt]
%      (iF2) to node[hole] {$H$} (iF1); 
  % part over X:
  \node (X) at (2,0,0) {$X$};
  \node (F1) at (2,1,-0.5) {$F_1$};
  \node (F2) at (2,1,0.5) {$F_2$};
  \draw[->] (F1) to (X);
  \draw[->] (F2) to (X);
  \draw[->,bend left=12,auto] (F1) to node {$\scriptstyle \overline{w}$} (F2);
  % connecting arrows:
  \draw[c'->] (Y) to (X);
  \draw[c'->] (iF2) to (F2);
  \draw (0.2,0.7,0.35) -- (0.4,0.7,0.35) -- (0.4,0.85,0.425); % pullback sign
  \draw[c'->] (iF1) to (F1);
%  \draw (0.2,0.8,-0.4) -- (0.4,0.8,-0.4) -- (0.4,0.9,-0.45); % pullback sign
\end{tikzpicture}\]
A filler then corresponds to a fibered left homotopy inverse $(\bar{l},\bar{H})$ to $\bar{w}$, extending $(l,H)$.

These data and desiderata may be summed up in a single commuting diagram:
\[\begin{tikzpicture}[x={(1.8cm,0)},y={(0,-1.2cm)}]
% top left:
\node (F1tl) at (2,0) {$F_1$};
\node (iF1) at (1,0) {$i^* F_1$};
\node (iF1xI) at (1,1) {$i^* F_1 \times \Delta[1]$};
\node (iF1') at (0,1) {$i^* F_1$};
\node (iF2) at (0,2) {$i^*F_2$};
\draw[right hook->,auto] (iF1) to (F1tl);
\draw[->,auto] (iF1) to node {$\scriptstyle \iota_1$} (iF1xI);
\draw[->,auto] (iF1') to node {$\scriptstyle \iota_0$} (iF1xI);
\draw[->,auto] (iF1') to node {$\scriptstyle w$} (iF2);
% top right:
\node (iF1tr) at (4.5,1) {$i^*F_1$};
\node (F1tr) at (6,1) {$F_1$};
\draw[->] (iF1tr) to (F1tr);
% bottom left:
\node (F1xI) at (1.5,5) {$F_1 \times \Delta[1]$};
\node (F1bl) at (0.5,5) {$F_1$};
\node (F2) at (0.5,4) {$F_2$};
\draw[->,auto] (F1bl) to node {$\scriptstyle \bar{w}$} (F2);
\draw[->,auto] (F1bl) to node {$\scriptstyle \iota_0$} (F1xI);
% bottom right;
\node (F1br) at (4.5,5) {$F_1$};
\node (X) at (6,5) {$X$};
\draw[->] (F1br) to (X);
% top row:
\draw[->,auto] (F1tl) to node {$\scriptstyle 1$} (F1tr);
\draw[->,auto] (iF1xI) to node {$\scriptstyle H$} (iF1tr);
\draw[->,auto] (iF2) to node {$\scriptstyle l$} (iF1tr);
% left column:
\draw[right hook->] (iF2) to (F2);
\draw[right hook->] (iF1xI) to (F1xI);
\draw[->,auto] (F1tl) to node {$\scriptstyle \iota_1$} (F1xI);
% bottom row:
\draw[->,auto] (F1xI) to node {$\scriptstyle \pi_1$} (F1br);
\draw[->] (F2) to (X);
% right column:
\draw[->] (F1tr) to (X);
% diagonal:
\draw[->,dotted,auto] (F2) to node {$\scriptstyle \bar{l}$} (F1tr);
\draw[->,dotted,auto] (F1xI) to node {$\scriptstyle  H$} (F1tr);
\end{tikzpicture}\]

Replacing the sub-diagrams on the left by their colimits, we see that we seek precisely a diagonal filler for an associated square:
\[\xymatrix{
 i^*F_2 +_{i^*F_1} (i^*F_1 \times \Delta[1]) +_{i^*F_1} F_1 \ar[rr] \ar[d] & & F_1 \ar[d] \\
 F_2 +_{F_1} (F_1 \times \Delta[1]) \ar@{..>}[urr] \ar[rr] & & X
}\]
So since $F_1 \to X$ is a fibration, we just need to show that the left-hand map of pushouts, induced by
\[\begin{tikzpicture}[x={(1.8cm,0)},y={(0,-1.2cm)}]
% top left:
\node (F1tl) at (2,0) {$F_1$};
\node (iF1) at (1,0) {$i^* F_1$};
\node (iF1xI) at (1,1) {$i^* F_1 \times \Delta[1]$};
\node (iF1') at (0,1) {$i^* F_1$};
\node (iF2) at (0,2) {$i^*F_2$};
\draw[->,auto] (iF1) to (F1tl); % should be “into”
\draw[->,auto] (iF1) to node {$\scriptstyle \iota_1$} (iF1xI);
\draw[->,auto] (iF1') to node {$\scriptstyle \iota_0$} (iF1xI);
\draw[->,auto] (iF1') to node {$\scriptstyle w$} (iF2);
% bottom left:
\node (F1xI) at (2.5,2.5) {$F_1 \times \Delta[1]$};
\node (F1bl) at (1.5,2.5) {$F_1$};
\node (F2) at (1.5,3.5) {$F_2$};
\draw[->,auto] (F1bl) to node {$\scriptstyle \bar{w}$} (F2);
\draw[->,auto] (F1bl) to node {$\scriptstyle \iota_0$} (F1xI);
% left column:
\draw[right hook->] (iF2) to (F2);
\draw[right hook->] (iF1') to (F1bl);
\draw[right hook->] (iF1xI) to (F1xI);
\draw[->,auto] (F1tl) to node {$\scriptstyle \iota_1$} (F1xI);
\end{tikzpicture} \]
is a trivial cofibration.  For convenience, call this map $t$.
 
To see that $t$ is a weak equivalence, consider it in the square
\[\xymatrix{
 (i^*F_1 \times \Delta[1]) +_{i^*F_1} F_1 \ar[r] \ar[d] & F_1 \times \Delta[1] \ar[d] \\
 i^*F_2 +_{i^*F_1} ((i^*F_1 \times \Delta[1] ) +_{i^*F_1} F_1) \ar[r]^-t & F_2 +_{F_1} (F_1 \times \Delta[1]).
}\]
The top map is a trivial cofibration by the pushout-product property; the vertical maps are pushouts of $w$ and $\bar{w}$ along cofibrations, so are also weak equivalences; and so by 2-out-of-3, $t$ is a weak equivalence.

On that other hand, to see that $t$ is a cofibration, consider it as induced by maps $t_0$, $t_1$ as in:
\[\xymatrix{
 i^*F_1 \ar[r] \ar[d] & F_1 \ar[d]^{t_1} \\
 i^*F_2 +_{i^*F_1} (i^*F_1 \times \Delta[1] ) \ar[r]^-{t_0} & F_2 +_{F_1} (F_1 \times \Delta[1]).
}\]
Here $t_0$ is isomorphic to the inclusion 
\[ i^*(F_2 +_{F_1} (F_1 \times \Delta[1])) \into F_2 +_{F_1} (F_1 \times \Delta[1])\]
(since pulling back preserves products and pushouts), so is mono.  Next, $i_0$ and $i_1$ have disjoint images, so $t_1$ is also mono.  Finally, the intersection of the images of $t_0$ and $t_1$ is exactly the image of $i^*F_1$; so $t$, as the induced map from $(i^*F_2 +_{i^*F_1} (i^*F_1 \times \Delta[1] )) +_{i^*F_1} F_1$, is mono as desired.

Thus $t$ is a trivial cofibration, completing the proof of the lemma.
\end{proof}
 
\begin{lemma} \label{lemma:rhinv-over-eq}
For $B$, $E_1$, $E_2$ as above, the map 
\[ \intEqRHInv_B (E_1, E_2) \to \intEq_B (E_1, E_2)\]
is a trivial fibration.  Equivalently, right homotopy inverses to equivalences between fibrant objects extend along cofibrations.
\end{lemma}
 
\begin{proof}
We must provide lifts against any cofibration $i \colon Y \into X$:
 \[\xymatrix{Y \ar[rr] \ar@{^{(}->}[d]_i & & \intEqRHInv_B (E_1, E_2) \ar[d] \\
 X \ar[rr] \ar@{..>}[rru] & & \intEq_B(E_1, E_2) }\]

Analogously to the previous lemma, and again writing $f \colon X \to B$, $F_i := f^* E_1$, the square corresponds to a weak equivalence $\bar{w} \colon F_1 \to F_2$ over $X$ together with a fibered right homotopy inverse to $w := i^* \bar{w}$, i.e.\ $r \colon i^*F_2 \to i^*F_1$ and a homotopy $H \colon w \cdot r \homot 1_{i^*F_2}$ over $Y$;
\[\begin{tikzpicture}[hole/.style={fill=white,inner sep=1pt},x={(1.5cm,0cm)},y={(0cm,1.8cm)},z={(1.8cm,-0.7cm)}]
  % part over Y:
  \node (Y) at (0,0,0) {$Y$};
  \node (iF1) at (0,1,-0.5) {$i^*F_1$};
  \node (iF2) at (0,1,0.5) {$i^*F_2$};
  \draw[->] (iF1) to (Y);
  \draw[->] (iF2) to (Y);
  \draw[->,bend left=12,font=\scriptsize] (iF1) to node[hole] {$w$} (iF2);
  \draw[->,bend left=32,font=\scriptsize] (iF2) to node[hole] {$r$} (iF1);
%  \draw[->,bend right=10,font=\scriptsize,double equal sign distance,-implies,shorten >=1pt,shorten <=3pt]
%      (iF1) to node[hole] {$H$} (iF2); 
  % part over X:
  \node (X) at (2,0,0) {$X$};
  \node (F1) at (2,1,-0.5) {$F_1$};
  \node (F2) at (2,1,0.5) {$F_2$};
  \draw[->] (F1) to (X);
  \draw[->] (F2) to (X);
  \draw[->,bend left=12,auto] (F1) to node {$\scriptstyle \overline{w}$} (F2);
  % connecting arrows:
  \draw[c'->] (Y) to (X);
  \draw[c'->] (iF2) to (F2);
  \draw (0.2,0.7,0.35) -- (0.4,0.7,0.35) -- (0.4,0.85,0.425); % pullback sign
  \draw[c'->] (iF1) to (F1);
%  \draw (0.2,0.8,-0.4) -- (0.4,0.8,-0.4) -- (0.4,0.9,-0.45); % pullback sign
\end{tikzpicture}\]
and a filler corresponds to a fibered right homotopy inverse $(\bar{r},\bar{H})$ for $\bar{w}$, extending $(r,H)$.

Again, putting these conditions together, we see that they correspond to filling another square:
\[\xymatrix{
  i^*F_2 \ar[r]^-{(r,H)} \ar[d] & i^*F_1 \times_{i^*F_2} \paths_Y (i^*F_2) \ar[r] & F_1 \times_{F_2} \paths_X F_2  \ar[d]^{\ev_1 \cdot \pi_1} \\
  F_2 \ar[rr]^1 \ar@{..>}[urr]|{(\bar{r},\bar{H})} & & F_2
}\]
where the pullbacks are just the fibered mapping path spaces.
\[\xymatrix{
  i^*F_1 \times_{i^*F_2} \paths_Y (i^*F_2) \ar[d] \ar[r] \pb & \paths_Y(i^*F_2) \ar[d]^{\ev_0} \\
  i^*F_1 \ar[r]^w  & i^*F_2
} \qquad
\xymatrix{
  F_1 \times_{F_2} \paths_X F_2 \ar[d] \ar[r] \pb & \paths_Y F_2 \ar[d]^{\ev_0} \\
  F_1 \ar[r]^{\bar{w}}  & F_2
}\]

Now $i^*F_2 \into F_2$ is certainly a cofibration; so to provide the filler, it suffices to show that the right-hand map is a trivial fibration.  As the target map from a mapping path space, it is certainly a fibration.  To see that it is a weak equivalence, consider the triangle 
\[ \xymatrix@C=3cm{
  F_1 \ar[r]^-{x\; \mapsto\; (x,c_{\bar{w}x})} \ar[dr]_{\bar{w}} & F_1 \times_{F_2} \paths_X F_2 \ar[d]^{\ev_0} \\
  & F_2 
}\]
The top map is the inclusion of a deformation retraction, so is a weak equivalence; so by 2-out-of-3, the source map $\ev_0$ is a  weak equivalence.  But $\ev_1$ is homotopic to $\ev_0$, so is also a weak equivalence, as required.
\end{proof}

Putting these two lemmas together concludes the proof of Lemma~\ref{lemma:hiso-over-eq}: $\intHIso$ is trivially fibrant over $\intEq$.
\end{proof}

\begin{theorem} \label{thm:tt_vs_simpl_univalence}
Let $B$ be a Kan complex, $p \colon E \to B$ a fibration; choose some names $\name{B} \colon 1 \to \UU$, $\name{E} \colon B \to \UU$ for these.  Then $E$ is simplicially univalent if and only if the type $\synisUnivalent(E)$ is inhabited in the model.
\end{theorem}

\begin{proof}
 By definition, $p \colon E \to B$ is type-theoretically univalent when there exists a section of the type
$ \interp{ x_1, x_2 \oftype B \types \synisHIso (w_{x_1,x_2})\ \type } $
over $B \times B$ (where $w_{x_1,x_2}$ is as in Definition~\ref{def:type-theoretic-univalence}).  By Lemma \ref{lemma:interpretation_correct} this is equivalent to the map 
\[w_E := \interp{ x_1, x_2 \oftype B,\,\synId_B (x_1, x_2) \types w_{x_1,x_2}(p) : \synHIso(E(x_1),E(x_2)) }\] admitting the structure of a homotopy isomorphism, or equivalently being a weak equivalence.
 \[\mathclap{\xymatrix{
  \interp{ x_1, x_2 \oftype B, p \oftype \synId_B (x_1, x_2) } \ar[rr]^{w_E} \ar[rd] & & \interp{ x_1, x_2 \oftype B, f \oftype \synHIso (E(x_1), E(x_2)) } \ar[ld] \\ 
   & B \times B & }}\]

By Lemma~\ref{lemma:interpretation_correct}, we may fit $w_E$ into the following diagram.
 \[\xymatrix{B \ar[r]^-{r_B} \ar@/_3em/[dddrrr]_{\Delta_B} & \P(B) \ar[rr]^-{w_E} \ar@/_2em/[dddrr] & & \intHIso_{B\times B}(\pi_1^*E,\pi_2^*E) \ar[d] \\
& & & \intEq_{B\times B}(\pi_1^*E,\pi_2^*E) \ar[d] \\
& & & \intHom_{B \times B} (\pi_1^*E,\pi_2^*E) \ar[d] \\
& & & B \times B}\]
Then by the $\synId$-\comp\ rule applied to the definition of $w_{x_1,x_2}$, the overall composite map $B \to \intHom_{B \times B} (\pi_1^*E,\pi_2^*E) $ is the interpretation of $\interp{x \oftype B \types \lambda y\oftype E(x).\,y : [E(x),E(x)] }$, which corresponds under the universal property of $\intHom$ to $(\Delta_B,1_E)$.  So the composite map $B \to \intEq_{B\times B}(\pi_1^*E,\pi_2^*E)$ is exactly $\delta_E$ of Definition~\ref{def:simplicial-univalence}. But by definition, $E$ is univalent precisely if $\delta_E$ is a weak equivalence; and by 2-out-of-3 and Lemma~\ref{lemma:hiso-over-eq}, $\delta_E$ is a weak equivalence if and only if $w_E$ is. So we are done.
\end{proof}

\subsection{Univalence of the simplicial universes} \label{subsec:univalence-of-uu}

\begin{theorem} \label{thm:univalence}
 The fibration ${p_\alpha} \colon \UUt \to \UU$ is univalent.
\end{theorem}

\begin{proof}
We will show that the target map $t \colon \intEq(\UUt) \to \UU$ is a trivial fibration.  Since $t$ is a retraction of $\delta_{\UUt}$, this implies by 2-out-of-3 that $\delta_{\UUt}$ is a weak equivalence.

So, we need to fill a square
 \[\xymatrix{
  A \ar[r] \ar@{^{(}->}[d]_i & \intEq(\UUt) \ar[d]^{t} \\
  B \ar[r] \ar@{.>}[ur]     & \UU
 }\]
where $i \colon A\ \to/^{(}->/ B$ is a cofibration.

By the universal properties of $\UU$ and $\intEq(\UUt)$, these data correspond to a weak equivalence $w \colon E_1 \to E_2$ between $\alpha$-small well-ordered fibrations over $A$, and an extension $\Ebar_2$ of $E_2$ to an $\alpha$-small, well-ordered fibration over $B$; and a filler corresponds to an extension $\Ebar_1$ of $E_1$, together with a weak equivalence $\overline{w}$ extending $w$:
\[\begin{tikzpicture}[x={(1.5cm,0cm)},y={(0cm,1.8cm)},z={(1.8cm,-0.7cm)}]
  % part over A:
  \node (A) at (0,0,0) {$A$};
  \node (E1) at (0,1,-0.5) {$E_1$};
  \node (E2) at (0,1,0.5) {$E_2$};
  \draw[->] (E1) to (A);
  \draw[->] (E2) to (A);
  \draw[->,auto] (E1) to node {$\scriptstyle w$} (E2);
  % part over B:
  \node (B) at (2,0,0) {$B$};
  \node (Eb1) at (2,1,-0.5) {$\Ebar_1$};
  \node (Eb2) at (2,1,0.5) {$\Ebar_2$};
  \draw[->,dashed] (Eb1) to (B);
  \draw[->] (Eb2) to (B);
  \draw[->,dashed,auto] (Eb1) to node {$\scriptstyle \overline{w}$} (Eb2);
  % connecting arrows:
  \draw[->] (A) to (B);
  \draw[->] (E2) to (Eb2);
  \draw (0.2,0.7,0.35) -- (0.4,0.7,0.35) -- (0.4,0.85,0.425); % pullback sign
  \draw[->,dashed] (E1) to (Eb1);
  \draw[dashed] (0.2,0.8,-0.4) -- (0.4,0.8,-0.4) -- (0.4,0.9,-0.45); % pullback sign
\end{tikzpicture}\]

As usual, it is sufficient to construct this first without well-orderings on $\Ebar_1$; these can then always be chosen so as to extend those of $E_1$. \\

Recalling Lemmas~\ref{lemma:exp_along_cofib}--\ref{lemma:joyal-lemma}, we define $\Ebar_1$ and $\overline{w}$ as the pullback
 \[\xymatrix{
  \Ebar_1 \ar[d]_{\overline{w}} \ar[r] \pb & \toposPi_i E_1 \ar[d]^{\toposPi_i w} \\
  \Ebar_2 \ar[r]_{\eta}                    & \toposPi_i E_2
 }\]
in $\sSets/B$, where $\eta$ is the unit of $i^* \adjoint \toposPi_i$ at $\Ebar_2$.  To see that this construction works, it remains to show:
\begin{enumerate}[(a)]
\item $i^*\Ebar_1 \iso E_1$ in $\sSets/A$, and under this, $i^* \overline{w}$ corrsponds to $w$;
\item $\Ebar_1$ is $\alpha$-small over $B$;
\item $\Ebar_1$ is a fibration over $B$, and $\overline{w}$ is a weak equivalence.
\end{enumerate}

For (a), pull the defining diagram of $\Ebar_1$ back to $\sSets/A$; by Lemma~\ref{lemma:exp_along_cofib} part 2, we get a pullback square
\[\xymatrix{
  i^*\Ebar_1 \ar[d]_{i^*\overline{w}} \ar[r] \pb & E_1 \ar[d]^w \\
  E_2 \ar[r]^{1_{E_2}}                               & E_2
}\]
in $\sSets/A$, giving the desired isomorphism.

For (b), Lemma~\ref{lemma:exp_along_cofib} part 3 gives that $\toposPi_i E_1$ is $\alpha$-small over $B$, so $\Ebar_1$ is a subobject of a pullback of $\alpha$-small maps.

For (c), note first that by factoring $w$, we may reduce to the cases where it is either a trivial fibration or a trivial cofibration.

In the former case, by Lemma~\ref{lemma:exp_along_cofib} part 1 $\toposPi_i w$ is also a trivial fibration, and hence so is $\overline{w}$; so $\Ebar_1$ is fibrant over $\Ebar_2$, hence over $B$.

In the latter case, $E_1$ is then a deformation retract of $E_2$ over $A$; we will show that $\Ebar_1$ is also a deformation retract of $\Ebar_2$ over $B$.  Let $H \colon E_2 \times \Delta[1] \to E_2$ be a deformation retraction of $E_2$ onto $E_1$.  We want some homotopy $\Hbar \colon \Ebar_2 \times \Delta[1] \to \Ebar_2$ extending $H$ on $E_2 \times \Delta[1]$, $1_{\Ebar_1} \times \Delta[1]$ on $\Ebar_1 \times \Delta[1]$, and $1_{\Ebar_2}$ on $\Ebar_2 \times \{0\}$.  Since these three maps agree on the intersections of their domains, this is exactly an instance of the homotopy lifting extension property, i.e.\ a square-filler
\[\xymatrix{
  (E_2 \times \Delta[1]) \cup (\Ebar_1 \times \Delta[1]) \cup (\Ebar_2 \times \{0\})
     \ar@{^{(}->}[d] \ar[rr]^<>(0.4){H \cup 1 \cup 1}   % extra close-bracket to unconfuse emacs parser: )
     & & \Ebar_2 \ar[d] \\
  \Ebar_2 \times \Delta[1] \ar[rr] \ar@{.>}[urr]|{\Hbar} & & B
} \qquad \quad \]
which exists since the left-hand map is a trivial cofibration.

For $\Hbar$ to be a deformation retraction, we need to see that $\Hbar_{\{1\}} \colon \Ebar_2 \to \Ebar_2$ factors through $\Ebar_1$.  By the definition of $\Ebar_1$, a map $f \colon X \to \Ebar_2$ over $b \colon X \to B$ factors through $\Ebar_1$ just if the pullback $i^*f \colon i^*X \to E_2$ factors through $E_1$.  In the case of  $\Hbar_{\{1\}}$, the pullback is by construction $i^*(\Hbar_{\{1\}}) = (i^*\Hbar)_{\{1\}} = H_{\{1\}} \colon E_2 \to E_2$, which factors through $E_1$ since $H$ was a deformation retraction onto $E_1$.

So $\overline{w}$ embeds $\Ebar_1$ as a deformation retract of $\Ebar_2$ over $B$; thus $\Ebar_1$ is a fibration over $B$ and $\overline{w}$ a weak equivalence, as desired.
\end{proof}

Putting this together with Corollary~\ref{cor:simplicial-model}, we obtain our main theorem:
 
\begin{theorem} \label{thm:simplicial-model-univalent}
Let $\beta < \alpha$ be inaccessible cardinals.  Then there is a model of dependent type theory in $\sSets_{\UU}$ with all the logical constructors of Section~\ref{subsec:logical-rules}, and a universe (given by $\UU[\beta]$) closed under these constructors and satisfying the Univalence Axiom. \qed
\end{theorem}

From this, we can immediately deduce:

\begin{theorem} \label{thm:uf-consistent}
  Assuming the existence of two inaccessible cardinals, the contextual category presentation of $\mathsf{MLTT}+\mathsf{UA}$ (as given in Definition~\ref{def:uf-and-models}) is consistent. \qed
\end{theorem}

In practice one often considers a type theory with a sequence of $n$ or $\omega$ univalent universes. We expect that the techniques used in the proof of Theorem \ref{thm:uf-consistent} can be adapted to yield a consistency proof for such a theory, relative to set theory with suitable many inaccessible cardinals; but we do not pursue that here.

\begin{remark}
One can prove, within the type theory, that the Univalence Axiom together with the $\synPi$-$\eta$ rule implies functional extensionality; see \cite{voevodsky:univalent-foundations-coq}, \cite[Sec.~4.9]{hott:book}.  So we could have omitted functional extensionality from Proposition~\ref{prop:eta-and-funext}, and instead deduced it here as a corollary of univalence.
\end{remark} 

\subsection{Univalence and pullback representations} \label{subsec:pullback-reps}

We are now ready to give a uniqueness statement for the representation of an $\alpha$-small fibration as a pullback of ${p_\alpha} \colon \UUt \to \UU$: we define the space of such representations, and show that it is contractible. 

In fact, we work a bit more generally.  Given fibrations $q$, $p$, we define a space $\PP_{q,p}$ of representations of $q$ as a pullback of $p$; and we show that a fibration $p$ over a Kan base is univalent exactly when for every $q$, $\PP_{q,p}$ is either empty or contractible.

Let $p \colon E \to B$ and $q \colon Y \to X$ be fibrations.  We define a functor
\[ \P_{q,p} \colon \sSets^{\op} \to \Sets, \]
setting $\P_{q,p} (S)$ to be the set of pairs of a map $f \colon S \times X \to B$, and a weak equivalence $w \colon S \times Y \to f^*E$ over $S \times X$; equivalently, the set of squares
\[\xymatrix{
 S \times Y \ar[r]^{f'} \ar[d]_{S \times q} & E \ar[d]^{p} \\
 S \times X \ar[r]_{f}                     & B
}\]
such that the induced map $S \times Y \to f^*E$ is a weak equivalence.  Lemma~\ref{lemma:weqs_pull_back} ensures that this is functorial in $S$, by pullback.

\begin{lemma}
The functor $\P_{q,p}$ is representable, represented by the object 
\[ \PP_{q,p} := \toposPi_{X \shortto 1} \toposSigma_{\pi_1} \intEq_{X \times B} (\pi_1^* Y, \pi_2^* E).\]
\[ \begin{tikzpicture}[x=2cm,y=1.2cm]
\node (1) at (0,0) {$1$};
\node (X) at (0,1) {$X$};
\node (B) at (1,1) {$B$};
\node (Y) at (-0.5,2) {$Y$};
\node (E) at (1.5,2) {$E$};
\node (XxB) at (0.5,2) {$X \times B$};
\node (Y') at (0,3) {$\pi_1^* Y$};
\node (E') at (1,3) {$\pi_2^* E$};
\draw[->] (X) to (1);
\draw[->,auto,swap] (Y) to node {$\scriptstyle q$} (X);
\draw[->,auto,swap] (XxB) to node {$\scriptstyle \pi_1$} (X);
\draw[->,auto] (XxB) to node {$\scriptstyle \pi_2$} (B);
\draw[->,auto] (E) to node {$\scriptstyle p$} (B);
\draw[->] (Y') to (Y);
\draw[->] (Y') to (XxB);
\draw[->] (E') to (XxB);
\draw[->] (E') to (E);
\end{tikzpicture} \]
\end{lemma}

\begin{proof}
For any $S$, we have:
\begin{equation*}
\begin{split}
  \Hom &(S,\toposPi_{X \shortto 1} \toposSigma_{\pi_1} \intEq_{X \times B} (\pi_1^* Y, \pi_2^* E)) \\
    & \iso \Hom_X (X \times S, \toposSigma_{\pi_1} \intEq_{X \times B} (\pi_1^* Y, \pi_2^* E)) \\
    & \iso \{ (\hat{f},\hat{w})\ |\ \hat{f} \colon X \times S \to X \times B\ \text{over}\ X, \\
    & \qquad \qquad \quad \  \hat{w} \colon X \times S \to \intEq_{X \times B} (\pi_1^* Y, \pi_2^* E)\ \text{over}\ X \times B \} \\
    & \iso \{ (f,w)\ |\ f \colon X \times S \to B,\ w \colon Y \times S \to f^* E\ \text{w.e.\ over}\ X \times S \} \\
    & \iso \P_{q,p}(S) \qedhere
\end{split}
\end{equation*}
\end{proof}

\begin{remark}
By Yoneda, we see from this that $(\PP_{q,p})_n \iso \P_{q,p}(\Delta[n])$.
\end{remark}

\begin{theorem} \label{thm:univalent-characterization}
Let $p \colon E \to B$ be a fibration, with $B$ Kan.  Then $p$ is univalent if and only if for every fibration $q \colon Y \to X$, $\PP_{q,p}$ is either empty or contractible.\footnote{Constructively-minded readers might prefer to phrase this as: if $\PP_{q,p}$ is inhabited, then it is contractible.  In the language of \cite{hott:book}, it is a \emph{mere proposition}.}
\end{theorem}

\begin{proof}
First, suppose that $p$ is univalent.  Take any $q$ such that $\PP_{q,p}$ is non-empty; then we have some map $1 \to \PP_{q,p}$, corresponding to a square
\[ \begin{tikzpicture}[x=2cm,y=2cm]
\node (X) at (0,0) {$X$};
\node (B) at (1,0) {$B$};
\node (Y) at (-0.4,1) {$Y$};
\node (f*E) at (0.3,1) {$f^*E$};
\node (E) at (1.3,1) {$E$};
\draw[->,auto] (X) to node {$\scriptstyle f$} (B);
\draw[->,auto,swap] (Y) to node {$\scriptstyle q$} (X);
\draw[->,auto] (f*E) to node {$\scriptstyle f^*p$} (X);
\draw[->,auto] (E) to node {$\scriptstyle p$} (B);
\draw[->,auto] (Y) to node {$\scriptstyle w$} (f*E);
\draw[->] (f*E) to (E);
\end{tikzpicture} \]

We claim that $\PP_{q,p} \to 1$ is a trivial fibration, and hence $\PP_{q,p}$ is contractible.  $\toposPi$-functors preserve trivial fibrations (since their left adjoints, pullback, preserve cofibrations), so it is enough to show that 
\[\intEq_{X \times B} (\pi_1^* Y, \pi_2^* E) \to X \times B \to^{\pi_1} X\]
is a trivial fibration.

For this, first note that $w$, as a weak equivalence between fibrations, is a homotopy equivalence over $X$, so induces a homotopy equivalence 
\[(w\cdot -) \colon \intEq_{X \times B} ((\pi_1^* (f^* E), \pi_2^* E) \to \intEq_{X \times B} (\pi_1^* Y, \pi_2^* E).\]
 So it is enough to show that $\intEq_{X \times B} ((\pi_1^* (f^* E), \pi_2^* E) \to X \times B \to^{\pi_1} X$ is a trivial fibration; but this follows since it is the pullback along $f$ of the “source” map $\intEq(E) = \intEq_{B \times B}(\pi_1^*E,\pi_2^*E) \to B \times B \to^{\pi_1} B$, which is a trivial fibration since $p$ is univalent and $B$ is Kan.

Conversely, suppose that for every fibration $q$, $\PP_{q,p}$ is either empty or contractible; now, we wish to show $p$ univalent.  For this, it is enough to show that the source map $s \colon \intEq(E) \to B$ is a trivial fibration, which will hold if each of its fibers is contractible.

So, take some $f \colon 1 \to B$, and consider the fiber $f^*\intEq(E)$.  By the universal property of $\intEq(E)$, this is isomorphic to $\PP_{f^*p,p}$; and it is certainly non-empty, containing the pair $(f, 1_{f^*E})$; so by assumption, it is contractible, as desired.
\[ \begin{gathered}[b]
  \xymatrix{
    f^* \intEq(E) \ar[r] \ar[d] \pb & \intEq(E) \ar[d]^s \\
    1 \ar[r]^f & B
  } \\[-\dp\strutbox]
\end{gathered} \qedhere
\]
\end{proof}

\begin{corollary}
For any $\alpha$-small fibration $q$, the simplicial set $\PP_{q, p_\alpha}$ of representations of $q$ as a pullback of $p_\alpha$ is contractible. \qed
\end{corollary}

%%% Local Variables:
%%% mode: latex
%%% TeX-master: "simplicial-model"
%%% End: 

\appendix

\addtocontents{toc}{\protect\setcounter{tocdepth}{1}}

% Note: standard ordering of rules we currently follow throughout is:
% Pi; Sigma; Id; W; 0; 1; +; universes.

\section{Syntax of Martin-L\"of Type Theory} \label{app:type-theory}

A full introduction to Martin-Löf type theory is beyond the scope of this paper.
For the reader new to type theory, we recommend \cite{martin-lof:bibliopolis} or \cite{n-p-s:programming} as a general introduction, and \cite{hofmann:syntax-and-semantics} for a more detailed presentation of the syntax.

However, there are many variant presentations of the theory; in this appendix, we lay out the one we have in mind for the present paper, and which is intended to correspond to the contextual categories described in Appendix~\ref{app:cxl-structure}.

We consider the syntax as constructed in two stages: first the \emph{raw} or \emph{untyped} syntax of the theory---the set of expressions that are at least parseable, but not necessarily meaningful---and then the \emph{derivable judgements}, certain inductively-generated predicates picking out the genuinely meaningful contexts, types, and terms.

The raw syntax may be constructed as certain strings of symbols, or alternatively, certain labelled trees. On this, one then defines \emph{$alpha$-equivalence} (i.e.\ syntactic identity modulo renaming of bound variables), and the operation of \emph{(capture-free) substitution}.  This step is well-standardised in the literature.

At the second stage, one defines on the raw syntax several multi-place relations, picking out the \emph{derivable judgements} of the theory.  For instance, “$\Gamma \types a : A$” will be a relation on triples $(\Gamma, a, A)$ of a raw context, term, and type expression respectively, to be read as “$a$ is a term of type $A$, in context $\Gamma$”.  These relations are defined by mutual induction, as the smallest family of relations closed under a bevy of specified closure conditions, the \emph{inference rules} of the theory.

Details of the judgements and inference rules used vary somewhat; we therefore set our choice out here in full.  For the structural rules, our presentation is based largely on \cite{hofmann:syntax-and-semantics}; our selection of logical rules, and in particular our treatment of the universe, follows \cite{martin-lof:bibliopolis}.

We take as basic four judgement forms:
\begin{mathpar}
  \Gamma \types A\ \type
\and
  \Gamma \types A = A'\ \type
\and
  \Gamma \types a : A
\and
  \Gamma \types a = a' : A.
\end{mathpar}
We take the context judgement as defined from these: that is, if $\Gamma$ is a list $(x_i \oftype A_i)_{i < n}$, with $x_i$ distinct variables and $A_i$ raw type expressions, then $\types \Gamma\ \cxt$ is an abbreviation for the statement that for each $i < n$, $(x_j \oftype A_j)_{j < i} \types A_i\ \type$.

\emph{Derivability}, for these four basic judgements, is then defined as the smallest family of relations closed under the closure conditions specified by the inference rules below.
These are given in “rule-notation”, so for instance
\[
  \inferrule*{\Gamma \types a : A \\ \Gamma \types A = B \ \type}{\Gamma \types a : B}
\]
expresses the closure condition “for all suitable raw expressions $\Gamma$, $a$, $A$, $B$, if the judgements $\Gamma \types a : A$ and $\Gamma \types A = B$ are derivable, then so is $\Gamma \types a : B$”.

The inference rules fall into two groups: the \emph{structural rules}, which we assume are always included, and the \emph{logical rules}, which different type theories may include different subsets of.

\subsection{Structural Rules} \label{subsec:structural-rules}
The structural rules of the type theory are (where $\mathcal{J}$ may be the conclusion of any of the judgement forms):

\begin{mathpar}
  \inferrule*[right=$\Vble$]{\types \Gamma,\ x \oftype A,\ \Delta\ \cxt}{\Gamma,\ x \oftype A,\ \Delta \types x : A}
\and
  \inferrule*[right=$\Subst$]{\Gamma \types a : A \\ \Gamma,\ x \oftype A,\ \Delta \types \mathcal{J}}{\Gamma,\ \Delta[a/x] \types \mathcal{J}[a/x]}
\and
  \inferrule*[right=$\Weak$]{\Gamma \types A\ \type \\ \Gamma,\ \Delta \types \mathcal{J}}{\Gamma,\ x \oftype A,\ \Delta \types \mathcal{J}}
\end{mathpar}

Definitional equality (also known as syntactic or judgemental equality):
\begin{mathparpagebreakable}
  \inferrule*{\Gamma \types A\ \type}{\Gamma \types A = A\ \type} 
\and
  \inferrule*{\Gamma \types A=B\ \type}{\Gamma \types B = A\ \type}
\and
  \inferrule*{\Gamma \types A=B\ \type \\ \Gamma \types B=C \ \type}{\Gamma \types A = C\ \type}
\and
  \inferrule*{\Gamma \types a : A}{\Gamma \types a = a : A}
\and
  \inferrule*{\Gamma \types a=b : A}{\Gamma \types b=a : A}
\and
  \inferrule*{\Gamma \types a=b : A \\ \Gamma \types b=c : A}{\Gamma \types a=c : A}
\and
  \inferrule*{\Gamma \types a : A \\ \Gamma \types A = B \ \type}{\Gamma \types a : B}
\and
  \inferrule*{\Gamma \types a = b : A \\ \Gamma \types A = B \ \type}{\Gamma \types a = b : B}
\end{mathparpagebreakable}

\subsection{Logical Constructors} \label{subsec:logical-rules}

In this and subsequent sections, we present rules introducing various type- and term-constructors.  For each such constructor, we assume (besides the explicitly stated rules introducing and governing it) a \emph{congruence} rule stating that it preserves definitional equality in each of its arguments; for instance, along with the $\synPi$-\intro\ rule introducing the constructor $\lambda$, we assume the rule
\[\inferrule*[right=$\lambda$-eq]{\Gamma \types A = A'\ \type \\ \Gamma,\ x \oftype A \types B(x) = B'(x)\ \type \\ \Gamma,\ x \oftype A \types b(x) = b'(x) : B(x)}{\Gamma \types \lambda x \oftype A.b(x) = \lambda x \oftype A'.b'(x) : \synPi_{x \oftype A} B(x)}\]

The rules fall naturally into groups according to the various logical constructors.  Many of the constructors considered ($\synSigma$-, $\synId$-, $\synW$-, $\synZero$-, $\synOne$-, and $+$-types) follow a common pattern, as \emph{inductive types/families} with constructors and an elimination principle; here, as in \cite{martin-lof:bibliopolis}, this pattern is purely heuristic, but in approaches such as the Calculus of Inductive Constructions \cite{werner:thesis} they are unified formally as instances of a single scheme.
 
\paragraph{$\synPi$-types} (Dependent products; dependent function types).

\begin{mathparpagebreakable}
  \inferrule*[right=$\synPi$-\form]{\Gamma,\ x \oftype A \types B(x)\ \type}{\Gamma \types \synPi_{x \oftype A} B(x)\ \type}
\and
  \inferrule*[right=$\synPi$-\intro]{\Gamma,\ x \oftype A \types B(x)\ \type \\ \Gamma,\ x \oftype A \types b(x) : B(x)}{\Gamma \types \lambda x \oftype A.b(x) : \synPi_{x \oftype A} B(x)}
\and
  \inferrule*[right=$\synPi$-\appRule]{\Gamma \types f \oftype \synPi_{x \oftype A} B(x) \\ \Gamma \types a : A}{\Gamma \types \app (f, a) : B(a)}
\and
  \inferrule*[right=$\synPi$-\comp]{\Gamma,\ x \oftype A \types B(x)\ \type \\ \Gamma,\ x \oftype A \types b(x) : B(x) \\ \Gamma \types a : A}{\Gamma \types \app(\lambda x \oftype A .b(x), a)=b(a) : B(a)}
\end{mathparpagebreakable}

As a special case of this, when $B$ does not depend on $x$, we obtain the ordinary function type $[A, B] := \synPi_{x \oftype A} B$. \\

\paragraph*{$\synSigma$-types}  (Dependent sums; type-indexed disjoint sums.)

\begin{mathparpagebreakable}
  \inferrule*[right=$\synSigma$-\form]{\Gamma \types A\ \type \\ \Gamma,\ x \oftype A \types B(x)\ \type}{\Gamma \types \synSigma_{x \oftype A} B(x)\ \type}
\and
  \inferrule*[right=$\synSigma$-\intro]{\Gamma \types A\ \type \\ \Gamma,\ x \oftype A \types B(x)\ \type}{\Gamma,\ x \oftype A,\ y \oftype B(x) \types \pair (x, y) : \synSigma_{x \oftype A} B(x)}
\and
  \inferrule*[right=$\synSigma$-\elim]{\Gamma,\ z \oftype \synSigma_{x \oftype A} B(x) \types C(z)\ \type \\ \Gamma,\ x \oftype A,\ y \oftype B(x) \types d(x,y) : C(\pair(x, y))}{\Gamma,\ z \oftype \synSigma_{x \oftype A} B(x) \types \synsplit_d (z) : C(z)}
\and
  \inferrule*[right=$\synSigma$-\comp]{\Gamma,\ z \oftype \synSigma_{x \oftype A} B(x) \types C(z)\ \type \\ \Gamma,\ x \oftype A,\ y \oftype B(x) \types d(x,y) : C(\pair(x, y))}{\Gamma,\ x \oftype A,\ y \oftype B(x) \types \synsplit_d (\pair (x,y))=d(x,y) : C(\pair (x, y))}
\end{mathparpagebreakable}

Again,  the special case where $B$ does not depend on $x$ is of particular interest: this gives the cartesian product $A \times B := \synSigma_{x \oftype A} B$. \\

\paragraph*{$\synId$-types.}  (Identity types, equality types.)

\begin{mathparpagebreakable}
  \inferrule*[right=$\synId$-\form]{\Gamma \types A\ \type}{\Gamma,\ x,y\oftype A \types \synId_A(x,y)\ \type}
\and
  \inferrule*[right=$\synId$-\intro]{\Gamma \types A\ \type}{\Gamma,\ x\oftype A \types \refl_A(x):\synId_A(x,x)}
\and
  \inferrule*[right=$\synId$-\elim]{\Gamma,\ x,y\oftype A,\ u\oftype \synId_A(x,y) \types C(x,y,u)\ \type \\ \Gamma,\ z\oftype A \types d(z):C(z,z,\refl_A(z))}{\Gamma,\ x,y\oftype A,\ u\oftype \synId_A(x,y) \types \synJ_{z. d}(x,y,u) : C(x,y,u)}
\and
  \inferrule*[right=$\synId$-\comp]{\Gamma,\ x,y\oftype A,\ u\oftype \synId_A(x,y) \types C(x,y,u)\ \type \\ \Gamma,\ z\oftype A \types d(z):C(z,z,r(z))}{\Gamma,\ x\oftype A \types \synJ_{z.d}(x,x,\refl_A(x)) = d(x) : C(x,x,\refl_A(x))}
\end{mathparpagebreakable}

\paragraph*{$\synW$-types.}  (Types of well-founded trees; free term algebras.)

\begin{mathparpagebreakable}
  \inferrule*[right=$\synW$-\form]{\Gamma,\ x \oftype A \types B(x)\ \type}{\Gamma \types \synW_{x \oftype A} B(x)\ \type}
\and
  \inferrule*[right=$\synW$-\intro]{\Gamma,\ x \oftype A \types B(x)\ \type}{\Gamma, \ x \oftype A, \ y \oftype [B(x), \synW_{u \oftype A} B(u)] \types \synsup(x, y) : \synW_{u \oftype A} B(u)}
\\
  \mathclap{\inferrule*[right=$\synW$-\elim]
    {\Gamma, \ w \oftype \synW_{x \oftype A} B(x) \types C(w) \ \type \\
     \Gamma,\ x \oftype A,\ y \oftype [B(x), \synW_{u \oftype A} B(u)],\ z \oftype \synPi_{u \oftype B(x)} C(\app(y, u)) \hspace{1.55cm} \\
     \hspace{6cm} \types d(x, y, z) : C(\synsup(x, y))}
  {\Gamma,\ w \oftype \synW_{x \oftype A} B(x) \types \wrec_{d} (w) : C(w)}}
\\
  \mathclap{\inferrule*[right=$\synW$-\comp]
  {\Gamma, \ w \oftype \synW_{x \oftype A} B(x) \types C(w) \ \type \\
   \Gamma,\ x \oftype A,\ y \oftype [B(x), \synW_{u \oftype A} B(u)],\ z \oftype \synPi_{u \oftype B(x)} C(\app(y, u)) \hspace{2cm} \\
    \hspace{6.45cm} \types d(x, y, z) : C(\synsup(x, y))}
  {\Gamma, \ x \oftype A, \ y \oftype [B(x), \synW_{u \oftype A} B(u)] \types \wrec_{d} (\synsup(x, y)) \hspace{2.8cm} \\
    \hspace{2.4cm} = d(x, y, \lambda u \oftype B(x). \wrec_d(\app(y, u))): C(\synsup(x, y))}}
\end{mathparpagebreakable}

\paragraph*{$\synZero$}  (Empty type.)

\begin{mathparpagebreakable}
  \inferrule*[right=$\synZero$-\form]{\types \Gamma\ \cxt}{\Gamma \types \synZero\ \type}
\and
  \text{(No introduction rules.)}
\\
  \inferrule*[right=$\synZero$-\elim]{\Gamma, \ x \oftype \synZero \types C(x)\ \type}{\Gamma, \ x \oftype \synZero \ \types \syncase (x) : C(x)}
\and
  \text{(No computation rules.)}
\end{mathparpagebreakable}

\paragraph*{$\synOne$}  (Unit type, singleton type.)

\begin{mathparpagebreakable}
  \inferrule*[right=$\synOne$-\form]{\types \Gamma\ \cxt }{\Gamma \types \synOne \ \type}
\and
  \inferrule*[right=$\synOne$-\intro]{\types \Gamma\ \cxt }{\Gamma \types * : \synOne}
\and
  \inferrule*[right=$\synOne$-\elim]{\Gamma,\ x \oftype \synOne \types C(x)\ \type \\ \Gamma \types d : C(*)}{\Gamma,\ x \oftype \synOne \types \synrec_d (x) : C(x)}
\and
  \inferrule*[right=$\synOne$-\comp]{\Gamma,\ x \oftype \synOne \types C(x)\ \type \\ \Gamma \types d : C(*)}{\Gamma \types \synrec_d (*) = d : C(*)}
\end{mathparpagebreakable}
% \todo{Improve naming conventions of eliminators?}

\paragraph*{$+$-types} (Binary disjoint sums.)

\begin{mathparpagebreakable}
  \inferrule*[right=$+$-\form]{\Gamma \types A\ \type \\ \Gamma \types B\ \type}{\Gamma \types A + B \ \type}
\and
  \inferrule*[right=$+$-\intro\ 1.]{\Gamma \types A\ \type \\ \Gamma \types B\ \type}{\Gamma,\ x \oftype A \types \inl(x) : A+ B}
\and
  \inferrule*[right=$+$-\intro\ 2.]{\Gamma \types A\ \type \\ \Gamma \types B\ \type}{\Gamma,\ y \oftype B \types \inr(y) : A +B}
\and
  \inferrule*[right=$+$-\elim]{\Gamma,\ z \oftype A + B \types C(z)\ \type \\ \Gamma,\ x \oftype A \types d_l(x) : C(\inl (x)) \\ \Gamma,\ y \oftype B \types d_r(y) : C(\inr (y))}{\Gamma,\ z \oftype A + B \types \syncase_{d_l,d_r}(z) : C(z)}
\and
  \inferrule*[right=$+$-\comp\ 1.]{\Gamma,\ z \oftype A + B \types C(z)\ \type \\ \Gamma,\ x \oftype A \types d_l(x) : C(\inl (x)) \\ \Gamma,\ y \oftype B \types d_r(y) : C(\inr (y))}{\Gamma,\ x \oftype A \types \syncase_{d_l, d_r}(\inl(x))=d_l(x) : C(\inl(x))}
\and
  \inferrule*[right=$+$-\comp\ 2.]{\Gamma,\ z \oftype A + B \types C(z)\ \type \\ \Gamma,\ x \oftype A \types d_l(x) : C(\inl (x)) \\ \Gamma,\ y \oftype B \types d_r(y) : C(\inr (y))}{\Gamma,\ y \oftype B \types \syncase_{d_l, d_r}(\inr(y))=d_r(y) : C(\inr(y))}
\end{mathparpagebreakable}

\subsection{Universes} \label{subsec:universe-rules}

A universe within the theory may be closed under some or all of the logical constructors of the theory; we include below the rules corresponding to closure under all of the constructors given above.

\begin{mathparpagebreakable}
  \inferrule{ }{\types \synU\ \type}
\and
  \inferrule{ }{x \oftype \synU \types \el (x) \ \type}
\\
  \inferrule{\Gamma \types a : \synU \\ \Gamma,\ x \oftype \el(a) \types b(x) : \synU}{\Gamma \types \synpi (a, x. b(x)) : \synU}
\and
  \inferrule{\Gamma \types a : \synU \\ \Gamma,\ x \oftype \el(a) \types b(x) : \synU}{\Gamma \types \el (\synpi (a, x. b(x)) = \synPi_{x : \el(a)} \el(b(x)) \ \type}
\\
  \inferrule{\Gamma \types a : \synU \\ \Gamma,\ x \oftype \el(a) \types b(x) : \synU}{ \Gamma \types \synsigma (a, x. b(x)) : \synU}
\and
  \inferrule{ \Gamma \types a : \synU \\ \Gamma,\ x \oftype \el(a) \types b(x) : \synU}{ \Gamma \types \el (\synsigma (a, x. b(x)) = \synSigma_{x : \el(a)} \el(b(x)) \ \type}
\\
  \inferrule{ \Gamma \types a : \synU \\ \Gamma \types b, c : \el (a)}{ \Gamma \types \synid_A (b,c) : \synU}
\and
  \inferrule{ \Gamma \types a : \synU \\  \Gamma \types b, c : \el (a)}{ \Gamma \types \el (\synid_a(b,c)) = \synId_{\el (a)}(b, c) \ \type}
\\
  \inferrule{ }{\types \synz : \synU}
\and
  \inferrule{ }{\types \el (\synz) = \synZero \ \type}
\and
  \inferrule{ }{\types \syno : \synU}
\and
  \inferrule{ }{\types \el (\syno) = \synOne \ \type}
\\
  \inferrule{\Gamma \types a, b : \synU}{ \Gamma \types a \smallplus b : \synU}
\and
  \inferrule{ \Gamma \types a, b : \synU}{ \Gamma \types \el (a \smallplus b) = \el (a) + \el (b) \ \type}
\\
  \inferrule{ \Gamma \types a : \synU \\  \Gamma ,\ x \oftype \el(a) \types b(x) : \synU}{ \Gamma \types \synw (a, x. b(x)) : \synU}
\and
  \inferrule{ \Gamma \types a : \synU \\  \Gamma ,\ x \oftype \el(a) \types b(x) : \synU}{ \Gamma \types \el (\synw (a, x. b(x)) = \synW_{x : \el(a)} \el(b(x)) \ \type}
\end{mathparpagebreakable}

\subsection{Further rules} \label{subsec:optional-rules}

The rules above are somewhat weak in their implications for equality of functions.  To this end, some further rules are often adopted: the \emph{$\eta$-rule} for $\synPi$-types, and the \emph{functional extensionality} rule(s).  Our formulation of the latter is taken from \cite{garner:on-the-strength}; see also \cite{hofmann:thesis}.

\begin{mathparpagebreakable}
  \inferrule*[right=$\synPi$-$\eta$]
    {\Gamma \types f : \synPi_{x \oftype A} B(x)}
    {\Gamma \types f = \lambda x \oftype A. \app(f,x) : \synPi_{x \oftype A} B(x) }
\and
  \inferrule*[right=$\synPi$-ext]
    {\Gamma \types f, g : \synPi_{x \oftype A} B(x) \\
     \Gamma \types h : \synPi_{x \oftype A} \synId_{B(x)}(\app(f,x),\app(g,x)) }
    {\Gamma \types \ext(f,g,h) : \synId_{\synPi_{x \oftype A} B(x)}(f,g) }
\and
  \inferrule*[right=$\synPi$-ext-comp-prop]
    {\Gamma, x \oftype A \types b : B(x) }
    {\Gamma \types \extcomp(x.b): \synId_{\synPi_{x \oftype A} B(x)} \hspace{4.4cm}  \\
     \hspace{2cm} (\ext(\lambda x \oftype A. b,\lambda x \oftype A. b,\lambda x \oftype A. \refl b), \refl (\lambda x \oftype A. b)) }
\end{mathparpagebreakable}

%%% Local Variables: 
%%% mode: latex
%%% TeX-master: "simplicial-model"
%%% End: 

% Note: standard ordering of rules we currently follow throughout is:
% Pi; Sigma; Id; W; 0; 1; +; universes.

\section{Logical structure on contextual categories} \label{app:cxl-structure}

We give here full translations of the various type-theoretic rules and axioms into the language of contextual categories: the logical rules of Section~\ref{subsec:logical-rules}, the universe rules of Section \ref{subsec:universe-rules}, the extensionality and $\eta$-rules of Section~\ref{subsec:optional-rules}, and Axiom~\ref{axiom:univalence}, the Univalence Axiom.

\subsection{Logical structure}

\begin{definition}
A \emph{$\synPi$-type structure} on a contextual category $\CC$ consists of:
\begin{enumerate}
  \item for each $(\Gamma, A, B) \in \ob_{n+2} \CC$, an object $(\Gamma, \synPi(A, B)) \in \ob_{n+1} \CC$;
  \item for each such $(\Gamma, A, B) $ and section $b \colon (\Gamma, A) \to (\Gamma, A, B)$, a section $\lambda(b) \colon \Gamma \to$ $(\Gamma, \synPi(A, B))$;
  \item for each such $(\Gamma, A, B) $ and pair of sections $k \colon \Gamma \to (\Gamma, \synPi (A, B))$, $a \colon \Gamma \to (\Gamma, A)$, a section $\app(k,a) \colon \Gamma \to (\Gamma, A, B)$ such that $p_B \cdot \app(k,a) = a$,
 \item such that for each $(\Gamma, A, B)$, $a \colon \Gamma \to (\Gamma, A)$ and $b \colon (\Gamma, A) \to (\Gamma, A, B)$, we have $\app(\lambda(b),a) = b \cdot a$;
 \item and for any $f \colon \Gamma' \to \Gamma$, and all appropriate arguments as above,
   \begin{gather*}
     f^*(\Gamma, \synPi(A, B)) = (\Gamma', \synPi(f^*A, f^*B)), \\
     f^*\lambda(b) = \lambda({f^*b}), \qquad f^*(\app(k,a)) = \app(f^*k,f^*a).
   \end{gather*}
 \end{enumerate}
\end{definition}

Given a $\synPi$-structure on $\CC$, and $(\Gamma,A,B)$ as above, write $\app_{A,B}$ for the morphism 
\begin{multline*}
  q(q(p_{\synPi(A,B)} \cdot p_{p_{\synPi(A,B)}^*A},A),B) \cdot {} \\
  \app_{(p_{\synPi(A,B)} \cdot p_{p_{\synPi(A,B)}^*A})^* A,\, (p_{\synPi(A,B)} \cdot p_{p_{\synPi(A,B)}^*A})^* B}\, \big( (1,p_{p_{\synPi(A,B)}^*A}) , (1, q(p_{\synPi(A,B)},A)) \big) \\
  \colon (\Gamma,\synPi(A,B),p_{\synPi(A,B)}^*A) \to (\Gamma,A,B);
\end{multline*}
the general form $\app_{A,B}(k,a)$ can be re-derived from these instances.  Also, for objects $(\Gamma,A)$, $(\Gamma,B)$ in $\CC$, write $(\Gamma,[A,B])$ for $(\Gamma, \synPi(A,p_A^*B))$.

\begin{definition}
A \emph{$\synSigma$-type structure} on a contextual category $\CC$ consists of:
\begin{enumerate}
  \item for each $(\Gamma, A, B) \in \ob_{n+2} \CC$, an object $(\Gamma, \synSigma(A, B)) \in \ob_{n+1} \CC$;
  \item for each such $(\Gamma,A,B)$, a morphism $\pair_{A,B} \colon (\Gamma,A,B) \to (\Gamma, \synSigma(A,B))$ over $\Gamma$;
  \item for each $(\Gamma,A,B)$, object $(\Gamma, \synSigma(A,B), C)$, and map $d \colon (\Gamma,A,B) \to (\Gamma, \synSigma(A,B), C)$ with $p_C \cdot d = \pair_{A,B}$, a section $\synsplit_d \colon (\Gamma,\synSigma(A,B)) \to (\Gamma, \synSigma(A,B), C)$, with $\synsplit_d \cdot \pair_{A,B} = d$;
\item such that for $f \colon \Gamma' \to \Gamma$, and all appropriate arguments as above,
   \begin{gather*}
     f^*(\Gamma, \synSigma(A, B)) = (\Gamma', \synSigma(f^*A, f^*B)), \\
     f^*\pair_{A,B} = \pair_{f^*A,f^*B}, \qquad f^*\synsplit_d = \synsplit_{f^*d}.
   \end{gather*}
 \end{enumerate}
\end{definition}

\begin{definition}
An \emph{$\synId$-type structure} on a contextual category $\CC$ consists of:
\begin{enumerate}
  \item for each $(\Gamma, A)$, an object $(\Gamma, A, p_A^*A, \synId_A)$;
  \item for each $(\Gamma,A)$, a morphism $\refl_A \colon (\Gamma,A) \to (\Gamma, A, p_A^*A, \synId_A)$, such that $p_{\synId_A} \cdot \refl_A = (1_A,1_A) \colon (\Gamma,A) \to (\Gamma, A, p_A^*A)$;
  \item for each $(\Gamma, A, p_A^*A, \synId_A, C)$ and $d \colon (\Gamma,A) \to (\Gamma, A, p_A^*A, \synId_A, C)$ with $p_C \cdot d = \refl_A$, a section $\synJ_{C,d} \colon (\Gamma, A, p_A^*A, \synId_A) \to (\Gamma, A, p_A^*A, \synId_A, C)$, such that $\synJ_{C,d} \cdot \refl_A = d$;
\item such that for $f \colon \Gamma' \to \Gamma$, and all appropriate arguments as above,
   \begin{gather*}
     f^*(\Gamma, A, p_A^*A, \synId_A) = (\Gamma', f^*A, (p_{f^*A})^*(f^*A), \synId_{f^*A}), \\
     f^*\refl_A = \refl_{f^*A}, \qquad f^*\synJ_{C,d} = \synJ_{f^*C,f^*d}.
   \end{gather*}
 \end{enumerate}
\end{definition}

\begin{definition}
Given a contextual category $\CC$ equipped with a $\synPi$-type structure, a \emph{$\synW$-type structure} on $\CC$ consists of:
\begin{enumerate}
   \item for each $(\Gamma, A, B)$, an object $(\Gamma, \synW(A,B))$;
  \item for each $(\Gamma,A,B)$, a map over $\Gamma$
\[ \syn{sup}_{A,B} \colon (\Gamma, A, \synPi(B,p_B^*p_A^*\synW(A,B))) \to (\Gamma, \synW(A,B)); \]
  \item for each $(\Gamma,\synW(A,B),C)$ and map 
\begin{multline*}
  d \colon \big( \Gamma, A, [B,p_A^*\synW(A,B)], \synPi \big( p_{[B,p_A^*\synW(A,B)]}^*B, \\
  \app(p_{p_{[B,p_A^*\synW(A,B)]}^*B},q(p_{[B,p_A^*\synW(A,B)]},B))^*C\big)\big) \to (\Gamma,\synW(A,B),C)
\end{multline*}
such that 
\[ p_C \cdot d = \syn{sup}_{A,B} \cdot p_{\synPi(p_{[B,p_A^*\synW(A,B)]}^*B,\app(p_{p_{[B,p_A^*\synW(A,B)]}^*B},q(p_{[B,p_A^*\synW(A,B)]},B))^*C)}, \]
a section $\wrec_{C,d} \colon (\Gamma,\synW(A,B)) \to (\Gamma,\synW(A,B),C)$, such that 
\begin{multline*}
  \wrec_{C,d} \cdot \syn{sup}_{A,B} = \\
  d \cdot \lambda(\wrec_{C,d} \cdot \app(p_{p_{[B,p_A^*\synW(A,B)]}^*B},q(p_{[B,p_A^*\synW(A,B)]},B)));
\end{multline*}
  \item and such that for $f \colon \Gamma' \to \Gamma$, and all appropriate arguments as above,
    \begin{gather*}
     f^*(\Gamma, \synW(A,B)) = (\Gamma', \synW(f^*A,f^*B)), \\
     f^*\syn{sup}_{A,B} = \syn{sup}_{f^*A,f^*B}, \qquad f^*\wrec_{C,d} = \wrec_{f^*C,f^*d}.
    \end{gather*}
 \end{enumerate}
\end{definition}

\begin{definition}
A \emph{zero-type structure} on a contextual category $\CC$ consists of:
\begin{enumerate}
  \item for each $\Gamma$, an object $(\Gamma,\synZero_\Gamma)$;
  \item for any object $(\Gamma,\synZero_\Gamma, C)$, a section $\syncase_C \colon (\Gamma,\synZero_\Gamma) \to (\Gamma,\synZero_\Gamma, C)$;
  \item such that for each $f \colon \Gamma' \to \Gamma$, $f^*(\Gamma, \synZero_\Gamma) = (\Gamma', \synZero_{\Gamma'})$; and for each such $f$ and $(\Gamma,\synZero_\Gamma, C)$ as above, $f^*(\syncase_C) = \syncase_{f^*C}$.
  \end{enumerate}
\end{definition}

\begin{definition}
A \emph{unit-type structure} on a contextual category $\CC$ consists of:
\begin{enumerate}
  \item for each $\Gamma$, an object $(\Gamma,\synOne_\Gamma)$;
  \item a section $\ast_{\Gamma} \colon (\Gamma) \to (\Gamma,\synOne_\Gamma)$;
  \item for each object $(\Gamma,\synOne_\Gamma, C)$, and map $d \colon \Gamma \to (\Gamma,\synOne_\Gamma, C)$ with $p_C \cdot d = \ast_{\Gamma}$, a section $\synrec_{C,d} \colon (\Gamma,\synOne_\Gamma) \to (\Gamma,\synOne_\Gamma, C)$, such that $\synrec_{C,d} \cdot \ast_{\Gamma} = d$;
  \item such that for $f \colon \Gamma' \to \Gamma$, and appropriate arguments as above,
   \begin{gather*}
     f^*(\Gamma, \synOne_\Gamma) = (\Gamma', \synOne_{\Gamma'}), \\
     f^*\ast_{\Gamma} = \ast_{\Gamma'}, \qquad f^*\synrec_{C,d} = \synrec_{f^*C,f^*d}.
   \end{gather*} 
 \end{enumerate}
\end{definition}
% \todo{Add the target type argument into recursors where it’s not already shown, since it’s \emph{not} typically deducible?}

\begin{definition}
A \emph{sum-type structure} on a contextual category $\CC$ consists of:
\begin{enumerate}
  \item for any objects $(\Gamma,A)$ and $(\Gamma,B)$, an object $(\Gamma,A+B)$;
  \item for each such $(\Gamma,A)$, $(\Gamma,B)$, maps $\inl_{A,B} \colon (\Gamma,A) \to (\Gamma,A+B)$ and $\inr_{A,B} \colon (\Gamma,B) \to$ $(\Gamma,A+B)$, over $\Gamma$;
  \item for each object $(\Gamma,A+B, C)$, and maps $d_l \colon (\Gamma,A) \to (\Gamma,A+B, C)$, $d_r \colon (\Gamma,B) \to$ $(\Gamma,A+B, C)$ with $p_C \cdot d_l = \inl_{A,B}$ and $p_C \cdot d_r = \inr_{A,B}$, a section $\syncase_{C,d_l,d_r} \colon (\Gamma,A+B) \to (\Gamma,A+B,C)$, such that $\syncase_{C,d_l,d_r} \cdot \inl_{A,B} = d_l$ and $\syncase_{C,d_l,d_r} \cdot \inr_{A,B} = d_r$;
  \item such that for $f \colon \Gamma' \to \Gamma$, and appropriate arguments as above,
   \begin{gather*}
     f^*(\Gamma,A+B) = (\Gamma', f^*A + f^*B), \qquad f^*(\inl_{A,B}) = \inl_{f^*A,f^*B}, \\
     f^*(\inr_{A,B}) = \inr_{f^*A,f^*B}, \qquad f^*(\syncase_{C,d_l,d_r}) = \syncase_{f^*C,f^*d_l,f^*d_r}.
   \end{gather*} 
 \end{enumerate}
\end{definition}

\subsection{Universes}

As in Section~\ref{subsec:universe-rules}, one may consider a universe closed under some or all of the logical structure of the ambient contextual category; that is, carrying operations reflecting the global operations on types.

\begin{definition}
A \emph{universe} in a contextual category $\CC$ is a distinguished object $(\synU,\el) \in \ob_2 \CC$.

Assuming a $\synPi$-type structure on $\CC$, we say $(\synU,\el)$ is \emph{closed under $\synPi$-types} if for all maps $a \colon \Gamma \to \synU$ and $b \colon (\Gamma, a^*\el) \to \synU$, we are given a map $\synpi(a,b) \colon \Gamma \to \synU$, such that $(\Gamma,\synpi(a,b)^*\el) = (\Gamma,\synPi(a^*\el,b^*\el))$, and moreover stably in $\Gamma$, i.e.\ such that for $f \colon \Gamma' \to \Gamma$ and $a$, $b$ as above, $f^*(\synpi(a,b)) = \synpi(f^*a,f^*b)$.

Similarly, given a $\synSigma$-type structure on $\CC$, say $(\synU,\el)$ is \emph{closed under $\synSigma$-types} if for each $a \colon \Gamma \to \synU$ and $b \colon (\Gamma, a^*\el) \to \synU$, we are given $\synsigma(a,b) \colon \Gamma \to \synU$, such that $(\Gamma,\synsigma(a,b)^*\el) = (\Gamma,\synSigma(a^*\el,b^*\el))$, and such that for $f \colon \Gamma' \to \Gamma$ and $a$, $b$ as above, $f^*(\synsigma(a,b)) = \synsigma(f^*a,f^*b)$.

Given an $\synId$-type structure on $\CC$, say $(\synU,\el)$ is \emph{closed under $\synId$-types} if for each $a \colon \Gamma \to \synU$, we are given a map $\synid_a \colon (\Gamma,a^*\el,p_{a^*\el}^*a^*\el) \to \synU$, such that $(\Gamma,a^*\el,p_{a^*\el}^*a^*\el,\synid_a^*\el) = (\Gamma,a^*\el,p_{a^*\el}^*a^*\el,\synId_{a^*\el})$, and such that for $f \colon \Gamma' \to \Gamma$ and $a$ as above, $f^*(\synid_a) = \synid_{f^*a}$.

Given a zero-type (resp.\ unit-type) structure on $\CC$, say $(\synU,\el)$ \emph{contains $\synZero$ (resp.\ $\synOne$)} if we are given a map $\synz \colon \pt \to \synU$ (resp.\ $\syno$) such that $\synz^*\el = \synZero$ ($\syno^*\el = \synOne$).

Given a sum-type structure on $\CC$, say $(\synU,\el)$ is \emph{closed under sum types} if for each pair of maps $a,b \colon \Gamma \to \synU$, we are given $a \smallplus b \colon \Gamma \to \synU$, such that $(a \smallplus b)^*\el = a^*\el + b^*\el$, and moreover such that $f \colon \Gamma' \to \Gamma$ and $a,b$ as above, $f^*(a \smallplus b) = f^*a \smallplus f^*b$.

Given $\synW$-type structure on $\CC$, say $(\synU,\el)$ is \emph{closed under $\synW$-types} if for each $a \colon \Gamma \to \synU$ and $b \colon (\Gamma, a^*\el) \to \synU$, we are given $\synw(a,b) \colon \Gamma \to \synU$, such that $(\Gamma,\synw(a,b)^*\el) = (\Gamma,\synW(a^*\el,b^*\el))$, and such that for $f \colon \Gamma' \to \Gamma$ and $a$, $b$ as above, $f^*(\synw(a,b)) = \synw(f^*a,f^*b)$.
\end{definition}

\subsection{Extensionality and Univalence} \label{subsec:optional-rules-alg}

For the following group of rules, let $\CC$ be a contextual category equipped with chosen $\synPi$- and $\synId$-type structures.

\begin{definition}
Say that $\CC$ satisfies the \emph{$\synPi$-$\eta$ rule} if for any $(\Gamma,A,B)$, the ``$\eta$-expansion'' map
\[ q(p_{\synPi(A,B)},\synPi(A,B)) \cdot \lambda (1_{p_{\synPi(A,B)}^*A},\app_{A,B})  \colon (\Gamma,\synPi(A,B)) \to (\Gamma,\synPi(A,B)) \]
is equal to $1_{(\Gamma,\synPi(A,B))}$.

A \emph{$\synPi$-\textsc{ext} structure} on $\CC$ is an operation giving for each $(\Gamma,A,B)$ a map 
\begin{multline*}
  \ext_{A,B} \colon (\Gamma,\synPi(A,B),p_{\synPi(A,B)}^*\synPi(A,B),\syn{Htp}_{A,B}) \to \\
  (\Gamma,\synPi(A,B),p_{\synPi(A,B)}^*\synPi(A,B),\synId_{A,B})
\end{multline*}
over $(\Gamma,\synPi(A,B),p_{\synPi(A,B)}^*\synPi(A,B))$, stably in $\Gamma$, where $\syn{Htp}_{A,B}$ is the object
\begin{multline*}
   \Big( \Gamma,\ \synPi(A,B),\ p_{\synPi(A,B)}^*\synPi(A,B),\ 
   \synPi\Big( \big(p_{\synPi(A,B)} \cdot p_{p_{\synPi(A,B)}^*\synPi(A,B)} \big)^*A, \\
   \big(\app_{A,B} \cdot q(p_{\synPi(A,B)},A), \app_{A,B} \cdot q(q(p_{\synPi(A,B)},\synPi(A,B)),A) \big)^* \synId_B \Big) \Big).
\end{multline*}

Given a $\synPi$-\textsc{ext} structure on $\CC$, a \emph{$\synPi$-\textsc{ext-comp-prop} structure} for it is an operation giving, for each $(\Gamma,A,B)$ and section $f \colon \Gamma \to (\Gamma,\synPi(A,B))$, a map 
\begin{multline*} \extcomp(f) \colon \Gamma \to \\
  (\Gamma,\,\synPi(A,B),\, p_{\synPi(A,B)}^*\synPi(A,B),\, \synId_{\synPi(A,B)},\, p_{\synId_{\synPi(A,B)}}^*\synId_{\synPi(A,B)},\, \synId_{\synId_{\synPi(A,B)}})
\end{multline*}
over the pair of maps
\begin{multline*} \ext_{A,B}(f,g) \cdot \lambda(1_A,\refl_B \cdot p_{p_A^*B}\cdot \app((1_A,f),(1_A,1_A)))\ \ ,\ \ \refl_{\synPi(A,B)} \cdot f \\
\colon \Gamma \to (\Gamma,\,\synPi(A,B),\, p_{\synPi(A,B)}^*\synPi(A,B),\, \synId_{\synPi(A,B)}),
\end{multline*}
stably as ever in $\Gamma$.

\end{definition}

Before defining the Univalence Axiom, we first (as in Section~\ref{subsec:type-theoretic-univalence}) set up several auxiliary definitions.  We assume, from here on, that $\CC$ is equipped also with $\synSigma$- and $\synId$-type structures.

\begin{definition}
For objects $(\Gamma,X)$, $(\Gamma,Y)$ of $\CC$, let 
\[ \syn{exch}_{X,Y} \colon (\Gamma,X,p_X^*Y) \to (\Gamma,Y,p_Y^*X) \]
be the evident ``exchange'' map, $(q(p_X,Y), p_{p_X^*Y})$.
\end{definition}

\begin{definition}
Let $(\Gamma,A)$, $(\Gamma,B)$ be objects of $\CC$.

Take $(\Gamma,[A,B],\synLHInv_{A,B})$ to be the object
\begin{multline*}
  \Big(\Gamma,\ [A,B],\ \synSigma\Big( p_{[A,B]}^*[B,A] ,\, \synPi \Big( (p_{[A,B]}  \cdot p_{p_{[A,B]}^*[B,A]})^*A, \\
  \big(q(p_B,A) \cdot \app_{B,A} \cdot q(p_{p_{[B,A]}^*A},p_{[B,A]}^*B) \cdot \app_{p_{[B,A]}^*A,\,p_{[B,A]}^*B}  \qquad \qquad \\
   \cdot q(\syn{exch}_{[A,B],[B,A]},A),\ q(p_{[A,B]} \cdot p_{p_{[A,B]}^*[B,A]},A)\big)^*\synId_A \Big) \Big) \Big).
\end{multline*}
% (G,A,p_A^*B) -> (G,B,p_B^*A) : (q(p_A,B),p_{p_A^*B}) 

% LHInv:
% obtained by Sigma from a context (G, [A,B], p_{[A,B]}^*[B,A], ?? )
% where ?? is obtained by Pi from context (G, [A,B], p_{[A,B]}^*[B,A], p*p*A, ?? )
% where ?? is obtained by pullback of Id_B along two maps (G, [A,B], p_{[A,B]}^*[B,A], p*p*A) -> (G, A):
% second is q(p_{p*[B,A]}.p_{[A,B]},A),
% first is composite:
% (G, [A,B], p_{[A,B]}^*[B,A], p*p*A) —q(exch_{[A,B],[A,B]},A)—>(G, [B,A], p_{[B,A]}^*[A,B], p*p*A)
% —ev_{p_{[B,A]}^*A,p_{[B,A]}^*B,}—> (G, [B,A], p*A, p_{p*A}*p*B)
% —q(p_{p*A})—> (G, [B,A], p*B)
% —ev_{B,A}—> (G, B, p_B^*A)
% —q(p_B,A)—> (G, A)

Similarly, take $(\Gamma,[A,B],\synRHInv_{A,B})$ to be the object
\begin{multline*}
  \Big(\Gamma,\ [A,B],\ \synSigma\Big( p_{[A,B]}^*[B,A] ,\, \synPi \Big( (p_{[A,B]}  \cdot p_{p_{[A,B]}^*[B,A]})^*B, \\
  \big(q(p_A,B) \cdot \app_{A,B} \cdot q(p_{p_{[A,B]}^*B},p_{[A,B]}^*A) \cdot \app_{p_{[A,B]}^*B,\,p_{[A,B]}^*A}, \qquad \qquad \\
   q(p_{[A,B]} \cdot p_{p_{[A,B]}^*[B,A]},B)\big)^*\synId_B \Big) \Big) \Big).
\end{multline*}

% RHInv:
% obtained by Sigma from a context (G, [A,B], p_{[A,B]}^*[B,A], ?? )
% where ?? is obtained by Pi from context (G, [A,B], p_{[A,B]}^*[B,A], p*p*B, ?? )
% where ?? is obtained by pullback of Id_B along two maps (G, [A,B], p_{[A,B]}^*[B,A], p*p*B) -> (G, B):
% second is q(p_{p*[B,A]}.p_{[A,B]},B),
% first is composite:
% (G, [A,B], p_{[A,B]}^*[B,A], p*p*B) —ev_{p_{[A,B]}^*B,p_{[A,B]}^*A,}—> (G, [A,B], p*B, p_{p*B}*p*A)
% —q(p_{p*B})—> (G, [A,B], p*A)
% —ev_{A,B}—> (G, A, p_A^*B)
% —q(p_A,B)—> (G, B)

Now, set 
\[ (\Gamma,[A,B],\synisHIso_{A,B}) := (\Gamma,[A,B],\synLHInv_{A,B} \times \synRHInv_{A,B}), \]
\[ (\Gamma,\synHIso(A,B)) := (\Gamma,\synSigma([A,B],\synisHIso_{A,B}). \]
\end{definition}

\begin{definition}
For any $(\Gamma,A)$, there is a canonical ``identity'' section 
  \[ \syn{id}_A := \lambda( (1_A,1_A) ) \colon \Gamma \to (\Gamma,[A,A]); \]
and this morevover lifts to a section
$ \syn{id}^{\syn{isHIso}}_A \colon \Gamma \to (\Gamma,[A,A],\synisHIso_{A,A}), $
given in full by
\begin{multline*}
  \pair_{\synLHInv_{A,A},\synRHInv_{A,A}} \cdot \Big( % \\ linebreak here, if one needed
  \pair_{p_{[A,A]}^*[A,A], \syn{H}_{\syn{L}} } \cdot q( (\syn{id}_A, \syn{id}_A), \syn{H}_{\syn{L}}) \cdot \lambda((1_A,\refl_A)),\\
  \pair_{p_{[A,A]}^*[A,A], \syn{H}_{\syn{R}} } \cdot q( (\syn{id}_A, \syn{id}_A), \syn{H}_{\syn{R}}) \cdot \lambda((1_A,\refl_A))  \Big),
\end{multline*}
where $\syn{H}_{\syn{L}}$ is the object
\begin{multline*}
  \Big(\Gamma,\ [A,A],\ p_{[A,A]}^*[A,A] ,\, \synPi \Big( (p_{[A,A]}  \cdot p_{p_{[A,A]}^*[A,A]})^*A, \\
  \big(q(p_A,A) \cdot \app_{A,A} \cdot q(p_{p_{[A,A]}^*A},p_{[A,A]}^*A) \cdot \app_{p_{[A,A]}^*A,\,p_{[A,A]}^*A}  \qquad \qquad \\
  \cdot q(\syn{exch}_{[A,A],[A,A]},A),\ q(p_{[A,A]} \cdot p_{p_{[A,A]}^*[A,A]},A)\big)^*\synId_A \Big) \Big)
\end{multline*}
and $\syn{H}_{\syn{R}}$ is the same but with $q(\syn{exch}_{[A,A],[A,A]},A)$ omitted. 

Lastly, set $\syn{id}^{\syn{HIso}} := \pair_{[A,A],\synisHIso_{A,A}} \cdot \syn{id}^{\syn{isHIso}}_A \colon \Gamma \to (\Gamma, \synHIso(A,A))$.
\end{definition}

\begin{definition}
For any object $(\Gamma,A,B)$, we can now define a map
\[ \syn{w}_{A,B} \colon (\Gamma,A,p_A^*A,\synId_A) \to (\Gamma,A,p_A^*A,\synHIso(p_{p_A^*A}^*B,q(p_A,A)^*B)) \]
by
\begin{multline*}
  \syn{w}_{A,B} := q(p_{\synId_A}, \synHIso(p_{p_A^*A}^*B,q(p_A,A)^*B)) \cdot {}\\
  \synJ\big( p_{\synId_A}^*\synHIso(p_{p_A^*A}^*B,q(p_A,A)^*B), \\
    q(\refl_A,p_{\synId_A}^*\synHIso(p_{p_A^*A}^*B,q(p_A,A)^*B)) \cdot \syn{id}^{\syn{HIso}}_B \big)
\end{multline*}

Given this, define the predicate ``$B$ is a univalent family over $A$'' as:
\begin{multline*}
 (\Gamma,\syn{isUvt}(A,B)) := \\
  \big(\Gamma,\synPi\big(A,\synPi\big(p_A^*A,\lambda((1_{\synId_A},w_{A,B}))^*\synisHIso_{\synId_A,\synHIso(p_{p_A^*A}^*B,q(p_A,A)^*B)}\big)\big)\big).
\end{multline*}
\end{definition}

\begin{definition}
Given a universe $(\synU,\el)$ in $\CC$, say $(\synU,\el)$ \emph{satisfies the Univalence Axiom} if $\CC$ is equipped with a map $\syn{uvt}_{\synU,\el} \colon 1 \to \syn{isUvt}(\synU,\el)$.
\end{definition}

%%% Local Variables: 
%%% mode: latex
%%% TeX-master: "simplicial-model"
%%% End: 

\bibliographystyle{amsalphaurlmod}
\bibliography{model-bibliography}

\end{document}